\documentclass [12pt]{amsart}
\usepackage[utf8]{inputenc}
\pdfoutput=1

\usepackage{amsmath,amssymb,amsthm,amsfonts}
\usepackage{mathtools}%

\usepackage[main=english,french]{babel}%
\usepackage{comment}%
\usepackage[bookmarksdepth=4]{hyperref}%
\usepackage{bbm}%
\usepackage{bm}%
\usepackage{mathrsfs}%

\usepackage{tikz,graphicx,color}
\usepackage{tikz-cd}%
\usepackage{tikz-3dplot}%
\usetikzlibrary{calc}%
\usetikzlibrary{arrows}%
\usetikzlibrary{shapes}%
\usetikzlibrary{patterns}%
\usetikzlibrary{positioning}%
\usetikzlibrary{arrows.meta}
\usetikzlibrary{decorations.markings}
\usepackage{epstopdf}%
\pdfpageattr{/Group <</S /Transparency /I true /CS /DeviceRGB>>}

\usepackage[arrow]{xy}%
\usepackage{diagbox}%
\usepackage{subfig}%
\usepackage{arcs}%
\usepackage{xcolor}%
\usepackage{tabu}
\usepackage{booktabs}%

\usepackage[margin=1in]{geometry}

\usepackage{soul}
\usepackage{accents}
\usepackage{enumitem}
\usepackage{letltxmacro}
\usepackage{thmtools,etoolbox}

\def\myarabic#1{\normalfont(\roman{#1})}
\newlist{theoremlist}{enumerate}{1}
\setlist[theoremlist]{label=\myarabic{theoremlisti},ref={\myarabic{theoremlisti}},itemindent=0pt,labelindent=0pt,
leftmargin=*,noitemsep}

\makeatletter
\renewcommand{\p@theoremlisti}{\perh@ps{\thetheorem}}
\protected\def\perh@ps#1#2{\textup{#1#2}}
\newcommand{\itemrefperh@ps}[2]{\textup{#2}}
\newcommand{\itemref}[1]{\begingroup\let\perh@ps\itemrefperh@ps\ref{#1}\endgroup}
\makeatother

\usepackage{nameref,hyperref}
\usepackage[capitalize]{cleveref}

\usepackage{tikz-cd}
\usepackage{stmaryrd}

\usepackage{xspace}

\usepackage{stmaryrd}

\usepackage{adjustbox}

\newtheorem{theorem}{Theorem}[section]

\newtheorem{lemma}[theorem]{Lemma}
\newtheorem{proposition}[theorem]{Proposition}
\newtheorem{corollary}[theorem]{Corollary}
\newtheorem{conjecture}[theorem]{Conjecture}
\newtheorem{question}[theorem]{Question}
\theoremstyle{definition}
\newtheorem{remark}[theorem]{Remark}
\newtheorem{problem}[theorem]{Problem}
\theoremstyle{definition}
\newtheorem{definition}[theorem]{Definition}
\newtheorem{example}[theorem]{Example}
\newtheorem{notation}[theorem]{Notation}

\def\<{{\langle}}
\def\>{{\rangle}}

\def\la{{\lambda}}

\def\det{{ \operatorname{det}}}

\def\sh{{ \operatorname{sh}}}
\def\Conv{ \operatorname{Conv}}

\def\wt{\operatorname{wt}}

\def\Racc{R^\circ}
\def\Rich_#1^#2{\Racc_{#1,#2}}
\def\Richcl_#1^#2{R_{#1,#2}}

\def\bv{{\mathbf{v}}}

\def\bx{{\mathbf{x}}}
\def\by{{\mathbf{y}}}

\def\xrasim{\xrightarrow{\sim}}

\def\H{{\mathbb{H}}}

\def\Q{{\mathbb{Q}}}

\def\Pio{\Pi^\circ}
\def\HOMFLY{\operatorname{HOMFLY}}

\def\DC{{\Dscr_3}}

\numberwithin{equation}{section}

\def\FLY{HOMFLY\xspace}

\def\Richafftp_#1^#2{{\Rcal_{#1,#2}^{>0}}}

\def\Richaff_#1^#2{\accentset{\circ}{\mathcal{R}}_{#1,#2}}

\def\area{\operatorname{area}}
\def\dinv{\operatorname{dinv}}
\def\Dyck{\operatorname{Dyck}}
\def\Dyckxx(#1,#2){\Dyck_{#1,#2}}

\def\Pit_#1{\Xcal^\circ_{#1}}
\def\KR{{\operatorname{KR}}}

\def\top{{\operatorname{top}}}

\def\PKR{\Pcal_\KR}

\def\SL{\operatorname{SL}}

\def\Povar_#1{\accentset{\circ}{\Pi}_{#1}}
\def\Povarcl_#1{\Pi_{#1}}
\def\RPovar_#1{\accentset{\circ}{\Pi}^\R_{#1}}
\def\RPovarcl_#1{\Pi^\R_{#1}}
\def\Povtp_#1{\Pi_{#1}^{>0}}
\def\Povtnn_#1{\Pi_{#1}^{\geq0}}

\def\Gr{\operatorname{Gr}}

\def\KLR_#1^#2{R_{#1,#2}(q)}

\crefname{figure}{Figure}{Figures}
\crefname{notation}{Notation}{Notations}

\addtotheorempostheadhook[theorem]{\crefalias{theoremlisti}{theorem}}
\addtotheorempostheadhook[lemma]{\crefalias{theoremlisti}{lemma}}
\addtotheorempostheadhook[proposition]{\crefalias{theoremlisti}{proposition}}
\addtotheorempostheadhook[corollary]{\crefalias{theoremlisti}{corollary}}

\def\Ecal{\mathcal{E}}\def\Pcal{\mathcal{P}}\def\Rcal{\mathcal{R}}\def\Xcal{\mathcal{X}}

\def\C{\mathbb{C}}
\def\R{\mathbb{R}}
\def\N{\mathbb{N}}
\def\Z{\mathbb{Z}}
\def\Q{\mathbb{Q}}

\def\<{{\langle}}
\def\>{{\rangle}}

\def\la{{\lambda}}

\def\det{{ \operatorname{det}}}

\def\sh{{ \operatorname{sh}}}
\def\Conv{ \operatorname{Conv}}

\def\wt{\operatorname{wt}}

\def\SL{\operatorname{SL}}

\def\Gr{\operatorname{Gr}}

\def\Z{{\mathbb Z}}
\def\R{{\mathbb R}}
\def\Gr{{\rm Gr}}

\def\id{{\operatorname{id}}}

\newcounter{todobackgr}[section]

\newcounter{todofigure}[section]

\makeatletter
\DeclareRobustCommand{\cev}[1]{%
  \mathpalette\do@cev{#1}%
}
\newcommand{\do@cev}[2]{%
  \fix@cev{#1}{+}%
  \reflectbox{$\m@th#1\vec{\reflectbox{$\fix@cev{#1}{-}\m@th#1#2\fix@cev{#1}{+}$}}$}%
  \fix@cev{#1}{-}%
}
\newcommand{\fix@cev}[2]{%
  \ifx#1\displaystyle
    \mkern#23mu
  \else
    \ifx#1\textstyle
      \mkern#23mu
    \else
      \ifx#1\scriptstyle
        \mkern#22mu
      \else
        \mkern#22mu
      \fi
    \fi
  \fi
}

\makeatother

\def\Rich_#1^#2{\vec R^\circ_{#1,#2}}
\def\LRich_#1^#2{\cev R^\circ_{#1,#2}}
\def\RichL_#1^#2{\LRich_{#1}^{#2}}
\def\Rtp_#1^#2{\vec R_{#1,#2}^{>0}}
\def\LRtp_#1^#2{\cev R_{#1,#2}^{>0}}

\def\Rtnn_#1^#2{\vec R_{#1,#2}^{\geq0}}
\def\LRtnn_#1^#2{\cev R_{#1,#2}^{\geq0}}

\def\PR_#1^#2{\accentset{\circ}{\Pi}_{#1,#2}}%
\def\PRtp_#1^#2{\Pi_{#1,#2}^{>0}}%
\def\PRtnn_#1^#2{\Pi_{#1,#2}^{\geq0}}%
\def\PRcl_#1^#2{\Pi_{#1,#2}}%
\def\PRR_#1^#2{\accentset{\circ}{\Pi}_{#1,#2}^\R}%
\def\PRRcl_#1^#2{\Pi_{#1,#2}^\R}%

\def\bv{{\mathbf{v}}}

\def\bx{{\mathbf{x}}}
\def\by{{\mathbf{y}}}

\def\hjmap{\kappa}
\def\hjmp_#1{\hjmap_{#1}}

\def\Uom_#1{U^{\diamond,-}_{#1}}

\def\xrasim{\xrightarrow{\sim}}

\def\Richaff_#1^#2{\accentset{\circ}{\mathcal{R}}_{#1}^{#2}}

\def\Povar_#1{\accentset{\circ}{\Pi}_{#1}}
\def\Povarcl_#1{\Pi_{#1}}
\def\RPovar_#1{\accentset{\circ}{\Pi}^\R_{#1}}
\def\RPovarcl_#1{\Pi^\R_{#1}}
\def\Povtp_#1^#2{\Pi_{#1,#2}^{>0}}
\def\Povtnn_#1{\Pi_{#1}^{\geq0}}

\def\Star_#1{\operatorname{Star}_{#1}}
\def\Startnn_#1{\operatorname{Star}^{\geq0}_{#1}}

\def\Link{\operatorname{Lk}}
\def\Lkx_#1{\Link_{#1}}
\def\Lkxx_#1^#2{\accentset{\circ}{\Link}_{#1}^{#2}}

\def\Lktxx_#1^#2{\Link^{>0}_{#1,#2}}
\def\Starxx_#1^#2{\operatorname{Star}_{#1,#2}}
\def\Startxx_#1^#2{\operatorname{Star}^{\geq0}_{#1,#2}}

\def\sctnn_#1{\sc^{\geq0}_{#1}}

\def\sctp_#1^#2{\sc^{>0}_{#1,#2}}

\def\eps{\varepsilon}
\def\Seps_#1{S_{#1}}

\def\Lktpe_#1^#2{\Link^{>0}_{#1,#2}}
\def\Lktnne_#1{\Link^{\geq0}_{#1}}

\def\Lktp_#1^#2{\Link^{>0}_{#1,#2}}
\def\Lktnn_#1{\Link^{\geq0}_{#1}}

\def\sc{Z}

\def\sco_#1^#2{\accentset{\circ}{\sc}_{#1,#2}}

\def\sccl_#1^#2{\sc_{#1}^{#2}}

\def\Y{\mathcal{Y}}
\def\Yo_#1{\accentset{\circ}{\Y}_{#1}}
\def\Ycl_#1{\Y_{#1}}

\def\Ytp_#1{\Y_{#1}^{>0}}

\def\strg(#1){\normg{#1}}

\def\normg#1{\|#1\|}

\def\int{{\operatorname{init}}}

\def\DOM^#1_#2{\Delta^{#1 \omega_i}_{#2 \omega_i}}
\def\DOMr^#1_#2{\Delta^{#1 \omega_r}_{#2 \omega_r}}
\def\DOMir^#1_#2{\Delta^{#1 \omega_{i_r}}_{#2 \omega_{i_r}}}

\def\lw{\line{w}}

\def\sv{s^{\mathbf{v}}}

\newcommand*{\smallcap}{{\mathbin{\scalebox{0.5}{\ensuremath{\cap}}}}}%

\def\RLtp_#1^#2{\cev R_{#1,#2}^{>0}}
\def\Rsf_#1^#2{\Rich_{#1}^{#2}(K)}
\def\LRsf_#1^#2{\LRich_{#1}^{#2}(K)}

\def\pre{{\,\operatorname{pre}}}
\def\tpre_#1^#2{\vec\twistop^\pre_{#1,#2}}
\def\Ltpre_#1^#2{\cev\twistop^\pre_{#1,#2}}
\def\tpreL_#1^#2{\Ltpre_{#1}^{#2}}
\def\twistop{\tau}
\def\twist_#1^#2{\vec\twistop_{#1,#2}}
\def\twistL_#1^#2{\cev\twistop_{#1,#2}}

\def\Sn{\mathfrak{S}_\N}

\def\sv_#1{s^{\bv}_{#1}}

\def\capBWB(#1){(#1)^{\smallcap w}}
\def\capBWoB(#1){(#1)^{\smallcap w_0}}
\def\capBWBnopar(#1){#1^{\smallcap w}}

\def\Pio_#1^#2{\Pi^\circ_{#1,#2}}

\def\chmnrR_#1{\vec\Delta_{#1}}
\def\chmnrL_#1{\cev\Delta_{#1}}

\def\Yo{y_0}

\def\H{\mathbf{H}_{q,t}}

\def\Dpm_#1{\Delta^{\pm}_{#1}}
\def\Dmp_#1{\Delta^{\mp}_{#1}}

\def\MS{\operatorname{MS}}
\def\TMSR_#1^#2{\vec \tau^{\,\MS}_{#1,#2}}
\def\TMSL_#1^#2{\cev \tau^{\,\MS}_{#1,#2}}

\def\Vlax_#1{V(\la)_{#1}}

\makeatletter
\newcommand*\bigcdot{\mathpalette\bigcdot@{.4}}
\newcommand*\bigcdot@[2]{\mathbin{\vcenter{\hbox{\scalebox{#2}{$\m@th#1\bullet$}}}}}
\makeatother

\def\Ztnn{\Z^{\geq}}
\def\Ztp{\Z^{>}}

\makeatletter
\newcommand{\xMapsto}[2][]{\ext@arrow 0599{\Mapstofill@}{#1}{#2}}
\def\Mapstofill@{\arrowfill@{\Mapstochar\Relbar}\Relbar\Rightarrow}
\makeatother

\makeatletter
\@namedef{subjclassname@2020}{%
  \textup{2020} Mathematics Subject Classification}
\makeatother

\def\DAHA_#1{\ddot{\mathbf{H}}_{#1;q,t}}

\def\Heckeq_#1{\mathbf{H}_{#1;q^{-1}}}
\def\Hecketi_#1{\mathbf{H}_{#1;t}}
\def\Heckeqcomm_#1{\mathbf{H}'_{#1;q^{-1}}}
\def\DASHA_#1{\mathbf{S}\ddot{\mathbf{H}}_{#1;q,t}}

\def\tt{t^{\frac12}}
\def\tti{t^{-\frac12}}

\def\BSk{\operatorname{BSk}}
\def\XSym{\mathbf{X}}
\def\tXSym{\tilde{\mathbf{X}}}

\def\Ht{\tilde H}

\def\la{\lambda}

\def\bx{{\mathbf{x}}}

\def\EHA{\mathcal{E}_{q,t}}
\def\EHAop{\mathcal{E}}

\def\La{\Lambda}
\def\Laq{\La_q}

\def\Laqt{\La_{q,t}}
\def\LaqtN{\La_{\N;q,t}}

\def\SH_#1{\DASHA_#1}
\def\DASHAp_#1{\DASHA_{#1}^+}
\def\SHp_#1{\DASHA_{#1}^+}
\def\DASHApinf{\DASHA_\infty^+}
\def\DASHAinf{\DASHA_\infty}

\def\SHpf{\DASHA_\infty^+}

\def\Curve{C}
\def\DC{D_{\Curve}}

\def\DCn{\DC^{(\N)}}
\def\DX_#1{D_{#1}}
\def\DXn_#1{D^{(\N)}_{#1}}
\def\DAS{D} %
\def\DASXn_#1{\DAS^{(\N)}_{#1}}
\def\DASXnf_#1{\DAS_{#1}}

\def\Ann{{\mathbb{A}}}
\def\Tor{{\mathbb{T}}}
\def\Disk{{\mathbb{D}}}
\def\PTor{{\Tor-\Disk}} %
\def\Surf{{\mathbb{S}}}
\def\Skop{\operatorname{Sk}}

\def\Sk{\Skop_{q^{-1}}}
\def\Skti{\Skop_{t}}
\def\Skpti{\Skop^+_{t}}
\def\SkptiMS{\Skop^{+;\operatorname{MS}}_{t}}
\def\SkpPTor{\Skop^+_{q,t}(\PTor)}
\def\SkPTor{\Skop_{q,t}(\PTor)}

\def\Skp{\Sk^+}
\def\Skpx_#1{\Skop^+_{#1}}
\def\BSkop{\operatorname{BSk}}
\def\BSkn{\BSkop_{\N;q,t}}
\def\Skn{\Skop_{\N;q,t}}
\def\EHAp{\EHA^+}
\def\EHApqi{\EHAop^+_{q,q^{-1}}}
\def\HOMFLY{\operatorname{HOMFLY}}

\def\PHOML{\Pcal^{\HOMFLY}_{\Link}}
\def\PHOMLC{\Pcal^{\HOMFLY}_{\LC}}

\def\PEHAC{\Pcal^{\Ecal}_{\Curve}}
\def\PEHATC{\widehat{\Pcal}^{\Ecal}_{\Curve}}
\def\PEHAX_#1{\Pcal^{\Ecal}_{#1}}
\def\PEHATX_#1{\widehat{\Pcal}^{\Ecal}_{#1}}
\def\KR{\operatorname{KR}}
\def\PKR{\Pcal^{\KR}}

\def\PKRLC{\Pcal^{\KR}_{\LC}}
\def\PKRX_#1{\Pcal^{\KR}_{#1}}
\def\Negut{Neguţ\xspace}
\def\Cseg_#1{\Curve_{#1}} %
\def\Csegp_#1{\Curve'_{#1}} %
\def\point{{\bm{p}}}
\def\ehact#1{\cdot #1}
\def\dahact#1{\cdot #1}
\def\Link{L}
\def\LC{\Link_\Curve}
\def\LX_#1{\Link_{#1}}

\def\qq{q^{1/2}}
\def\tt{t^{1/2}}
\def\qqi{q^{-1/2}}
\def\tti{t^{-1/2}}
\def\Cq{\C(\qq)}
\def\Cqt{\C(\qq,\tt)}

\def\Caq{\C(a,\qq)}
\def\Caqt{\C(a,\qq,\tt)}
\def\Cvq{\C(v,\qq)}
\def\Cv{\C(v)}

\def\FC{F_{\Curve}}
\def\FX_#1{F_{#1}}
\def\introsubsec#1{\subsection{#1}}
\def\nseg{\operatorname{k}}
\def\PathC{\Pcal_{\Curve}}
\def\PathX_#1{\Pcal_{#1}}
\def\rectX{m}
\def\rectY{n}
\def\rectXY{\rectX,\rectY}
\def\rectXp{m_0}
\def\rectYp{n_0}
\def\rectXYp{\rectXp,\rectYp}

\def\gam{\gamma_{\N;t}}
\def\writheop{\operatorname{w}}
\def\wop{\writheop}
\def\wC{\writheop(\Curve)}
\def\bop{\operatorname{b}}
\def\bC{\bop(\Curve)}
\def\wX(#1){\writheop(#1)}
\def\Curveproj{\overline{\Curve}}
\def\Csegproj_#1{\overline{\Curve}_{#1}}
\def\taup{\tau_+}
\def\taum{\tau_-}

\def\pvec{\bm{r}} %
\def\PV{\pvec}
\def\PVp{\pvec'}
\def\eps{\epsilon}
\def\TPV{\Tor^\ast_{\PV}} %
\def\RPV{\R^{2,\ast}_{\PV}} %
\def\XPV{X_{\PV}}
\def\YPV{Y_{\PV}}
\def\TPVp{\Tor^\ast_{\PVp}} %
\def\RPVp{\R^{2,\ast}_{\PVp}} %

\def\psiPV{\psi_{\PV}^{(\N)}}
\def\psiPVp{\psi_{\PVp}^{(\N)}}
\def\BraidGroup{\mathcal{B}}
\def\Cproj{\overline{\Curve}}
\def\Cprojp{\overline{\Curve'}}

\def\PAL{piecewise almost linear\xspace}
\def\Sop{\bm{e}}
\def\S{\Sop_{\N}}
\def\SX_#1{\S_{[#1,\N]}}
\def\SX_#1{\Sop_{[#1,\N]}}
\def\SXX_#1^#2{\Sop_{[#1,#2]}}
\def\SNX_#1{\mathfrak{S}_{[#1,\N]}}
\def\SNXX_#1^#2{\mathfrak{S}_{[#1,#2]}}
\def\ehagen{u}
\def\ehagenb{D}
\def\bx{{\bm{x}}}
\def\by{{\bm{y}}}
\def\bz{{\bm{z}}}
\def\dop{\operatorname{d}}
\def\sign{\operatorname{sign}}
\def\BZ{\mathbf{Z}}       %
\def\BZx{\mathbf{Z}^{>}} %
\def\BZxy{\mathbf{Z}^{+}} %
\def\BZp{\mathbf{Z}^{+}} %
\def\BZast{\mathbf{Z}^{\ast}} %
\def\BZgeq{\mathbf{Z}^{\geq}} %

\def\Pgen{P^{(\N)}} %
\def\Pgenm{P^{(\N-1)}} %
\def\Pgenf{P} %
\def\expcoef{c}

\def\unkn{%
\begin{tikzpicture}[baseline=(ZUZU.base),scale=0.2]\coordinate(ZUZU) at (0,-0.6);
\draw[line width=0.7pt] (0,0) circle (1cm);
  \end{tikzpicture}
}
\def\Tr{\operatorname{Tr}}
\def\Tro{\omega\Tr}

\def\ucb{D_{\Curve}}
\def\AS{{\operatorname{AL}}}
\def\ux_#1{\ehagen_{#1}}
\def\ubx_#1{D_{#1}}
\def\uasx_#1{\ehagen^\AS_{#1}}

\def\W^#1_#2{#2^{#1}} %
\def\WC^#1{\Curve^{#1}} %

\def\TC{T^\Ann_{\Curve}}
\def\TX_#1{T^\Ann_{#1}}

\def\X{X_1}
\def\Y{Y_1}
\def\Poincare{Poincar\'e\xspace}
\def\Interval{[0,1]}
\def\I{\Interval}

\def\lw{2pt}

\def\LPlus{
\scalebox{0.5}{
\begin{tikzpicture}[xscale=0.4,yscale=0.6,baseline=(Z.base)]
\coordinate(Z) at (0,-0.4);
\draw[line width=\lw,->,>={latex}] (1,-1)--(-1,1);
\draw[line width=4*\lw,white] (-1,-1)--(1,1);
\draw[line width=\lw,->,>={latex}] (-1,-1)--(1,1);
\end{tikzpicture}
}
}

\def\LMinus{
\scalebox{0.5}{
\begin{tikzpicture}[xscale=0.4,yscale=0.6,baseline=(Z.base)]
\coordinate(Z) at (0,-0.4);
\draw[line width=\lw,->,>={latex}] (-1,-1)--(1,1);
\draw[line width=4*\lw,white] (1,-1)--(-1,1);
\draw[line width=\lw,->,>={latex}] (1,-1)--(-1,1);
\end{tikzpicture}
}
}

\def\rc{10}
\def\LZero{
\scalebox{0.5}{
\begin{tikzpicture}[xscale=0.4,yscale=0.6,baseline=(Z.base)]
\coordinate(Z) at (0,-0.4);
\draw[line width=\lw,->,>={latex},rounded corners=\rc] (-1,-1)--(0,0)--(-1,1);
\draw[line width=\lw,->,>={latex},rounded corners=\rc] (1,-1)--(0,0)--(1,1);
\end{tikzpicture}
}
}

\def\lw{2pt}

\def\posfrB{0.47}
\def\posfrA{0.75}
\def\lwfrout{3.5*\lw}
\def\lwfrin{3*\lw}
\def\epsA{0.02}
\def\epsB{0.02}
\def\frscl{0.8}

\def\LPlusfr{
\scalebox{\frscl}{
\begin{tikzpicture}[xscale=0.4,yscale=0.6,baseline=(Z.base)]
\coordinate(Z) at (0,-0.4);
\draw[line width=\lwfrout] (\posfrA,-\posfrA)--(-\posfrB,\posfrB);
\draw[line width=\lwfrin,white] (\posfrA-\epsA,-\posfrA+\epsA)--(-\posfrB+\epsB,\posfrB-\epsB);
\draw[line width=\lw,->,>={latex}] (1,-1)--(-1,1);
\draw[line width=4*\lw,white] (-1,-1)--(1,1);
\draw[line width=\lwfrout] (-\posfrA,-\posfrA)--(\posfrB,\posfrB);
\draw[line width=\lwfrin,white] (-\posfrA+\epsA,-\posfrA+\epsA)--(\posfrB-\epsB,\posfrB-\epsB);
\draw[line width=\lw,->,>={latex}] (-1,-1)--(1,1);
\end{tikzpicture}
}
}

\def\LMinusfr{
\scalebox{\frscl}{
\begin{tikzpicture}[xscale=0.4,yscale=0.6,baseline=(Z.base)]
\coordinate(Z) at (0,-0.4);
\draw[line width=\lwfrout] (-\posfrA,-\posfrA)--(\posfrB,\posfrB);
\draw[line width=\lwfrin,white] (-\posfrA+\epsA,-\posfrA+\epsA)--(\posfrB-\epsB,\posfrB-\epsB);
\draw[line width=\lw,->,>={latex}] (-1,-1)--(1,1);
\draw[line width=4*\lw,white] (1,-1)--(-1,1);
\draw[line width=\lwfrout] (\posfrA,-\posfrA)--(-\posfrB,\posfrB);
\draw[line width=\lwfrin,white] (\posfrA-\epsA,-\posfrA+\epsA)--(-\posfrB+\epsB,\posfrB-\epsB);
\draw[line width=\lw,->,>={latex}] (1,-1)--(-1,1);
\end{tikzpicture}
}
}

\def\rc{9}

\def\xsh{0.1cm}
\def\epsA{0.02}
\def\epsB{0.02}
\def\posfrBx{0.47}
\def\posfrAx{0.75}
\def\LZerofr{
\scalebox{\frscl}{
\begin{tikzpicture}[xscale=0.4,yscale=0.6,baseline=(Z.base)]
  \coordinate(Z) at (0,-0.4);
  \begin{scope}[xshift=-\xsh]
    \draw[line width=\lwfrout,rounded corners=\rc] (-\posfrAx,-\posfrAx)--(0,0)--(-\posfrBx,\posfrBx);
    \draw[line width=\lwfrin,white,rounded corners=\rc] (-\posfrAx+\epsA,-\posfrAx+\epsA)--(0,0)--(-\posfrBx+\epsB,\posfrBx-\epsB);
    \draw[line width=\lw,->,>={latex},rounded corners=\rc] (-1,-1)--(0,0)--(-1,1);
\end{scope}
  \begin{scope}[xshift=\xsh]
    \draw[line width=\lwfrout,rounded corners=\rc] (\posfrAx,-\posfrAx)--(0,0)--(\posfrBx,\posfrBx);
    \draw[line width=\lwfrin,white,rounded corners=\rc] (\posfrAx-\epsA,-\posfrAx+\epsA)--(0,0)--(\posfrBx-\epsB,\posfrBx-\epsB);
    \draw[line width=\lw,->,>={latex},rounded corners=\rc] (1,-1)--(0,0)--(1,1);
  \end{scope}
\end{tikzpicture}
}
}

\def\vertclw{0.2}
\def\looseclw{1}
\def\posclwA{0.2}
\def\posclwB{0.6}
\def\epsclw{0.03}
\def\xclw{1.2}
\def\yclw{1}
\def\yyclw{2.2}
\def\clwscl{\frscl}

\def\Lcclwfr{
\scalebox{\clwscl}{
\begin{tikzpicture}[xscale=0.4,yscale=0.6,baseline=(Z.base)]
  \coordinate(Z) at (0,0.6);
  \draw[line width=\lwfrout,looseness=\looseclw] (\xclw,\yclw) to[out=-90,in=-90] (0,\yclw+\vertclw) --(0,\yyclw-\posclwB);
  \draw[line width=\lwfrin,looseness=\looseclw,white] (\xclw,\yclw) to[out=-90,in=-90] (0,\yclw+\vertclw) --(0,\yyclw-\posclwB-\epsclw);
  \draw[line width=\lw,->,>={latex},looseness=\looseclw] (\xclw,\yclw) to[out=-90,in=-90] (0,\yclw+\vertclw) --(0,\yyclw);  
  \draw[line width=4*\lw,looseness=\looseclw,white] (0,0) -- (0,\yclw-\vertclw) to[out=90,in=90] (\xclw,\yclw);
  \draw[line width=\lwfrout,looseness=\looseclw] (0,\posclwA) -- (0,\yclw-\vertclw) to[out=90,in=90] (\xclw,\yclw);
  \draw[line width=\lwfrin,white,looseness=\looseclw] (0,\posclwA+\epsclw) -- (0,\yclw-\vertclw) to[out=90,in=90] (\xclw,\yclw);
  \draw[line width=\lw,looseness=\looseclw] (0,0) -- (0,\yclw-\vertclw) to[out=90,in=90] (\xclw,\yclw);
\end{tikzpicture}
}
}

\def\Lstraightfr{
\scalebox{\clwscl}{
\begin{tikzpicture}[xscale=0.4,yscale=0.6,baseline=(Z.base)]
  \coordinate(Z) at (0,0.6);
  \draw[line width=\lwfrout,looseness=\looseclw] (0,\posclwA) -- (0,\yyclw-\posclwB);
  \draw[line width=\lwfrin,white,looseness=\looseclw] (0,\posclwA+\epsclw) -- (0,\yyclw-\posclwB-\epsclw);
  \draw[line width=\lw,->,>={latex},looseness=\looseclw] (0,0) -- (0,\yyclw);
\end{tikzpicture}
}
}

\def\BaseStringA{
\scalebox{0.5}{
\begin{tikzpicture}[xscale=0.4,yscale=0.6,baseline=(Z.base)]
\coordinate(Z) at (0,-0.4);
\draw[line width=\lw,->,>={latex}] (1,-1)--(-1,1);
\draw[line width=4*\lw,white] (-1,-1)--(1,1);
\draw[line width=\lw,->,>={latex}] (-1,-1)--(1,1);
\node[scale=1.3,anchor=west](A) at (1,-1) {$\ast$};
\end{tikzpicture}
}
}
\def\BaseStringB{
\scalebox{0.5}{
\begin{tikzpicture}[xscale=0.4,yscale=0.6,baseline=(Z.base)]
\coordinate(Z) at (0,-0.4);
\draw[line width=\lw,->,>={latex}] (-1,-1)--(1,1);
\draw[line width=4*\lw,white] (1,-1)--(-1,1);
\draw[line width=\lw,->,>={latex}] (1,-1)--(-1,1);
\node[scale=1.3,anchor=west](A) at (1,-1) {$\ast$};
\end{tikzpicture}
}
}

\def\DClimEHA{D_C}

\def\DXlimEHA_#1{D_{#1}}
\def\DTXlimEHA_#1{\tilde D_{#1}}
\def\DASXlimEHA_#1{D_{#1}}

\def\ucbp{D_{\Curve'}}
\def\ucbas_#1{\ehagenb_{#1}}
\def\Xbraid{X}
\def\Ybraid{Y}
\def\Tbraid{T}
\def\brpt{\point}
\def\Diske{\Disk}
\def\BraidGroupP{\BraidGroup^+}

\def\MS{{\operatorname{MS}}}
\def\PMS_#1{W^{\Ann}_{#1;q}}
\def\PMSti_#1{W^{\Ann}_{#1;t^{-1}}}
\def\WMSTor_#1{W^{\Tor}_{#1;q}}
\def\WMSPTor_#1{W^{\PTor}_{#1;q}}
\def\WMSTorti_#1{W^{\Tor}_{#1;t^{-1}}}
\def\WMSPTorti_#1{W^{\PTor}_{#1;t^{-1}}}

\def\qi{q^{-1}}

\def\SQA{\text{\hypertarget{SQA}{\textcolor{blue}{\fbox{A}}}}}
\def\SQB#1{\text{\hypertarget{SQB#1}{\textcolor{blue}{\fbox{B#1}}}}}
\def\SQC#1{\text{\textcolor{blue}{\fbox{C#1}}}} %

\def\sqa{\protect\hyperref[fig:giant]{\text{\normalfont\textcolor{blue}{\fbox{A}}}}\xspace}
\def\sqb#1{\protect\hyperref[fig:giant]{\text{\normalfont\textcolor{blue}{\fbox{B#1}}}}\xspace}
\def\sqc#1{\protect\hyperref[fig:giant]{\text{\normalfont\textcolor{blue}{\fbox{C#1}}}}\xspace}

\def\actonintro{\text{act on $1$}}
\def\acton{\text{\begin{tabular}{c}
                   act\\[-5pt] on $1$
                 \end{tabular}}}
\def\Curvex{\Curve_{\bx}}
\def\Curvexo{\Curve_{\bx_0}}
\def\ti{t^{-1}}
\def\Peha{P^{\EHAop}}

\def\ttonemqi{\frac1{\tt(1-q)}}
\def\mul{\delta_{q,t}}
\def\mulqqi{\delta_{q,\qi}}
\def\WCx^#1{\Curvex^{#1}} 

\def\coaxial{coaxial almost torus\xspace}
\def\N{N} %
\def\br{\beta}
\def\brh{\hat\beta}

\def\PtsAbove{A}
\def\PtsAboveC{\PtsAbove_\Curve}
\def\PtsAboveCweak{\PtsAbove'_\Curve}
\def\ab{b}
\def\AB{\bm{b}}
\def\ABC{\AB_\Curve}
\def\SC{S_\Curve}
\def\cox{\operatorname{cox}}

\def\coxC{\operatorname{cox}(\Curve)}
\def\EPS{\bm{\eps}}
\def\EPSC{\EPS_{\Curve}}
\def\epsC_#1{\eps_{#1}}
\def\JM{\ell}
\def\betacox{\beta^{\cox}}
\def\betacoxC{\betacox_\Curve}
\def\sh{\operatorname{sh}}
\def\cshift{\rho}
\def\tx{\tilde x}
\def\SYT{\operatorname{SYT}}
\def\tphi{\tilde\phi}
\def\vrap{\psi_v}

\def\figref#1(#2){Figure~\hyperref[#1]{\ref*{#1}(#2)}}
\def\piC{\pi_\Curve}
\def\Ht{\tilde H}
\def\pol{\operatorname{pol}}
\def\BD{{\mathbf{D}}}

\begin{document}

\title{Monotone links in DAHA and EHA}

\author{Pavel Galashin}
\address{Department of Mathematics, University of California, Los Angeles, 520 Portola Plaza,
Los Angeles, CA 90025, USA}
\email{\href{mailto:galashin@math.ucla.edu}{galashin@math.ucla.edu}}

\author{Thomas Lam}
\address{Department of Mathematics, University of Michigan, 2074 East Hall, 530 Church Street, Ann Arbor, MI 48109-1043, USA}
\email{\href{mailto:tfylam@umich.edu}{tfylam@umich.edu}}
\thanks{P.G.\ was supported by an Alfred P. Sloan Research Fellowship and by the National Science Foundation under Grants No.~DMS-1954121 and No.~DMS-2046915. T.L.\ was supported by Grants No.~DMS-1464693 and No.~DMS-1953852 from the National Science Foundation.}



\date{\today}

\begin{abstract}
We define \emph{monotone links} on a torus, obtained as projections of curves in the plane whose coordinates are monotone increasing. 
Using the work of Morton--Samuelson, to each monotone link we associate elements in the double affine Hecke algebra and the elliptic Hall algebra.  In the case of torus knots (when the curve is a straight line), we recover symmetric function operators appearing in the rational shuffle conjecture.  

We show that the class of monotone links viewed as links in $\mathbb R^3$ coincides with the class of Coxeter links, studied by Oblomkov--Rozansky in the setting of the flag Hilbert scheme.  When the curve satisfies a convexity condition, we recover positroid links that we previously studied. In the convex case, we conjecture that the associated symmetric functions are Schur positive, extending a recent conjecture of Blasiak--Haiman--Morse--Pun--Seelinger, and we speculate on the relation to Khovanov--Rozansky homology. 

Our constructions satisfy a skein recurrence where the base case consists of piecewise almost linear curves.
We show that convex piecewise almost linear curves give rise to algebraic links. 
\end{abstract}

\numberwithin{equation}{section}

\maketitle
\section{Introduction}
Cherednik~\cite{Cherednik_Jones}, reinterpreting earlier work of Aganagic and Shakirov~\cite{AS15}, constructed elements in the \emph{double affine Hecke algebra (DAHA)} associated to an $(m,n)$-torus knot, with the aim of constructing triply-graded link homology.  This construction was expanded on in the influential work of Gorsky and \Negut \cite{GN}, who further constructed elements in the \emph{elliptic Hall algebra (EHA)} using the work of Schiffmann and Vasserot~\cite{SV11,SV13}.  These objects appear at the intersection of beautiful conjectures relating link homology, compactified Jacobians, Hilbert schemes, rational Cherednik algebras, symmetric functions, cluster algebras, and braid varieties; see 
\cite{GNR,GM1,CheDan1,ORS,OY,ObRo,GORS,FPST,STWZ,GL_qtcat,CGGS,GLSBS1,GLSBS2,CGGLSS}. 

There are multiple indications that these algebraic constructions extend beyond the case of torus knots.  For instance, some of the conjectures mentioned above extend to the case of algebraic knots and links~\cite{CheDan2,ORS,KiTs}. In another direction, Oblomkov--Rozansky~\cite{ObRo} related the homology of \emph{Coxeter links} (which include torus knots and torus links) to the space of sections of a line bundle on a Hilbert scheme of points in $\C^2$. 

In this work, we study DAHA, EHA, and symmetric function invariants for the class of monotone links which includes torus knots and torus links. A \emph{monotone link} is a projection (under the map $\R^2\to\R^2/\Z^2$) of a curve $\Curve$ from $(0,0)$ to $(\rectXY)$ with both coordinates monotone increasing; see \cref{sec:intro:links}. 
Our construction depends on viewing monotone links as links in a (thickened) torus. It turns out that viewing them as links in $\R^3$ instead, one recovers precisely the Coxeter links of~\cite{ObRo}; see \cref{sec:positroid_links,sec:cox_links}.

In our earlier works~\cite{GL_qtcat,GL_cat_combin,GL_plabic_links}, we studied positroid links associated to the positroid stratification of the Grassmannian and discovered that the associated combinatorics involves Dyck paths under an arbitrary \emph{convex curve}. This surprising convexity property of positroid links was a starting point for this work.  
  Around the same time, Blasiak--Haiman--Morse--Pun--Seelinger~\cite{BHMPS} constructed symmetric functions in the context of the \emph{shuffle conjecture}~\cite{HHLRU,BGLX,GN,CaMe,Mellit_rat} and independently conjectured that under a convexity assumption, these symmetric functions are Schur positive.

We showed in~\cite{GL_cat_combin} that for each \emph{repetition-free} positroid link, the Euler characteristic of the associated positroid variety equals the number of Dyck paths below a convex curve $\Curve$, and that any convex curve appears in this way. Our proof relied on a Dyck path recurrence involving three curves passing above, below, and through a lattice point, respectively. One of our main results is a lift of this Dyck path recurrence to the $q,t$-level; see \cref{thm:intro:skein}.

Our main conjecture (\cref{conj:intro:Schur_pos}) states that convex monotone links produce Schur positive symmetric functions.  We show that under the specialization $t=\qi$, we recover the HOMFLY polynomial of the link, and we give an explicit formula for our symmetric function when $t = 1$.  We further conjecture that our invariants recover  Khovanov--Rozansky link homology in the case of algebraic links.  For the case of torus links with multiple components, we expect that our conjecture coincides with that of Cherednik--Danilenko \cite{CheDan2} who more generally studied iterated torus cables.

To produce our invariants, we apply the construction of Morton--Samuelson \cite{MS17,MS21}, who connected the skein algebra of the torus and of the punctured torus to the DAHA and the EHA.  At the heart of the technical challenge is the choice of a location for the puncture when lifting our links to the punctured torus; see~\eqref{eq:intro_WC_PTor}.

\introsubsec{Overview}\label{sec:intro_overview}
Let $\rectX,\rectY$ be positive integers. A \emph{monotone curve} (or simply a \emph{curve}) is the graph of a strictly increasing continuous function $f:[0,\rectX]\to[0,\rectY]$ satisfying $f(0)=0$ and $f(\rectX)=\rectY$. 
Given a curve $\Curve$, all of our constructions depend only on the sets of lattice points in $[0,\rectX]\times[0,\rectY]$ which are strictly above, strictly below, and on $\Curve$.

 We compute the image of a curve $\Curve$ inside the commutative diagram in \cref{fig:big}, where:
\begin{itemize}
\item $\Disk$ is a (two-dimensional) disk, $\Ann$ is an annulus, $\Tor$ is a torus, and $\PTor$ is a punctured torus;
\item $\Sk(\Surf)$ is the skein algebra of a surface $\Surf$;
\item $\Laq$ and $\Laqt$ are the algebras of symmetric functions over $\Cq$ and $\Cqt$;
\item $\DASHA_\N$ is the spherical DAHA;
\item $\EHA$ is the EHA;
\item the superscript ${}^+$ denotes the positive part of the corresponding algebra.
\end{itemize}
The full background on the above objects is given in the main body of the paper. Some examples of symmetric functions in $\Laqt$ associated to curves can be found in \cref{tab:small_curves}. We now explain our results in more detail.

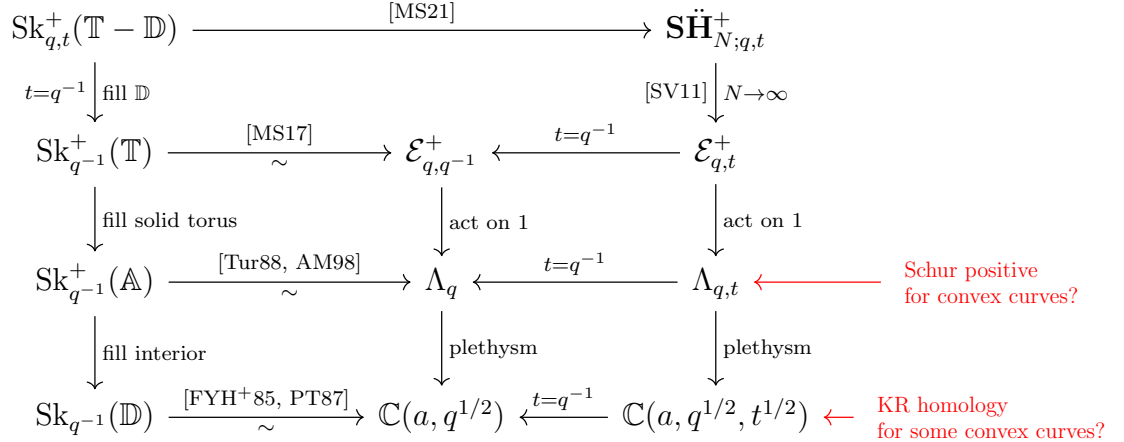
\begin{figure}
$\hspace{0.71in}\displaystyle\begin{tikzcd}[column sep=1.0em,row sep=-0.0em]
\SkpPTor 
\arrow[rrrr,"\text{\cite{MS21}}"] 
\arrow[dd,"\text{fill $\Disk$}","t=\qi"']
&&
 && \DASHAp_\N \arrow[dd,"\N\to\infty","\text{\cite{SV11}}"']\\
&& 
\strut
&&
 \\
\Skp(\Tor) 
\arrow[dd,"\text{fill solid torus}"] 
\arrow[rr,"\text{\cite{MS17}}","\sim"'] 
&& \EHApqi\arrow[dd,"\actonintro"]%
&& 
\EHAp \arrow[ll,"t=\qi"']\arrow[dd,"\actonintro"]%
\\
& 
\strut \qquad\ \  \strut
 && 
\strut 
& \\
\Skp(\Ann)  
\arrow[rr,"\text{\cite{Turaev,AiMo}}","\sim"'] 
\arrow[dd,"\text{fill interior}"]
&& \Laq 
\arrow[dd,"\text{plethysm}"]
 && \Laqt
\arrow[dd,"\text{plethysm}"] 
\arrow[ll,"t=\qi"'] 
&
\scalebox{0.7}{\textcolor{red}{\begin{tabular}{l}
                               Schur positive \\ for convex curves?
                             \end{tabular}}}
\arrow[red]{l}
\\ 
& 
\strut
 && 
\strut
& \\
\Sk(\Disk)  
\arrow[rr,"\text{\cite{HOMFLY,PT}}","\sim"']
 && \Caq && \Caqt \arrow[ll,"t=\qi"'] 
&
\scalebox{0.7}{\textcolor{red}{\begin{tabular}{l}
                               KR homology \\ for some convex curves?
                             \end{tabular}}}
\arrow[red]{l}
\end{tikzcd}$
  \caption{\label{fig:big} A commutative diagram. See \cref{fig:giant} for a description of each map.}
\end{figure}

\introsubsec{Links}\label{sec:intro:links}
Given a curve $\Curve$, let $(0,0)=\point_0,\point_1,\dots,\point_k=(\rectX,\rectY)$ be the lattice points on $\Curve$ listed from left to right. For $i=1,2,\dots,k$, let $\Curve_i$ be the part of $\Curve$ connecting $\point_{i-1}$ to $\point_i$. The curves $\Curve_1,\Curve_2,\dots,\Curve_k$ are called the \emph{lattice segments} of~$\Curve$, and we write $\Curve=[\Cseg_1\Cseg_2\cdots\Cseg_k]$. We denote $\nseg(\Curve):=k$. We say that $\Curve$ is \emph{primitive} if $\nseg(\Curve)=1$, i.e., if $\Curve$ passes through no lattice points in the interior of $[0,\rectX]\times[0,\rectY]$.

Let $\Tor=\R^2/\Z^2$ be the torus and $\pi:\R^2\to\Tor$ be the quotient map. To a curve $\Curve$, we associate a link $\LC$ in $\Tor\times\Interval$, whose \emph{link diagram} is drawn on $\Tor$. Suppose first that $\Curve$ is primitive. Consider the projection $\Curveproj:=\pi(\Curve)$ of $\Curve$ to $\Tor$. Thus, $\Curveproj$ is a curve of homology class $(\rectXY)$. To convert it into a link diagram of $\LC$, for each self-intersection of $\Curveproj$ involving projections of points $(x_1,y_1),(x_2,y_2)\in[0,\rectX]\times[0,\rectY]$ with $x_1<x_2$, we draw the segment containing $(x_1,y_1)$ below the segment containing $(x_2,y_2)$.

Alternatively, if a primitive curve $\Curve$ is a plot of an increasing function $f:[0,\rectX]\to[0,\rectY]$ then the link $\LC$ is a concatenation of two curves in $\Tor\times\Interval$: the curve $\left(\pi(x,f(x)),x/\rectX\right)$, $x\in[0,\rectX]$, and the vertical line segment connecting $((0,0),1)$ to $((0,0),0)$.

The point $\pi(0,0)=\pi(\rectXY)\in\Tor$  is called the \emph{corner} of $\Curve$.

If $\Curve=[\Cseg_1\Cseg_2\cdots\Cseg_k]$ is not necessarily primitive, the link $\LC$ will have $k$ components $\Link_{\Cseg_1},\Link_{\Cseg_2},\dots,\Link_{\Cseg_k}$. The link diagrams of $\Link_{\Cseg_1},\Link_{\Cseg_2},\dots,\Link_{\Cseg_k}$ are drawn on top of each other so that for $i<j$, $\Link_{\Cseg_i}$ is drawn below $\Link_{\Cseg_j}$.

\introsubsec{Skein of a punctured torus}\label{sec:intro_skein}
Given an oriented surface $\Surf$, the \emph{skein algebra} $\Skti(\Surf)$ is the algebra of linear combinations of (isotopy classes of oriented) links inside $\Surf\times \Interval$ 
 subject to the \emph{\FLY skein relation}\footnote{Strictly speaking, we consider framed links, for which there is an extra framing change relation involving another variable $v$; see~\eqref{eq:Sk_dfn}. All our links are assumed to have blackboard framing; see \cref{sec:skein_of_a_surface}.} 
\begin{equation}\label{eq:intro_FLY_skein}
\LPlus-\LMinus=\left(\tt-\tti\right) \LZero.
\end{equation}
The multiplication in $\Skti(\Surf)$ is given by ``stacking the links on top of each other in the $\Interval$ direction,'' i.e., by applying the homeomorphism $(\Surf\times [0,1])\cup(\Surf\times[1,2])=\Surf\times[0,2]\cong\Surf\times[0,1]$.

\begin{figure}
  \setlength{\tabcolsep}{3.8pt}
\begin{tabular}{cccc}
  \includegraphics[width=0.2\textwidth]{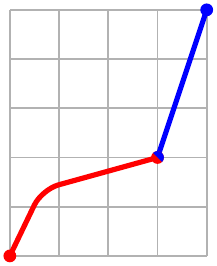}
&
  \includegraphics[width=0.2\textwidth]{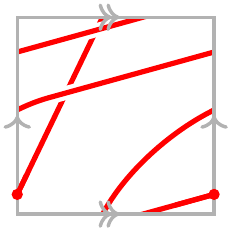}
&
  \includegraphics[width=0.2\textwidth]{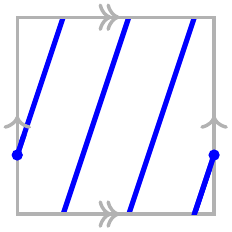}
&
  \includegraphics[width=0.2\textwidth]{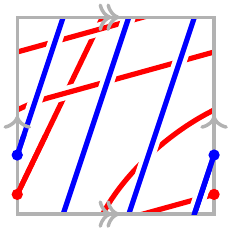}
\\
(a) $\Curve=[\textcolor{red}{\Curve_1}\textcolor{blue}{\Curve_2}]$ & 
(b) $\textcolor{red}{\WC^\Tor_1}$ & 
(c) $\textcolor{blue}{\WC^\Tor_2}$ & 
(d) $\WC^\Tor=\textcolor{red}{\WC^\Tor_1}\cdot \textcolor{blue}{\WC^\Tor_2}$
\end{tabular}
  \caption{\label{fig:mon-curve-proj} Projecting a curve to the torus (\cref{sec:intro_skein}). In the projection, the red and blue segments of the curve are slightly shifted up for clarity.}
\end{figure}

Given a curve $\Curve=[\Cseg_1\Cseg_2\cdots\Cseg_k]$, the link $\LC$ may be naturally viewed as an element of $\Skti(\Tor)$ which we denote $\WC^\Tor$, and we have
\begin{equation*}%
  \WC^\Tor=\W^\Tor_{\Cseg_1}\cdot \W^\Tor_{\Cseg_2}\cdots \W^\Tor_{\Cseg_k}.
\end{equation*}
\noindent For example, in \cref{fig:mon-curve-proj}, $\WC^\Tor=\textcolor{red}{\WC^\Tor_1}\cdot \textcolor{blue}{\WC^\Tor_2}$ is obtained by drawing $\textcolor{blue}{\WC^\Tor_2}$ on top of $\textcolor{red}{\WC^\Tor_1}$.

Let $\PTor$ be the \emph{punctured torus}, obtained by removing a small disk $\Disk$ from $\Tor$. Our main construction associates an element of $\Skti(\PTor)$ to a curve $\Curve$. More precisely, we will work with 
\begin{equation}\label{eq:intro:SkPTor}
  \SkPTor:=\Skti(\PTor)\otimes \C(\qq).
\end{equation}
 Fix a small $\eps>0$ and suppose that $\PTor$ is obtained by removing a ball of radius $\eps/2$ around $(0,0)$ from $\Tor$. Let $\Curve$ be a primitive curve, and let $\Curve_+:=\Curve+(-\eps,\eps)$ and $\Curve_-:=\Curve+(\eps,-\eps)$. When projecting $\Curve_{\pm}$ to $\Tor$, the puncture in $\PTor$ will be slightly below the corner of $\Curve_+$ and slightly above the corner of $\Curve_-$. Viewing the links $\LX_{\Curve_\pm}$ as elements of the skein $\Skti(\PTor)$, we obtain elements $\W^{\PTor}_{\Curve_\pm}\in\Skti(\PTor)$. We define
\begin{equation}\label{eq:intro_WC_PTor}
  \WC^{\PTor}:=\frac{1}{1-q} (\W^{\PTor}_{\Curve_+} - q\W^{\PTor}_{\Curve_-})\quad \in \SkPTor.
\end{equation}

If $\Curve=[\Cseg_1\Cseg_2\cdots\Cseg_k]$ is not necessarily primitive, we set
\begin{equation}\label{eq:intro_WC_PTor_product}
  \WC^{\PTor}:=\W^{\PTor}_{\Cseg_1}\cdot   \W^{\PTor}_{\Cseg_2}\cdots \W^{\PTor}_{\Cseg_k}.
\end{equation}

\begin{remark}
There is a natural algebra homomorphism $\SkPTor \to\Sk(\Tor)$ obtained by setting $t=\qi$ and filling in the puncture $\Disk$. For any curve $\Curve$, this map sends $\WC^{\PTor}\mapsto \WC^\Tor$.
\end{remark}

\begin{remark}
Recall that $\Ann$ denotes the annulus. Then $\Ann\times \Interval$ is homeomorphic to the solid torus, and thus we have a natural action of $\Sk(\Tor)$ on $\Sk(\Ann)$ by identifying $\Tor$ with the boundary of the solid torus. It is well known (see \cref{thm:Tur}) that the skein of $\Ann$ may be identified with the algebra $\Laq$ of symmetric functions. Thus, each curve $\Curve$ gives rise to a symmetric function $\WC^\Tor\cdot 1\in\Laq$ which is the $t=\qi$-specialization of the symmetric function $\FC\in\Laqt$ of interest to us, defined in~\eqref{eq:intro_FC_dfn}.
\end{remark}

\introsubsec{DAHA and EHA elements}
Morton and Samuelson~\cite[Theorem~5.7]{MS21} (see also \cite[Theorem~5.10]{BCMN}) construct a map from the skein of $\PTor$ to the EHA. This map is obtained by constructing, for each $\N=1,2,\dots$, a map from the skein of $\PTor$ to the spherical part of the DAHA $\DAHA_\N$, and then taking a limit as $\N\to\infty$.  We compute the images of $\WC^\PTor$ inside the DAHA and the EHA under the maps of~\cite{MS21}.

By definition, $\DAHA_\N$ is a $\Cqt$-algebra generated by elements $X_i^{\pm1},Y_i^{\pm1}$ for $i=1,\dots,\N$ and $T_i^{\pm1}$ for $i=1,\dots,\N-1$ satisfying certain relations (see \cref{sec:DAHA_backgr}). The \emph{spherical DAHA} $\DASHA_\N$ is defined by $\DASHA_\N:=\S\DAHA_\N\S$, where $\S\in\DAHA_\N$ is a symmetrizing idempotent satisfying $\S^2=\S$ and $T_i\S=\S T_i=\tti\S$ for all $i$; see~\eqref{eq:S_dfn}.

Denote
\begin{equation}\label{eq:gamma_mul_dfn}
  \gam:=1-t^\N \quad\text{and}\quad  \mul:=\ttonemqi.
\end{equation}

\begin{definition}\label{dfn:intro:DCn}
For a curve $\Curve$, let $\PathC$ be the highest up-right lattice path from $(0,0)$ to $(\rectX,\rectY)$ staying weakly below $\Curve$. We construct an element $\DCn\in\DASHA_\N$ obtained by traversing $\PathC$ from $(0,0)$ to $(\rectX,\rectY)$, and taking the product of the following elements: 
\begin{itemize}
\item $\gam \S$ for each lattice point of $\Curve\cap\PathC$ other than $(\rectX,\rectY)$; 
\item $\S$ for the lattice point $(\rectX,\rectY)$;
\item $\Y$ for each up step of $\PathC$;
\item $\Y\X\Y^{-1}$ for each right step of $\PathC$.
\end{itemize}
\end{definition}
\begin{example}\label{ex:intro_DXn}
For the three curves $\Curve_+,\Curve_-,\Curve_0$ shown in \cref{fig:2x2}, we have $\PathX_{\Curve_+}=RURU$, $\PathX_{\Curve_-}=RRUU$, and $\PathX_{\Curve_0}=RU*RU$, where $R$ and $U$ indicate right and up steps and $*$ indicates the places where the curve passes through a lattice point other than $(0,0)$ and $(\rectXY)$. We have
\begin{equation*}%
  \DXn_{\Curve_+}=\gam \S\Y \X \Y \X  \S,\quad 
  \DXn_{\Curve_-}=\gam \S \Y\X \X \Y  \S,\quad 
  \DXn_{\Curve_0}=\gam^2 \S \Y\X \S \Y\X  \S.
\end{equation*}
\end{example}

\begin{figure}
  \setlength{\tabcolsep}{2pt}
\def\wid{0.15}
\begin{tabular}{cc|cc|cc}
\includegraphics[width=\wid\textwidth]{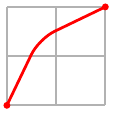}
&
\includegraphics[width=\wid\textwidth]{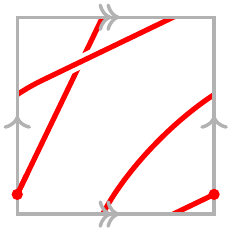}
&
\includegraphics[width=\wid\textwidth]{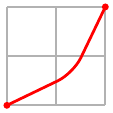}
&
\includegraphics[width=\wid\textwidth]{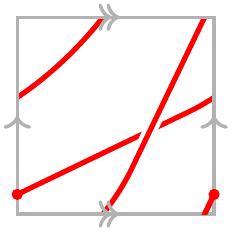}
&
\includegraphics[width=\wid\textwidth]{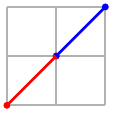}
&
\includegraphics[width=\wid\textwidth]{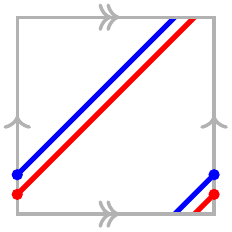}
\\
$\Curve_+$ & $\WC^\Tor_+$ &
$\Curve_-$ & $\WC^\Tor_-$ &
$\Curve_0$ & $\WC^\Tor_0$ 
\end{tabular}
\caption{\label{fig:2x2}Monotone curves from $(0,0)$ to $(2,2)$ and their projections to the torus.}
\end{figure}

 We are ready to state our first main result. 
\begin{theorem}\label{thm:intro_curve_to_DASHA}
Let $\Curve$ be a curve. The image of $\WC^\PTor$ under the map 
\begin{equation}\label{eq:intro_Sk_to_DASHA}
  \SkpPTor\to\DASHAp_\N
\end{equation}
 of~\cite{MS21} is given by
\begin{equation}\label{eq:intro_Sk_to_DASHA_image}
  \mul^{\nseg(\Curve)} \DCn.
\end{equation}
\end{theorem}

The element $\WC^{\PTor}$ belongs to the positive part $\SkpPTor$ of $\SkPTor$ defined in \cref{sec:skein_PTor},
 while the elements $\DCn\in\DASHA_\N$ belong to the positive part $\DASHAp_\N$ of $\DASHA_\N$ defined in~\eqref{eq:DASHAp_EHAp_dfn}.  By~\cite[Theorem~4.6]{SV11}, the positive part $\EHAp$ of the EHA can be realized as a limit of $\DASHAp_\N$ as $\N\to\infty$.
\begin{theorem}\label{thm:intro_limit_exists}
Let $\Curve$ be a curve. The elements $\DCn\in\DASHAp_\N$ for $\N=1,2,\dots$ give rise to a well-defined limiting element
  \begin{equation}\label{eq:intro_limit_exists}
    \lim_{\N\to\infty} \DCn=\DClimEHA \in\EHAp.
  \end{equation}  
\end{theorem}
\noindent An explicit algorithm for expressing $\DClimEHA$ in terms of the standard generators of $\EHAp$ is given in \cref{sec:skein_relation}. An immediate consequence of \cref{dfn:intro:DCn} (using that $\S^2=\S$) is that for any curve $\Curve=[\Cseg_1\Cseg_2\cdots\Cseg_k]$, the elements $\DCn$ and therefore $\DClimEHA$ satisfy
  \begin{equation}\label{eq:intro:prod}
    \DCn=\DXn_{\Curve_1}\cdot \DXn_{\Curve_2}\cdots \DXn_{\Curve_k} \quad\text{and}\quad 
\DClimEHA=\DXlimEHA_{\Curve_1}\cdot \DXlimEHA_{\Curve_2}\cdots \DXlimEHA_{\Curve_k},
  \end{equation}
where the products are taken in $\DASHAp_\N$ and $\EHAp$, respectively.

\introsubsec{Skein relation and piecewise almost linear curves}\label{sec:intro:PAL}

It turns out that the elements $\DCn$ and $\DC$ satisfy the following skein relation.
\begin{theorem}\label{thm:intro:skein}
  Consider curves $\Curve_+,\Curve_-,\Curve_0$ which pass above, below, and through some lattice point $\point$, respectively, and agree outside a small neighborhood of $\point$. Then
  \begin{equation}\label{eq:intro:skein}
    \DXn_{\Curve_+}=qt\DXn_{\Curve_-}+\DXn_{\Curve_0} \quad\text{and}\quad
    \DXlimEHA_{\Curve_+}=qt\DXlimEHA_{\Curve_-}+\DXlimEHA_{\Curve_0}.
  \end{equation}
\end{theorem}

This relation allows one to express any element $\DClimEHA$ in terms of elements associated to \emph{piecewise almost linear} curves, defined as follows. 
\begin{definition}\label{dfn:intro_AL}
 A primitive curve $\Curve$ from $(0,0)$ to $(\rectXY)$ is called \emph{$(\rectXY)$-almost linear}, or simply \emph{almost linear}, if it passes just above the diagonal connecting $(0,0)$ and $(\rectXY)$. (That is, it passes above each of the $\gcd(\rectXY)-1$ lattice points in the interior of the diagonal, and above/below all lattice points which are strictly below/above the diagonal.)
\end{definition}

As we explain in \cref{sec:shuffle}, for $C$ an almost linear curve, the operators $D_C$  are exactly the symmetric function operators appearing in various versions of the shuffle conjecture.

\begin{definition}\label{dfn:intro:PAL}
  A curve $\Curve=[\Cseg_1\Cseg_2\cdots\Cseg_k]$ is called \emph{piecewise almost linear} if each 
  lattice segment $\Curve_i$ of $\Curve$ is almost linear.
\end{definition}
\begin{remark}\label{rmk:intro:skein}
Applying~\eqref{eq:intro:skein} repeatedly, we may express the element $\ucb$ for an arbitrary curve $\Curve$ as a linear combination of elements of the form $\ucbp$, where $\Curve'$ is a \PAL curve. 
 Combining this with~\eqref{eq:intro:prod} allows one to express an arbitrary element $\ucb$ in terms of products of elements $\ubx_{\Curve'_i}$ corresponding to almost linear curves $\Curve'_i$.
\end{remark}

\introsubsec{Symmetric functions}
The EHA $\EHAp$ acts on the ring $\Laqt$ of symmetric functions over $\Cqt$ as described in \cref{sec:EHA_action}. We will be particularly interested in the action $\ucb\ehact1$ of $\ucb$ on $1\in\Laqt$.  The DAHA admits a polynomial representation described in \cref{sec:DAHA_polyrep}. For $\N=1,2,\dots$, the element $\DCn\dahact1$ is a symmetric polynomial in $x_1,\dots,x_\N$ obtained from the symmetric function $\ucb\ehact1$ by setting $x_{\N+1}=x_{\N+2}=\cdots=0$.

It is convenient to consider the plethystically transformed (cf. \cref{sec:relations}) symmetric functions
\begin{equation}\label{eq:intro_FC_dfn}
  \FC:=(\ucb\ehact1) \left[\frac{\XSym}{1-t}\right] \in\Laqt,
\end{equation}
where $\XSym:=x_1+x_2+\dots \in\Laqt$. Using another plethystic substitution, we define the \emph{EHA superpolynomial} %
\begin{equation}\label{eq:intro_PEHAC_dfn}
  \PEHAC(a,q,t):=\FC\left[a-a^{-1}\right].
\end{equation}
In the notation of \cref{thm:intro:skein}, we therefore get
\begin{equation}\label{eq:intro:skein_F}
  \FX_{\Curve_+}=qt\FX_{\Curve_-}+\FX_{\Curve_0} \quad\text{and}\quad \PEHAX_{\Curve_+}=qt\PEHAX_{\Curve_-}+\PEHAX_{\Curve_0}.
\end{equation}
As explained in \cref{rmk:q_t_symm}, the symmetric function $\FC$ is $q,t$-symmetric (i.e., unchanged under swapping $q$ and $t$). 
Note that~\eqref{eq:intro:prod} does not extend to symmetric functions: for $\Curve=[\Cseg_1\Cseg_2\dots\Cseg_k]$, knowing the symmetric functions
  $\ubx_{\Curve_1}\ehact1,\ubx_{\Curve_2}\ehact1,\dots,\ubx_{\Curve_k}\ehact1$
  does not allow one to reconstruct
  \begin{equation*}%
    \ucb\ehact1=\ubx_{\Curve_1} \ehact{( \ubx_{\Curve_2} \ehact{( \cdots \ubx_{\Curve_k} \ehact{ 1  } \cdots  )}  )}.
  \end{equation*}

\introsubsec{Convex curves}
The following class of curves is motivated by our results on cohomology of open positroid varieties~\cite{GL_qtcat,GL_cat_combin}.
\begin{definition}
We say that a curve $\Curve$ is \emph{convex} if it is a plot of a (weakly) convex function $f:[0,\rectX]\to[0,\rectY]$.
\end{definition}
We also consider the following slightly more general class of curves.
\begin{definition}\label{dfn:intro_Z_convex}
A set $\PtsAbove\subset\R^2$ is called \emph{$\Z$-convex} if $\Conv(\PtsAbove)\cap\Z^2=\PtsAbove\cap \Z^2$. For a curve $\Curve$ passing through lattice points $(0,0)=\point_0,\point_1,\dots,\point_k=(\rectX,\rectY)$, let $\PtsAboveC$ be the set of all lattice points in $[0,\rectX]\times[0,\rectY]$ strictly above $\Curve$ together with the set of (non-lattice) points $\{\point_i+(-\eps,\eps)\mid i=0,1,\dots,k\}$, where $\eps>0$ is small. We say that $\Curve$ is \emph{$\Z$-convex} if $\PtsAboveC$ is a $\Z$-convex set.
\end{definition}
For example, the curves in \figref{fig:convex-examples}(a,b,d) are $\Z$-convex while the curves in \figref{fig:convex-examples}(c,e) are not.

It is clear that any convex curve is $\Z$-convex, but the converse need not hold. For instance, an almost linear curve (\cref{dfn:intro_AL}) from $(0,0)$ to $(\rectXY)$ is always $\Z$-convex but it is convex only if $\gcd(\rectXY)=1$. It is not hard to check that a \PAL curve $\Curve=[\Cseg_1\Cseg_2\dots\Cseg_k]$, where each $\Cseg_i$ is $(m_i,n_i)$-almost linear, is $\Z$-convex if and only if
\begin{equation}\label{eq:Z_convex_geq}
  \frac{m_1}{n_1}\geq\frac{m_2}{n_2}\geq\cdots\geq\frac{m_k}{n_k}.
\end{equation}

\begin{figure}
 \includegraphics[width=1.0\textwidth]{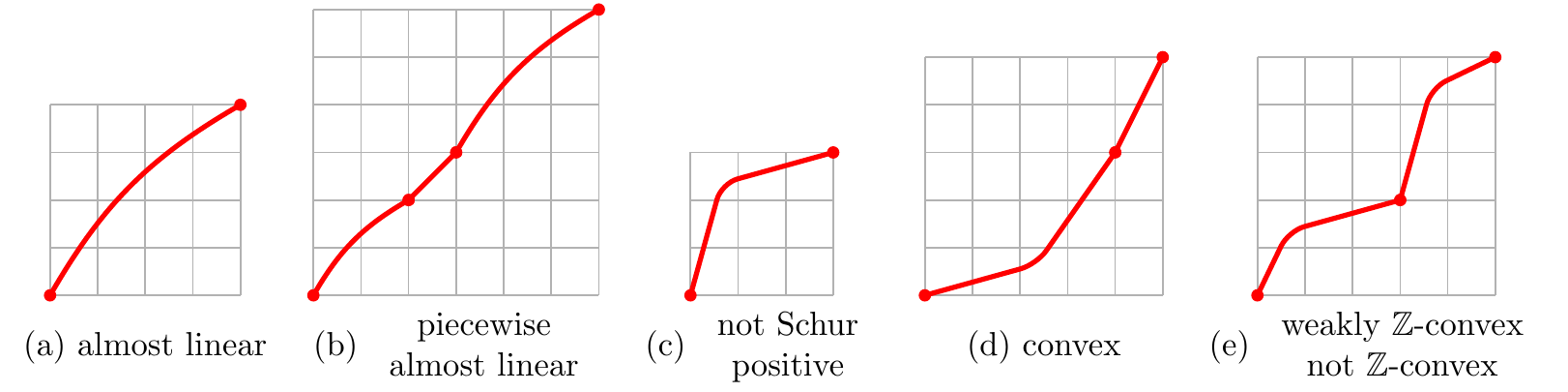}
  \caption{\label{fig:convex-examples} Examples of monotone curves. See also \cref{tab:small_curves}.}
\end{figure}

The following conjecture extends~\cite[Conjecture~7.1.1]{BHMPS} in view of \cref{rem:BHMPS}.
\begin{conjecture}\label{conj:intro:Schur_pos}
  For a $\Z$-convex curve $\Curve$, the formal power series
  \begin{equation}\label{eq:intro_Schur_pos}
    \frac1{(1-t)^{\nseg(\Curve)-1}} \FC
  \end{equation}
  is $q,t$-Schur positive.
\end{conjecture}
\noindent Here $\frac1{1-t}=1+t+t^2+\cdots$ is viewed as a formal power series. See \cref{ex:intro_percentage}, \cref{sec:intro_example}, \cref{fig:convex-examples}, and \cref{tab:small_curves} for examples. 

\begin{remark}
In fact, computational evidence (see \cref{ex:intro_percentage}) suggests that the class of $\Z$-convex curves can be slightly enlarged, as follows. In the notation of \cref{dfn:intro_Z_convex}, let $\PtsAboveCweak$ be the set of all lattice points in $[0,\rectX]\times[0,\rectY]$ strictly above $\Curve$ together with $\{\point_0+(-\eps,\eps),\point_k+(-\eps,\eps)\}$. We say that $\Curve$ is \emph{weakly $\Z$-convex} if $\PtsAboveCweak$ is $\Z$-convex. We conjecture that the formal power series~\eqref{eq:intro_Schur_pos} is Schur positive for all weakly $\Z$-convex curves. 
\end{remark}
\begin{example}\label{ex:intro_percentage}
For $1\leq \rectXY\leq 7$, there are $24,319$ monotone curves. Among them:
\begin{enumerate}
\item $6,781$ ($27\%$) have Schur positive formal power series~\eqref{eq:intro_Schur_pos};
\item $4,257$ ($17\%$) are weakly $\Z$-convex;
\item $3,313$ ($13\%$) are $\Z$-convex;
\end{enumerate}
and we have inclusions $(1)\supset(2)\supset(3)$.
\end{example}

\begin{remark}\label{rmk:intro_torus_link_knot}
Let $\Curve$ be a straight line segment from $(0,0)$ to $(\rectX,\rectY)$.  Then $\LC$ is the \emph{$(\rectX,\rectY)$-torus link}.  If $\gcd(\rectX,\rectY)=1$, then $\LC$ has a single component, and is the \emph{$(\rectX,\rectY)$-torus knot}.  If instead $\Curve$ is the almost linear curve from $(0,0)$ to $(\rectXY)$ then $\LC$ is a knot which we call the \emph{$(\rectXY)$-almost torus knot}.  This knot is algebraic and corresponds to a plane curve singularity with Puiseaux exponents $(n,m,m+1)$ studied in~\cite{Piontkowski,GMO}; see \cref{sec:multitorus} for further discussion.
\end{remark}
The following class includes the \emph{coaxial torus links} considered in~\cite[Section~17b)]{EN}.
\begin{definition}
  When $\Curve$ is a \PAL curve, we call $\LC$ a \emph{\coaxial link}. If $\Curve$ is in addition $\Z$-convex, then $\LC$ is called a \emph{$\Z$-convex \coaxial link}. %
\end{definition}

Recall that an \emph{algebraic link} is an intersection of a small $3$-sphere (inside $\C^2\cong\R^4$) with the zero set $\{(x,y)\in\C^2\mid f(x,y)=0\}$ of solutions to a polynomial equation.  We prove the following in \cref{sec:multitorus}.
\begin{proposition}\label{prop:intro_algebraic}
$\Z$-convex \coaxial links are algebraic. %
\end{proposition}

In \cref{cor:non_recursive}, we give an explicit formula expressing the element $\DClimEHA$ for an arbitrary $\Z$-convex curve $\Curve$ in terms of $\Z$-convex \PAL curves. In particular, the invariants we associate to $\Z$-convex links are expressed (with predictable alternating signs) in terms of invariants of algebraic links. We do not have an analog of this result where ``$\Z$-convex'' is replaced with ``weakly $\Z$-convex.'' See \figref{fig:convex-examples}(e).

\subsection{Khovanov--Rozansky homology}
To a link $\Link$, one can associate a link invariant 
 $\PKR(\Link;a,q,t)$  which is a Laurent polynomial in $a,\qq,\tt$ encoding the graded dimensions of \emph{Khovanov--Rozansky (KR) homology} of $L$; see~\cite{KR1,KR2,KhoSoe}.\footnote{More precisely, the link invariant $\PKRLC(a,q,t)$ differs from the triply-graded \Poincare series encoding the dimensions of KR homology by a monomial in $a,\qq,\tt$.} When the odd KR homology of $\Link$ vanishes, $\PKR(\Link;a,q,t)$ becomes a Laurent polynomial in $a,q,t$ (involving no half-integer powers of $q,t$).

It is well known that at $t=\qi$ (more precisely, at $\tt=-\qqi$), $\PKR(\Link;a,q,t)$ specializes to the \emph{\FLY polynomial} $\PHOML(a,q)$ of $L$. We show that the same holds for the superpolynomial $\PEHAC(a,q,t)$ defined in~\eqref{eq:intro_PEHAC_dfn}.
\begin{proposition}\label{prop:tqi}
For any curve $\Curve$, we have 
\begin{equation}\label{eq:intro_PEHAC=HOMFLY}
  \mulqqi^{\nseg(\Curve)}\cdot \PEHAC(a,q,t=\qi)=a^{\wC} \cdot  \PHOMLC(a,q),%
\end{equation}
where $\wC\in\Z$ is the \emph{writhe} of the braid $\br^\Ann_\Curve$ associated to $\Curve$ in \cref{dfn:TC}.
\end{proposition}
\noindent  An explicit formula for $\wC$ in terms of lattice points below $\Curve$ is given in~\eqref{eq:wC}. Here $\PHOML(a,q)$ has slightly non-standard normalization; see~\eqref{eq:FLY_dfn}.

A natural question arises: what is the relationship between $\PEHAC(a,q,t)$ and $\PKRLC(a,q,t)$? For arbitrary curves, it is easy to find examples where these two polynomials disagree: see \figref{fig:convex-examples}(c) and \cref{ex:non_convex_positroid}. However, the two polynomials appear to be very closely related for the class of $\Z$-convex curves. The following question has motivated most of our results.
\begin{question}\label{que:intro:KR}
For which $\Z$-convex curves $\Curve$ do we have
  \begin{equation}\label{eq:intro:KR}
   \PEHAC(a,q,t)= \left(1-t\right)^{\nseg(\Curve)-1} \PKRLC(a,q,t)
 \end{equation}
 (up to multiplication by an explicit monomial in $a,\qq,\tt$)?
\end{question}
\noindent Note that if~\eqref{eq:intro:KR} holds for $\Curve$ then the odd KR homology of $\LC$ vanishes. We refer to \cref{ex:non_unimodal} for an example of a convex curve $\Curve$ (discovered by A.~Mellit and brought to our attention by M.~Mazin) for which~\eqref{eq:intro:KR} does not hold. In particular, this example disproves~\cite[Conjecture~7.1(ii,iv)]{GL_cat_combin}.

\begin{remark}
  It was shown in~\cite[Corollary~1.4]{MeHog} (see also~\cite{Mellit_rat,ElHo,Mellit_torus}) that torus links have vanishing odd KR homology. Recall from \cref{rmk:intro_torus_link_knot} that torus links are special cases of links $\LC$ when $\Curve$ is a straight line segment from $(0,0)$ to $(\rectXY)$. It would be interesting to verify directly using the results of~\cite{MeHog} that~\eqref{eq:intro:KR} holds for torus links.  This case is also covered by the conjectures of Cherednik--Danilenko \cite{CheDan2}.
\end{remark}

\begin{remark}
It would be interesting to interpret the skein relation \eqref{eq:intro:skein} in terms of KR homology.  We note the similarity to the results of Gorsky and \Negut (for example \cite[Theorem 5.3]{GN_AH}) on traces in the affine Hecke category, which is related to annular, or affine KR homology. See also \cref{rmk:skein_diagrammatic}.
\end{remark}

\begin{conjecture}
Equation~\eqref{eq:intro:KR} holds for convex curves $\Curve$ such that the link $\LC$ is algebraic.
\end{conjecture}

\introsubsec{An example}\label{sec:intro_example}
  Let $(\rectX,\rectY)=(2,2)$, and consider the convex curves $\Curve_+$, $\Curve_-$, $\Curve_0$ shown in \cref{fig:2x2}; see also \cref{fig:TC}. Thus, $\Link_+:=\Link_{\Curve_+}$ is the (positive) trefoil knot, $\Link_-:=\Link_{\Curve_-}$ is the unknot, and $\Link_0:=\Link_{\Curve_0}$ is the (positive) Hopf link. The associated 
symmetric functions $\FC$ and normalized EHA superpolynomials $\PEHATC(a,q,t) := \PEHAC(a,q,t)/(a-a^{-1})$ are 
given by
\begin{align}%
\label{eq:intro_ex_FC}
  \FX_{\Curve_+}&=s_{11}+(q+t) s_2, &\FX_{\Curve_-}&=s_2, &\FX_{\Curve_0}&=s_{11}+(q+t-qt) s_2;\\
\label{eq:intro_ex_PEHAC}
  \PEHATX_{\Curve_+}&=a^3\cdot\left(\frac{q + t}{a^{2}} - \frac{1}{a^{4}}\right), &\PEHATX_{\Curve_-}&=a^1\cdot 1, &\PEHATX_{\Curve_0}&=a^3\cdot \left(\frac{q + t - qt}{a^{2}} - \frac{1}{a^{4}}\right).
\end{align}
(See \cref{ex:DC_and_Sym} for the full computation.) We see that the skein relation~\eqref{eq:intro:skein_F} holds in this case. 
 All of the above expressions are $q,t$-symmetric. In addition, $\FX_{\Curve_+},\FX_{\Curve_-}$ are $q,t$-Schur positive, and so is
\begin{equation*}%
  \frac1{1-t} \FX_{\Curve_0}=(1+t+t^2+t^3+\cdots) s_{11} + (q+t+t^2+t^3+\cdots) s_{2},
\end{equation*}
in agreement with \cref{conj:intro:Schur_pos}. Finally, computing KR homology of $L_+,L_-,L_0$ (see e.g. \cite[Theorem~1.2]{MeHog}, \cite[Section~1.2]{Mellit_torus}, or \cite[Section~3.5]{GL_qtcat}), one checks that indeed the odd KR homology of $L_+$, $L_-$, and $L_0$ vanishes, and~\eqref{eq:intro:KR} holds for these links. We have $\wX(\Curve_+)=3$, $\wX(\Curve_-)=1$, $\wX(\Curve_0)=3$, in agreement with~\eqref{eq:intro_PEHAC=HOMFLY}.

\introsubsec{Comparisons and specializations} Various special cases of our operators $\DC$ have appeared before in the literature, and we list some comparisons in \cref{sec:relations}.  The operator $\DC$ simplifies considerably when $t=1$, and we show in \cref{prop:t1} that
\begin{equation*}%
  F_C|_{t = 1} = \sum_{\Pcal} q^{\area(\PathC) - \area(\Pcal)} h_{\Pcal},
\end{equation*}
summed over lattice paths $\Pcal$ weakly below $C$, such that $\Pcal$ passes through all the lattice points that $C$ passes through.  Here, $h_{\Pcal}$ denotes a complete homogeneous symmetric function and $\area(\PathC) - \area(\Pcal)$ is a statistic on lattice paths to be defined in \cref{sec:relations}.

\introsubsec{Structure of the paper} We give background on DAHA and EHA in \cref{sec:DAHA,sec:EHA}, respectively. Along the way, we analyze the operators $\DCn$ associated to almost linear curves. In \cref{sec:skein_relation}, we prove the skein relation (\cref{thm:intro:skein}) and use it to prove \cref{thm:intro_limit_exists}. We also give a non-recursive formula (\cref{cor:non_recursive}) expressing an arbitrary $\Z$-convex curve in terms of $\Z$-convex \PAL curves. In \cref{sec:skein}, we discuss the skein-theoretic formalism of~\cite{MS17,MS21} and prove \cref{thm:intro_curve_to_DASHA}. In \cref{sec:diagram}, we describe the maps in the diagram in \cref{fig:big}, the image of a curve $\Curve$ under these maps, and check that the diagram commutes. %
In \cref{sec:relations}, we explain how our objects are related to the objects studied in~\cite{BHMPS,BGLX,Negut,GL_cat_combin,ObRo}. In particular, we identify monotone links in an annulus with Coxeter links in  \cref{prop:Coxeter=monotone} and discuss the $t=1$ specialization of our results in \cref{sec:specialization_t=1}. 
In \cref{sec:multitorus}, we review the results of~\cite{EN} on algebraic links and prove \cref{prop:intro_algebraic}.  Finally, we list some %
examples in \cref{sec:examples}.

\subsection*{Acknowledgments}
We are grateful to Misha Mazin, Eugene Gorsky, and Jos\'e Simental for stimulating discussions. We are also grateful to Paul Wedrich and Peter Samuelson for their comments on the first version of the manuscript.

\section{From curves to DAHA}\label{sec:DAHA}
In this section, we review the background on the DAHA and associate an element $\DXn_{\Curve} \in \DASHA_\N$ in the spherical DAHA to a monotone curve $\Curve$. 
\subsection{Double affine Hecke algebra}\label{sec:DAHA_backgr}
Let $\N\geq1$. The \emph{double affine Hecke algebra (DAHA)} introduced by Cherednik~\cite{Cherednik_book} is the $\Cqt$-algebra generated by $X_i^{\pm1},Y_i^{\pm1}$ for $i=1,2,\dots,\N$, and $T_i$ for $i=1,2,\dots,\N-1$ subject to the following relations:
\begin{gather}
\label{eq:T_i_relns}
  (T_i+\tt)(T_i-\tti) = 0,\quad T_iT_{i+1}T_i=T_{i+1}T_iT_{i+1}, \\
  T_iT_k=T_kT_i \quad\text{if $|i-k|>1$}, \\
  X_jX_k=X_kX_j \quad\text{and}\quad Y_jY_k=Y_kY_j \quad\text{for all $1\leq j,k\leq \N$},\\
  T_iX_iT_i=X_{i+1},\quad  T_i^{-1}Y_iT_i^{-1}=Y_{i+1},\\
  T_iX_k=X_kT_i \quad\text{and}\quad T_iY_k=Y_kT_i \quad\text{if $|i-k|>1$},\\
  Y_1X_1X_2\cdots X_\N = qX_1X_2\cdots X_\N Y_1,\\
  X_1^{-1}Y_2=Y_2X_1^{-1}T_1^{-2}.
\end{gather}
\begin{notation}
In the above relations, we follow the conventions of~\cite{SV11,MS21}, which differ from those of~\cite{Cherednik_book,GN} by substituting $\tt\mapsto \tti$.
\end{notation}
\begin{notation}\label{not:T_i}
Following~\cite{MS21} (cf. \cref{not:MS21}), we represent each $T_i$ by a \emph{negative} crossing $\LMinus$ in our figures. Thus,~\eqref{eq:T_i_relns} is consistent with~\eqref{eq:intro_FLY_skein}.
\end{notation}

Let $\Sn$ denote the symmetric group on $\N$ letters. For a permutation $w\in\Sn$, let $\ell(w)$ be the number of inversions in $w$. Given a reduced word $w=s_{i_1}s_{i_2}\cdots s_{i_{\ell(w)}}$, where $s_i=(i,i+1)$ are simple transpositions, we let $T_w:=T_{i_1}T_{i_2}\cdots T_{i_{\ell(w)}}$; the result depends only on $w$.

The symmetrizing idempotent is given by
\begin{equation}\label{eq:S_dfn}
  \S :=\frac{\sum_{w\in\Sn} t^{-\ell(w)/2} T_w}{\sum_{w\in\Sn} t^{-\ell(w)}}.
\end{equation}
 The symmetrizer $\S$ satisfies
\begin{equation}\label{eq:ST=TS=S}
  T_i\S =\S T_i=\tti \S  \quad\text{and}\quad \S^2=\S .
\end{equation}
The \emph{spherical DAHA} is defined by 
\begin{equation*}%
  \DASHA_\N:=\S \DAHA_\N \S.
\end{equation*}

For $1\leq i\leq j\leq \N$, let $\SNXX_i^j:=\{w\in \Sn\mid w(k)=k\text{ for $k<i$ or $k>j$}\}$, and set
\begin{equation}\label{eq:SX_k_dfn}
  \SXX_i^j:=\frac{\sum_{w\in\SNXX_i^j} t^{-\ell(w)/2} T_w}{\sum_{w\in\SNXX_i^j} t^{-\ell(w)}}.
\end{equation}
Thus, $\SX_1=\S$. The partial symmetrizers $\SX_k$ satisfy~\eqref{eq:ST=TS=S} for all $i=k,k+1,\dots,\N-1$, and can be computed inductively using the following well-known recurrence (see e.g.~\cite[Section~1.2]{ElHo} and \cref{rmk:skein_diagrammatic}):
\begin{equation}\label{eq:symmetrizer_recurrence}
\SX_{k+1} T_k\cdots T_{\N-1} T_{\N-1}\cdots T_k = -t(1-t^{-\N-1+k}) \SX_k + t\SX_{k+1},\quad \SX_{\N}=1.
\end{equation}

\subsection{Polynomial representation and \texorpdfstring{$\SL_2(\Z)$}{SL\_2(Z)}-action}\label{sec:DAHA_polyrep}
The DAHA $\DAHA_\N$ admits a polynomial representation $\phi_\N$ on $\Cqt[x_1,\dots,x_\N]$. For a polynomial $f(x_1,\dots,x_\N)$, set
\begin{align*}
  (\cshift\cdot f)(x_1,\dots,x_{\N-1},x_\N)&:=f(qx_\N,x_1,\dots,x_{\N-1}),\\
  (s_i\cdot f)(\dots,x_i,x_{i+1},\dots)&:=f(\dots,x_{i+1},x_i,\dots).
\end{align*}
The operator $\phi_\N(X_i)$ acts by multiplication by $\qi x_i$, and $T_i,Y_i$ correspond to the following operators; cf.~\cite[Lemma~4.2]{SV11}:
\begin{align}
\label{eq:T_act}
  \phi_\N(T_i) &= \tti s_i +\frac{\tti-\tt}{x_i/x_{i+1} - 1} (s_i - 1),\\
\label{eq:Y_act}
  \phi_\N(Y_i)&=t^{(\N-1)/2}\phi_\N(T_i\cdots T_{\N-1}) \cshift \phi_\N(T_1^{-1}\cdots T_{i-1}^{-1}).
\end{align}
\begin{notation}
We usually omit $\phi_\N$ from the notation and write $D\cdot f$ in place of $\phi_\N(D)\cdot f$ for $D\in\DAHA_\N$ and $f\in\Cqt[x_1,\dots,x_\N]$.
\end{notation}
\begin{remark}\label{rmk:polyrep_rescaling}
Let $\phi'_\N$ be the representation defined in~\cite[Lemma~4.2]{SV11} and~\cite[Section~1.4]{SV13}. Then we have $\phi_\N(T_i)=\phi'_\N(T_i)$, $\phi_\N(X_i)=\qi\phi'_\N(X_i)$, and $\phi_\N(Y_i)=t^{(\N-1)/2}\phi'_\N(Y_i)$. This rescaling is consistent\footnote{The formula for $\phi'_\N(T_i)$ in~\cite[Section~1.4]{SV13} appears to contain a typo: $\tt s_i$ should be replaced with $\tti s_i$. Additionally, below~\cite[Equation~(1.12)]{SV13}, the authors set $y_1=1$ as opposed to $y_1=\qi$. However, in order for $P_{1,0}^\infty$ to act by multiplication by $p_1$ as claimed in~\cite[Proposition~1.4]{SV13}, one needs to set $y_1=\qi$.} with the rescaling $P_{i,j}^\N\mapsto y_1^iy_2^jP_{i,j}^\N$ in~\cite[Equation~(1.12)]{SV13}, with $y_1=q^{-1}$ and $y_2=t^{(\N-1)/2}$.
\end{remark}

Consider the braid group $\BraidGroup_3$ with generators $\taup,\taum$ and relations $\taup \taum^{-1}\taup=\taum^{-1}\taup\taum^{-1}$. We have an 
action of $\BraidGroup_3$ by automorphisms on $\DAHA_\N$ defined by\footnote{Note that both~\cite[Section~2.1]{SV11} and~\cite[Equation~(13)]{GN} contain the same typo in the formulas~\eqref{eq:taup_dfn}--\eqref{eq:taum_dfn}: their $T_i$ should be $T_1$.}
\begin{align}%
\label{eq:taup_dfn}  \taup(T_i)&=T_i, &\taup(X_i)&=X_i, &\taup(Y_i)&=Y_iX_i(T_{i-1}^{-1}\cdots T_{1}^{-1})(T_{1}^{-1}\cdots T_{i-1}^{-1});\\
\label{eq:taum_dfn}  \taum(T_i)&=T_i, &\taum(Y_i)&=Y_i, &\taum(X_i)&=X_iY_i(T_{i-1}\cdots T_{1})(T_{1}\cdots T_{i-1}).
\end{align}
In particular,
\begin{equation*}%
  \taup(\X)=\X,\quad \taup(\Y)=\Y\X,\quad \taum(\X)=\X\Y,\quad \taum(\Y)=\Y,\quad \taup(\S )=\taum(\S )=\S .
\end{equation*}

The operators $\taup,\taum$ preserve $\DASHA_\N$, and under the identification
\begin{equation}\label{eq:taupm_to_SL2}
  \taup \mapsto \begin{pmatrix}
1 & 1\\ 0 & 1
  \end{pmatrix},\quad \taum\mapsto \begin{pmatrix}
1 & 0\\ 1 & 1
  \end{pmatrix},
\end{equation}
we obtain an action of $\SL_2(\Z)$ on $\DASHA_\N$. See e.g.~\cite[Proposition~5.2]{MS21} for an explicit check of the  $\SL_2(\Z)$ relation $(\taup\taum^{-1}\taup)^4=\id$.\footnote{We thank Jos\'e Simental for discussions regarding the relation $(\taup\taum^{-1}\taup)^4=\id$.}

For any polynomial $f(x_1,\dots,x_\N)$, the polynomial $\S\cdot  f$ is symmetric in $x_1,\dots,x_\N$. Thus, the polynomial representation of $\DAHA_\N$ restricts to an action of $\DASHA_\N$ on the space of $\LaqtN$ of symmetric polynomials in $\N$ variables over $\Cqt$.

We denote
\begin{equation*}%
  \Ztnn:=\{x\in\Z\mid x\geq0\},\quad \Ztp:=\{x\in\Z\mid x>0\}, \quad\text{and}
\end{equation*}
\begin{equation*}%
  \BZ:=\Z^2,\quad \BZast:=\BZ\setminus\{(0,0)\},\quad \BZgeq:=\Ztnn\times \Z,\quad \BZx:=\Ztp\times\Z,\quad \BZxy:=\BZx\sqcup (\{0\}\times\Ztnn).
\end{equation*}

For $d\geq1$, let
\begin{equation}\label{eq:Pgen_dfn}
  \Pgen_{d,0}:=q^d\S \sum_{i=1}^\N X_i^d \S \quad\text{and}\quad \Pgen_{0,d}:=\S \sum_{i=1}^\N Y_i^d \S.
\end{equation}
For an arbitrary vector $(\rectXY)\in\BZast$ with $\gcd(\rectXY)=d$, choose an element $g\in\SL_2(\Z)$ sending $(d,0)$ to $(\rectXY)$ and set
\begin{equation}\label{eq:SL2_DAHA_Pgen}
  \Pgen_{\rectXY}:=g \Pgen_{d,0}.
\end{equation}
The resulting element does not depend on the particular choice of $g$. %
By~\cite[Proposition~2.5]{SV11}, the elements $\Pgen_{\bx},\bx\in\BZast$ generate $\DASHA_\N$ as a $\Cqt$-algebra.

\subsection{Comparing two $\SL_2(\Z)$-actions}
The group $\SL_2(\Z)$ acts on $\DASHA_\N$, and it also acts linearly on $\R^2$ (and thus on curves in $\R^2$). The goal of this section is to show that these two actions are compatible with \cref{dfn:intro:DCn}; see \cref{cor:g_DC=D_gC}.

Choose a nonzero vector $\PV\in\R_{\geq0}^2\setminus\{0\}$. Consider the \emph{$\PV$-punctured torus} $\TPV:=\Tor\setminus\{\eps\PV\}$ for some small $\eps>0$.
In other words, $\TPV=\RPV/\Z^2$, where
$\RPV=\R^2\setminus \{\point+\eps\PV\mid \point\in\BZ\}$. 
 Since $\TPV$ is homotopy equivalent to a wedge of two circles, its fundamental group $\pi_1(\TPV,0)$ (with loops based at $0\in\TPV$) is the free product $\Z*\Z$. We choose two particular generators $\XPV,\YPV\in\pi_1(\TPV,0)$ such that the loop representing $\XPV$ (resp., $\YPV$) lifts to a path in $\RPV$ from $(0,0)$ to $(1,0)$ (resp., from $(0,0)$ to $(0,1)$) passing below (resp., to the left of) the 
point $\eps\PV$ 
 as shown in \figref{fig:puncture-XY}(middle). Since $\pi_1(\TPV,0)$ is freely generated by $\XPV,\YPV$, we can define 
 a group homomorphism%
\begin{equation*}%
  \psiPV: \pi_1(\TPV,0)\to \DAHA_\N^\times:\qquad \XPV\mapsto \X,\quad \YPV\mapsto \Y.
\end{equation*}

Let $\Curve:[0,1]\to\RPV$ be an arbitrary curve with $\Curve(0),\Curve(1)\in\Z^2$. Applying the projection $\R^2\to\Tor$, we get a curve $\Cproj:[0,1]\to \TPV$ with $\Curve(0)=\Curve(1)=0$, which therefore represents an element of $\pi_1(\TPV,0)$, also denoted $\Cproj$. We will be interested in the element $\psiPV(\Cproj)\in\DAHA_\N$.

We have a homomorphism $\BraidGroup_3\to\SL_2(\Z)$ defined by~\eqref{eq:taupm_to_SL2}. Let $\BraidGroupP_3\subset\BraidGroup_3$ be the monoid generated by $\taup,\taum$. 
In the following lemma, $g\in\BraidGroupP_3$ is simultaneously treated as an automorphism of $\DAHA_\N$ as well as a linear map $\R^2\to\R^2$ (induced by the image of $g$ in $\SL_2(\Z)$ which we also denote by $g$). 
\begin{lemma}\label{lemma:SL2_action_curves}
Suppose that $g\in\BraidGroupP_3$, and let $\Curve:[0,1]\to\RPV$ be a curve with $\Curve(0),\Curve(1)\in\Z^2$. Let $\PVp:=g(\PV)\in\R_{\geq0}^2\setminus\{0\}$ and $\Curve':=g\circ\Curve:[0,1]\to\RPVp$. Then the action of $g$ on $\DAHA_\N$ sends $\psiPV(\Cproj)$ to $\psiPVp(\Cprojp)$.
\end{lemma}
\begin{proof}
  The linear map $g:\R^2\to\R^2$ restricts to a homeomorphism $\RPV\xrasim\RPVp$ and to a bijection $\Z^2\xrasim \Z^2$. Therefore composition with $g$ gives rise to a group isomorphism $g_\ast: \pi_1(\TPV,0)\xrasim \pi_1(\TPVp,0)$. Our goal is to show that the group homomorphism $\psiPV$ coincides with $\psiPVp \circ g_\ast$.

  It suffices to consider the cases $g=\taup,\taum$, and only to check the equality 
$\psiPV(Z)=\psiPVp(g_\ast(Z))$
 when $Z=\XPV,\YPV$ is a generator of $\pi_1(\TPV,0)$. This is done in \cref{fig:puncture-XY}.
\end{proof}

\begin{figure}
  \includegraphics[width=0.8\textwidth]{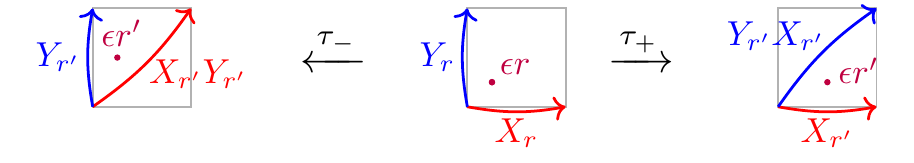}
  \caption{\label{fig:puncture-XY}The action of $\taup,\taum$ on $\pi_1(\TPV,0)$ agrees with the action~\eqref{eq:taup_dfn}--\eqref{eq:taum_dfn} on $\DASHA_\N$. See \cref{lemma:SL2_action_curves}.}
\end{figure}

Let $\Curve=[\Cseg_1\Cseg_2\cdots\Cseg_k]$ be a curve from $(0,0)$ to $(\rectXY)\in(\Ztp)^2$. Choosing $\PV:=(1,0)$ and $\eps>0$ small enough, we see that 
 whenever $\Curve$ passes through a lattice point $\point\neq(\rectXY)$, it passes above the point $\point+\eps\PV$. Thus, we find that the element $\DCn$ constructed in \cref{dfn:intro:DCn} is given by
\begin{equation*}%
  \DCn=\gam \S  \cdot \psiPV(\Csegproj_1)\cdot \gam \S  \cdot \psiPV(\Csegproj_2)\cdots \gam \S \cdot \psiPV(\Csegproj_k) \cdot \S .
\end{equation*}
Since the maps $\taup,\taum$ send monotone curves to monotone curves and $\S$ to $\S$, \cref{lemma:SL2_action_curves} implies the following result.
\begin{corollary}\label{cor:g_DC=D_gC}
  If $g\in\BraidGroupP_3$, then for any curve $\Curve$, we have
  \begin{equation*}%
    g (\DCn)= \DXn_{g(\Curve)}.
  \end{equation*}
\end{corollary}

\subsection{Almost linear elements}\label{sec:almost_linear_elts}
For an almost linear curve $\Curve$ from $(0,0)$ to $(\rectXY)$ (cf. \cref{dfn:intro_AL}), denote 
\begin{equation*}%
  \DASXn_{\rectXY}:=\DXn_{\Curve}.
\end{equation*}
A natural question arises: how is $\DASXn_{\rectXY}$ expressed in the generators $\Pgen_\bx$ of $\DASHA_\N$ defined in~\eqref{eq:SL2_DAHA_Pgen}? The goal of this subsection is to prove the following result.
\begin{theorem}\label{thm:indep_on_n}
  For all $\rectXY,\N\geq1$, the element $\DASXn_{\rectXY}$ is expressed as a $\Cqt$-linear combination of products of generators $\Pgen_\bx$ with coefficients that do not depend on $\N$.
\end{theorem}
The above coefficients are expressed in the language of symmetric functions; see~\cite{EC2} for background and notation.

First, we would like to understand the relationship between the elements $\DASXn_{\rectXY}\in\DASHA_\N$ and the $\SL_2(\Z)$-action. For $d\geq1$, define
\begin{equation}\label{eq:DASXn_d0}
  \DASXn_{d,0}:=\gam \S \cdot \Y \X^{d}\Y^{-1}\cdot \S .
\end{equation}
\begin{lemma}
  Let $\rectXY\geq1$ and $d:=\gcd(\rectXY)$.
  Then there exists $g\in\BraidGroupP_3$ such that the corresponding element of $\SL_2(\Z)$ sends $(d,0)$ to $(\rectXY)$. For any such $g$, 
 we have %
\begin{equation}\label{eq:SL2_almost_linear}
  g \DASXn_{d,0} = \DASXn_{\rectXY}.
\end{equation}
\end{lemma}
\begin{proof}

  First, it is well known (Euclid's algorithm) that the monoid $\BraidGroupP_3$ contains an element $g$ sending $(d,0)$ to $(\rectXY)$. Let $g\in\BraidGroupP_3$ be any such element. 

Let $\Curve_{d,0}$ be a curve  from $(0,0)$ to $(d,0)$ proceeding monotonously from left to right slightly above the $x$ axis. We extend \cref{dfn:intro:DCn} to $\Curve_{d,0}$, with the lattice path $\PathX_{\Curve_{d,0}}$ consisting of $d$ right steps. The resulting element $\DXn_{\Curve_{d,0}}$ coincides with the element $\DASXn_{d,0}$ given by~\eqref{eq:DASXn_d0}. The result now follows from \cref{cor:g_DC=D_gC}.
\end{proof}

Let $\XSym_\N := x_1+\cdots+x_\N$.
 We start by computing the action of $\DASXn_{d,0}$ on $1\in\Laqt$.
\begin{proposition}\label{prop:Dd0}
For $d,\N\geq1$, we have %
\begin{equation}\label{eq:DASXn_vs_h_d}
 \DASXn_{d,0}\cdot 1=(-t)^d e_d\left[(1-\ti)\XSym_\N\right].
\end{equation}
\end{proposition}
Here $e_d\in\Laqt$ denotes the $d$-th elementary symmetric function, and the square brackets denote plethystic substitution; see \cref{sec:relations}. 
\begin{proof}
We have 
\begin{equation*}%
   \DASXn_{d,0}\cdot 1 
= \gam \S \Y \X^d \Y^{-1} \S\cdot 1
= \gam \S \Y \X^d \Y^{-1} \cdot 1 
= \gam \S \Y \X^d \cdot 1 
= \gam \S \Y \cdot q^{-d} x_1^d.
\end{equation*}
Applying~\eqref{eq:Y_act} and then~\eqref{eq:ST=TS=S}, we find
\begin{equation*}%
     \DASXn_{d,0}\cdot 1 
= \gam \S t^{(\N-1)/2}T_1\cdots T_{\N-1} \cdot x_d^d
= \gam \S \cdot x_d^d.
\end{equation*}

By~\eqref{eq:SX_k_dfn}, we have
\begin{equation*}%
  \S=\SXX_1^\N = \frac{1-t^{-1}}{1-t^{-\N}} \left(\sum_{i=1}^\N t^{-\frac{N-i}2}T_i\cdots T_{\N-1}   \right) \SXX_1^{\N-1}.
\end{equation*}
Observe that $\gam\frac{1-t^{-1}}{1-t^{-\N}}=(1-t)t^{\N-1}$. Thus, $\gam \S = A_N  \SXX_1^{\N-1}$, where we set
\begin{equation*}%
  A_M:=(1-t)t^{M-1} \left(\sum_{i=1}^M t^{-\frac{M-i}2}T_i\cdots T_{M-1}   \right) \quad\text{for $1\leq M\leq \N$.}
\end{equation*}
(For $i=M$, the term $t^{-\frac{M-i}2}T_i\cdots T_{M-1}$ is equal to $1$.) Thus, $A_\N = \tt A_{\N-1} T_{\N-1} + (1-t)t^{\N-1}$. Since $\SXX_1^{\N-1}\cdot x_\N=x_\N$, we get
\begin{equation}\label{eq:gam_SXX}
  \gam \SXX_1^\N\cdot x_\N^d = A_\N \SXX_1^{\N-1}\cdot x_\N^d = A_\N\cdot x_\N^d = \tt A_{\N-1} T_{\N-1}\cdot x_\N^d + (1-t)t^{\N-1}x_\N^d.
\end{equation}
We now prove by induction on $M=1,2,\dots,\N$ that 
\begin{equation}\label{eq:AM_cdot_xMd}
  A_M\cdot x_M^d = (-t)^d e_d\left[\tXSym_M\right],
\end{equation}
where $\tXSym_M := \XSym_M(1-\ti)$. The desired result follows from~\eqref{eq:AM_cdot_xMd} with $M=\N$.

 Denote $\tx_i:=x_i(1-t^{-1})$. By a simple computation as in~\cite[Section~3.8]{Bergeron}, we see that 
\begin{equation}\label{eq:plethysm_1_var}
  (-t)^de_d[\tx_i]=(1-t)x_i^d.
\end{equation}
 From here, the base case follows since for $M=1$, both sides of~\eqref{eq:AM_cdot_xMd} become equal to $(1-t)x_1^d$.

For the induction step, observe that 
\begin{equation}\label{eq:T_N-1_x_N}
  T_{\N-1} x_\N^d = \tti x_{\N-1}^d + (\tti-\tt) \left(x_\N^d+\sum_{i=1}^{d-1}x_{\N-1}^ix_\N^{d-i} \right).
\end{equation}
Each $T_i$, $1\leq i< \N-1$ treats $x_\N$ as a constant (although note that $T_i\cdot P=\tti P$ for any polynomial $P$ such that $s_i\cdot P=P$), and thus $A_{\N-1}x_\N^d=(1-t^{\N-1})x_\N^d$. Combining this with~\eqref{eq:gam_SXX} and~\eqref{eq:T_N-1_x_N}, we get
\begin{equation*}%
  A_\N\cdot x_\N^d = (-t)^d e_d \left[\tXSym_{\N-1}\right] + (1-t) x_\N^d + \sum_{i=1}^{d-1} (-t)^i e_i \left[\tXSym_{\N-1}\right]  (1-t) x_\N^{d-i}.
\end{equation*}
By~\eqref{eq:plethysm_1_var}, we have $(1-t)x_\N^k=(-t)^k e_k \left[\tx_\N\right]$ for all $k\geq1$. Substituting, we find
\begin{equation*}%
  A_\N\cdot x_\N^d = (-t)^d \left( e_d \left[\tXSym_{\N-1}\right] + e_d \left[\tx_\N\right] + \sum_{i=1}^{d-1}  e_i \left[\tXSym_{\N-1}\right]  e_{d-i}\left[\tx_\N\right]\right) = (-t)^d e_d\left[\tXSym_\N\right]
\end{equation*}
using $e_d[Y+Z] = \sum_{i=0}^d e_{d-i}[Y] e_i[Z]$; see~\cite[Equation~(1.64)]{HagBook}.
\end{proof} 

Recall that for a partition $\la=(\la_1,\la_2,\dots,\la_r)$, we have the power sum symmetric function $p_\la:=p_{\la_1}p_{\la_2}\cdots p_{\la_r}$. Let $|\la|$ be the number of boxes in $\la$, $\ell(\la)$ be the number of nonzero parts of $\la$, and $m_i$ be the number of parts of $\la$ equal to $i$. We set $\eps_\la:=(-1)^{m_2+m_4+\cdots}=(-1)^{|\la|-\ell(\la)}$  and $z_\la:=1^{m_1}m_1!2^{m_2}m_2!\cdots$ as in~\cite[Section~7.7]{EC2}. We now express the right-hand side of~\eqref{eq:DASXn_vs_h_d} in terms of power sum symmetric functions.
\begin{proposition}[{\cite[Proposition~7.7.6]{EC2}}]\label{prop:h_d_p_la_EC2}
For all $d\geq1$, 
\begin{equation*}%
  e_d[\XSym(1-\ti)]= \sum_{\la\vdash d} p_\la \eps_\la z_\la^{-1}\prod_i (1-t^{-\la_i}).
\end{equation*}
\end{proposition}

Finally, for a partition $\la=(\la_1,\la_2,\dots,\la_r)$ of $d$, write $\Pgen_{d,0;\la}:=\Pgen_{\la_1,0}\Pgen_{\la_2,0}\cdots \Pgen_{\la_r,0}$. (The elements $\Pgen_{k,0}$ pairwise commute for different $k$.)
\begin{proposition}\label{prop:DASXn_vs_Pgen}
  For all $d,\N\geq1$, the element $\DASXn_{d,0}$ is expressed in $(\Pgen_{k,0})_{k\geq0}$ with the same coefficients as 
 $(-t)^d e_d\left[\XSym(1-\ti)\right]$ 
is expressed in $(p_k)_{k\geq0}$.
  Explicitly, we have %
\begin{equation}\label{eq:DASXn_vs_Pgen}
  \DASXn_{d,0}=\sum_{\la\vdash d} \expcoef^{d}_{\la}\Pgen_{d,0;\la}, \quad\text{where}\quad 
\expcoef^{d}_{\la}=(-t)^d \eps_\la z_\la^{-1}\prod_i (1-t^{-\la_i}).
\end{equation}
\end{proposition}
\begin{proof}
We have~\cite[Equation~(2.10)]{SV11}
\begin{equation}\label{eq:YXYi}%
  \Y\X\Y^{-1}=qT_1\cdots T_{\N-2}T_{\N-1}^2T_{\N-2}\cdots T_1 \X.
\end{equation}
Thus, $\DASXn_{d,0}=q^d\gam \S D \S$, where $D$ is a product of $T_i$-s and $\X$-s.  It is clear from~\eqref{eq:T_act} that for $f\in\C[x_1,\dots,x_\N]^{\Sn}$ and $g\in\C[x_1,\dots,x_\N]$, $T_i\cdot (fg)=fT_i\cdot g$. Since 
$X_j\cdot (fg)=\qi x_jfg=fX_j\cdot g$, we see that for any operator $D$ which is a product of $T_i$-s and $X_j$-s, $\S D\S$ acts on $\C[x_1,\dots,x_\N]^{\Sn}$ as the operator of multiplication by $\S D\S\cdot 1$. It follows that $\DASXn_{d,0}$ acts by multiplication by $(-t)^d e_d\left[(1-\ti)\XSym_\N\right]$ while $\Pgen_{k,0}$ defined in~\eqref{eq:Pgen_dfn} acts by multiplication by $p_k(x_1,\dots,x_\N)$. It remains to note that the polynomial representation of $\DASHAp_\N$ on $\C[x_1,\dots,x_\N]^{\Sn}$ is faithful; cf.~\cite[Section~1.4]{SV13}.
\end{proof}

\begin{example}
Let $\rectX=\rectY=1$. We find %
\begin{equation}\label{eq:ex_DASXn_11}
  (-t)^1e_1\left[\XSym_\N(1-\ti)\right]=(1-t)p_1, \quad\text{and thus}\quad 
\DASXn_{1,1} = (1-t) \Pgen_{1}.
\end{equation}
Next, let $\rectX=\rectY=2$. Using \cref{prop:h_d_p_la_EC2}, we find
\begin{equation*}%
  (-t)^2e_2\left[\XSym_\N(1-\ti)\right]=\frac12(1-t^2)p_2+\frac12 (1-t)^2 p_{11}.
\end{equation*}
Thus, \cref{prop:DASXn_vs_Pgen} yields
\begin{equation}\label{eq:ex_DASXn_22}
  \DASXn_{2,2} = \frac12(1-t^2)\Pgen_{2,2}+\frac12 (1-t)^2 (\Pgen_{1,1})^2.
\end{equation}
\end{example}

\begin{proof}[Proof of \cref{thm:indep_on_n}]
 By~\eqref{eq:SL2_DAHA_Pgen} and~\eqref{eq:SL2_almost_linear}, for $\rectXY,\N\geq1$ with $d:=\gcd(\rectXY)$, we have
\begin{equation}\label{eq:DASXn_vs_Pgen_rectXY}
    \DASXn_{\rectXY}=\sum_{\la\vdash d} \expcoef^{d}_{\la}\Pgen_{\rectXY;\la}.
\end{equation}
Here, the coefficients $\expcoef^d_\la$ are given in~\eqref{eq:DASXn_vs_Pgen}, and for $\la=(\la_1,\la_2,\dots,\la_r)$, we set $\Pgen_{\rectXY;\la}:=\Pgen_{\la_1\rectXp,\la_1\rectYp}\Pgen_{\la_2\rectXp,\la_2\rectYp}\cdots \Pgen_{\la_r\rectXp,\la_r\rectYp}$, where $(\rectXYp):=\frac1d(\rectXY)$.%
\end{proof}

\section{From DAHA to EHA}\label{sec:EHA}
In this section, we review the background on the \emph{elliptic Hall algebra (EHA)} $\EHA$, introduced in~\cite{BuSc}. We then explain how the elements $\DCn\in\DASHAp_\N$ give rise to elements of $\EHAp$ in the limit $\N\to\infty$.

\subsection{Generators and relations}
For a vector $\bx=(\rectXY)\in\Z^2$, we write $\dop(\bx):=\gcd(\rectXY)$, and we call $\bx$ \emph{primitive} if $\dop(\bx)=1$. By definition, $\EHA$ is an algebra over $\Cqt$ generated by elements $\ehagen_\bx$, $\bx\in\Z^2$, subject to the following relations (see~\cite[Section~1.5]{SV11} and~\cite[Section~2.1]{MS21}).
\begin{itemize}
\item If $\bx,\bx'$ belong to the same line in $\Z^2$ then $[\ehagen_\bx,\ehagen_{\bx'}]=0$.
\item Assume that $\bx$ is primitive and that the triangle with vertices $0,\bx,\bx+\by$ has no interior lattice points. Then
  \begin{equation}\label{eq:EHA_comm}
    [\ehagen_\by,\ehagen_\bx]=\epsilon_{\bx,\by} \frac{\theta_{\bx+\by}}{\alpha_1},
  \end{equation}
  where $\epsilon_{\bx,\by}:=\sign(\det(\bx\,\by))$ and $\theta_{\bz},\bz\in\Z^2$ is defined by
  \begin{equation*}%
    \sum_{i=0}^\infty \theta_{i\bz_0} z^i = \exp \left(\sum_{i=1}^\infty \alpha_i\ehagen_{i\bz_0}z^i\right).
  \end{equation*}
\noindent
  Here, $\bz_0\in\Z^2$ is primitive and
  \begin{equation}\label{eq:EHA_alpha_dfn}
\alpha_i=\frac1i (1-q^{-i})(1-t^{-i})(1-(qt)^{i}).
  \end{equation}
\end{itemize}
It follows directly from the above relations that $\SL_2(\Z)$ acts on $\EHA$ by $g \cdot \ehagen_\bx = \ehagen_{g \cdot \bx}$. 
\begin{example}
  When $\bx+\by$ is primitive, we have $\theta_{\bx+\by}=\alpha_1$ and~\eqref{eq:EHA_comm} becomes
\begin{equation}\label{eq:EHA_comm_prim}
      [\ehagen_\by,\ehagen_\bx]=\epsilon_{\bx,\by} \ehagen_{\bx+\by}.
\end{equation}
On the other hand, letting $\bx=(1,0)$ and $\by=(1,2)$, we have
\begin{equation}\label{eq:EHA_comm_22}
  \ehagen_{2,2}=\frac{2qt}{(1+q)(1+t)(1+qt)} [\ehagen_{1,2},\ehagen_{1,0}] - \frac{(1-q)(1-t)(1-qt)}{(1+q)(1+t)(1+qt)} \ehagen_{1,1}^2.
\end{equation}
\end{example}

\subsection{Taking the limit}
Define the \emph{positive parts} $\DASHAp_\N\subset\DASHA_\N$ (resp., $\EHAp\subset\EHA$) to be the subalgebras generated by the elements $\Pgen_\bx$ (resp., $\ehagen_\bx$) for $\bx\in\BZxy$:
\begin{equation}\label{eq:DASHAp_EHAp_dfn}
  \DASHAp_\N:=\<\Pgen_\bx\mid \bx\in\BZxy\> \quad\text{and}\quad \EHAp:=\<\ehagen_\bx\mid \bx\in\BZp\>.
\end{equation}

By~\cite[Proposition~4.1]{SV11}, we have a surjective algebra map $\DASHAp_\N\to\DASHAp_{\N-1}$, $\Pgen_\bx\mapsto\Pgenm_\bx$ for all $\bx\in\BZp$. Let $\DASHApinf$ be the corresponding projective limit, generated by $\{\Pgenf_\bx\mid\bx\in\BZp\}$.

\begin{theorem}[{\cite[Theorem~4.6]{SV11}}]
  The assignment
  \begin{equation}\label{eq:eha_daha_iso}
    \ehagen_\bx\mapsto \frac1{q^{\dop(\bx)}-1} \Pgenf_\bx
  \end{equation}
  induces an isomorphism $\EHAp\xrasim \DASHApinf$.
\end{theorem}
In view of~\eqref{eq:eha_daha_iso}, we denote%
\begin{equation}\label{eq:Peha_dfn}
  \Peha_{\bx}:=(q^{\dop(\bx)}-1)\ehagen_\bx.%
\end{equation}

According to \cite[Theorem~3.1]{SV11}, the isomorphism~\eqref{eq:eha_daha_iso} is the restriction of a $\Z^2$-graded $\SL_2(\Z)$-equivariant surjective algebra homomorphism $\EHA\to \DASHAinf$. In particular, the isomorphism~\eqref{eq:eha_daha_iso} commutes with the action of $\BraidGroupP_3$.
Applying \cref{thm:indep_on_n}, we obtain the following.
\begin{corollary}\label{cor:limit_almost_linear}
For each $\rectXY\geq1$, the sequence $\DASXn_{\rectXY}$, $\N=1,2,\dots$, gives rise to a well-defined limit $\DASXnf_{\rectXY}\in \SHpf$. 
\end{corollary}
\noindent We continue to denote by $\ucbas_{\rectXY}\in\EHAp$ the image of $\DASXnf_{\rectXY}$ under the identification $\SHpf\cong\EHAp$.

\subsection{Action on symmetric functions}\label{sec:EHA_action}
We follow the exposition of~\cite{SV13}; cf. \cref{rmk:polyrep_rescaling}. Consider the automorphism $\theta$ of $\DASHAp_\N$ given by
\begin{equation*}%
  \theta(\Pgen_{0,d})=\Pgen_{0,d} - \frac{1-t^{d\N}}{1-t^d}, \quad d\in\Ztp;\qquad 
  \theta(\Pgen_{i,j})=\Pgen_{i,j}, \quad (i,j)\in\BZx.
\end{equation*}
Then $\tphi_\N:=\phi_\N\circ \theta$ is a representation of $\DASHAp_\N$ on $\LaqtN$, where $\phi_\N$ is defined in \cref{sec:DAHA_polyrep}. Letting $\rho_\N:\Laqt\to\LaqtN$ be the operator setting $x_{\N+1}=x_{\N+2}=\cdots=0$, we get $\tphi_{\N-1}(\Pgenm_{i,j})\circ \rho_{\N-1} = \rho_{\N-1}\circ \tphi_\N(\Pgen_{i,j})\circ \rho_\N$ for $(i,j)\in\BZp$. Thus, we may take the limit of $\tphi_\N$ as $\N\to\infty$. The resulting representation $\tphi_\infty$ of $\SHpf\cong\EHAp$ on $\Laqt$ may be described explicitly in terms of the \emph{Macdonald polynomials} $P_\la(q,t)\in\Laqt$: for $d\in\Ztp$, we have
\begin{equation}\label{eq:EHA_polyrep}
  \tphi_\infty(u_{0,d}) \cdot P_\la(q,t^{-1}) = \left(\sum_i \frac{q^{d\la_i}-1}{q^d-1} t^{d(i-1)} \right) P_\la(q,t^{-1}),
\end{equation}
and $\tphi_\infty(u_{1,0})$ is the operator of multiplication by $\frac1{q-1} p_1$; see~\cite[Corollary~1.5]{SV13}.

\begin{remark}\label{rmk:q_t_symm}
  Observe from~\eqref{eq:EHA_alpha_dfn} that the parameters $q$, $t$, $(qt)^{-1}$ enter symmetrically into the definition of $\EHA$. Even though it may appear that~\eqref{eq:EHA_polyrep} breaks the $q,t$-symmetry, note that the plethysm $\left[\frac{\XSym}{1-t}\right]$ from~\eqref{eq:intro_FC_dfn} turns Macdonald polynomials $P_\la$ into modified Macdonald polynomials $\Ht_\la$. Using the well-known $q,t$-symmetry $\Ht_\la(q,t)=\Ht_{\la'}(t,q)$ (where $\la'$ denotes the conjugate partition), one may deduce 
 that the symmetric function $\FC$ is $q,t$-symmetric.\footnote{An equivalent description of $\EHAp$ is given in~\cite{BGLX} (cf. \cref{ssec:BGLX}), where the formulas for the action on $\Laqt$ are manifestly $(q,t)$-symmetric.} Alternatively, $q,t$-symmetry of $\FC$ may be seen directly from~\eqref{eq:nabla_H_tilde}. See also~\cite[Section~6.2]{GN}.
\end{remark}

\section{Skein relation}\label{sec:skein_relation}
The goal of this section is to prove \cref{thm:intro:skein}. We will give an algebraic proof based on the diagrammatic formalism of Morton--Samuelson~\cite{MS21}; see \cref{sec:skein} and \cref{rmk:skein_diagrammatic}. Along the way, we will also prove \cref{thm:intro_limit_exists}. We start with a version of \cref{thm:intro:skein} for finite $\N$. 
\begin{theorem}
  Consider curves $\Curve_+,\Curve_-,\Curve_0$ which pass above, below, and through some lattice point $\point$ as in \cref{thm:intro:skein}. Then
  \begin{equation}\label{eq:skein_\N}
    \DXn_{\Curve_+}=qt\DXn_{\Curve_-}+\DXn_{\Curve_0}.
  \end{equation}
\end{theorem}
\begin{proof}%
 Consider the elements $\DXn_{\Curve_+},\DXn_{\Curve_-},\DXn_{\Curve_0}\in \DASHA_\N$ associated to $\Curve_+,\Curve_-,\Curve_0$ via \cref{dfn:intro:DCn}. The paths $\PathX_{\Curve_+}$, $\PathX_{\Curve_-}$, $\PathX_{\Curve_0}$ differ locally so that $\PathX_{\Curve_+}=\cdots UR\cdots$, $\PathX_{\Curve_-}=\cdots RU \cdots$, $\PathX_{\Curve_0}=\cdots U\ast R\cdots$, where $U$ and $R$ indicate up and right steps, respectively, and $\ast$ indicates that $\PathX_{\Curve_0}$ passes through the lattice point $\point$. In particular, the elements $\DXn_{\Curve_+},\DXn_{\Curve_-},\DXn_{\Curve_0}$ are given by 
  \begin{equation*}%
    \DXn_{\Curve_+}=A \Y\Y\X\Y^{-1}B,\quad     \DXn_{\Curve_-}=A \Y\X\Y^{-1}\Y B, \quad\text{and}\quad     \DXn_{\Curve_0}=A \Y \cdot \gam \S\cdot \Y\X\Y^{-1}B
  \end{equation*}
  for some $A,B\in\DASHA_\N$.
According to \cref{dfn:intro:DCn}, $A=\gam \S A'$ for some $A'\in\DASHA_\N$. We have $\S=\SX_1=\SX_1\SX_2$, and $\SX_2$ commutes with $\X$, $\Y$, and $\S$. Thus, we have $A\Y=A\Y\SX_2$. Applying~\eqref{eq:YXYi}, we find
\begin{equation*}%
  \DXn_{\Curve_+}=A \Y\Y\X\Y^{-1}B= A\Y\cdot \SX_2 \cdot  q T_1T_2\cdots T_{\N-1}  T_{\N-1}\cdots T_2T_1  \X\cdot B.
\end{equation*}
Applying~\eqref{eq:symmetrizer_recurrence} for $k=1$, we get
\begin{equation*}%
  \DXn_{\Curve_+}=qA\Y\cdot \left(-t(1-t^{-\N})\S+t \SX_2\right)  \X\cdot B  
=qt^{-\N+1}\gam A\Y\cdot \S\X\cdot B + qtA\Y \cdot \SX_2 \X\cdot B.
\end{equation*}
Since $A\Y=A\Y\SX_2$, the second term on the right-hand side becomes $qtA\Y \X B=qt\DXn_{\Curve_-}$. We claim that the first term equals $\DXn_{\Curve_0}$. Indeed, by~\eqref{eq:YXYi} and~\eqref{eq:ST=TS=S}, we get 
\begin{align*}%
\DXn_{\Curve_0}&=A \Y \cdot \gam \S\cdot \Y\X\Y^{-1}B 
 = A \Y \cdot \gam \S\cdot q T_1T_2\cdots T_{\N-1} T_{\N-1}\cdots T_2T_1  \X \cdot B \\
 &= qt^{-\N+1}\gam A \Y \cdot  \S \X \cdot B.
\end{align*}
This proves~\eqref{eq:skein_\N}. See \cref{fig:skein-diag} for a diagrammatic interpretation of the above argument.
\end{proof}

\begin{proof}[Proof of \cref{thm:intro_limit_exists,thm:intro:skein}]
As explained in \cref{rmk:intro:skein}, each element $\DXn_{\Curve}$ may be expressed via~\eqref{eq:skein_\N} in terms of products of almost linear elements $\DASXn_{\rectXYp}$ with coefficients not depending on $\N$. By \cref{cor:limit_almost_linear}, for each curve $\Curve$, we get a well-defined limit $\DC\in \SHpf$ of $\DXn_{\Curve}$ as $N\to\infty$. This shows \cref{thm:intro_limit_exists}. 
 The resulting limiting elements satisfy~\eqref{eq:intro:skein} by construction, which shows \cref{thm:intro:skein}.
\end{proof}

\begin{example}\label{ex:DC_and_Sym}
We calculate the elements $\DXlimEHA_{\Curve_+},\DXlimEHA_{\Curve_-},\DXlimEHA_{\Curve_0}$ for the three curves in \cref{fig:2x2}. By~\eqref{eq:intro:prod}, \eqref{eq:ex_DASXn_11}, and~\eqref{eq:Peha_dfn}, we have
\begin{equation*}%
  \ucbas_{\Curve_0}=\ucbas_{1,1}^2=(1-t)^2(q-1)^2\ehagen_{1,1}^2.
\end{equation*}
Next, by~\eqref{eq:ex_DASXn_22} and~\eqref{eq:Peha_dfn}, we find
\begin{equation*}%
  \ucbas_{\Curve_+}=\ucbas_{2,2}=\frac12(1-t^2)(q^2-1)\ehagen_{2,2}+\frac12 (1-t)^2(q-1)^2 \ehagen_{1,1}^2.
\end{equation*}
Therefore, by~\eqref{eq:intro:skein},
\begin{equation*}%
  \ucbas_{\Curve_-}=\frac1{2qt}(1-t^2)(q^2-1)\ehagen_{2,2}-\frac1{2qt} (1-t)^2(q-1)^2 \ehagen_{1,1}^2.
\end{equation*}
We may now compute the symmetric functions $\FX_{\Curve_+},\FX_{\Curve_-},\FX_{\Curve_0}$. One way to do that is by using the polynomial representation of the DAHA, i.e., by combining \cref{ex:intro_DXn} with \cref{sec:DAHA_polyrep}. Alternatively, we can use commutation relations, such as~\eqref{eq:EHA_comm_22}, together with \cref{sec:EHA_action}. Both computations lead to the following result:
\begin{align}%
\label{eq:DCp.1}
  \ucbas_{\Curve_+}\cdot 1 &= (1-t) \left( (1- t^2 -q t) s_{11} + q s_{2} \right);\\
\label{eq:DCm.1}
  \ucbas_{\Curve_-}\cdot 1 &= (1-t) \left( -ts_{11}+s_2 \right);\\
\label{eq:DC0.1}
  \ucbas_{\Curve_0}\cdot 1 &= (1-t)^2 \left( (1+t-q t) s_{11} + q s_{2} \right).
\end{align}
Applying the plethysm~\eqref{eq:intro_FC_dfn} yields~\eqref{eq:intro_ex_FC}, and applying~\eqref{eq:intro_PEHAC_dfn} yields~\eqref{eq:intro_ex_PEHAC}.
\end{example}

For a $\Z$-convex curve $\Curve$, \cref{rmk:intro:skein} may be used to obtain a non-recursive formula expressing $\DXlimEHA_{\Curve}$ in terms of \PAL convex curves.
\begin{corollary}\label{cor:non_recursive}
Let $\Curve$ be a $\Z$-convex curve passing through no lattice points other than $(0,0)$ and $(\rectXY)$. Then
\begin{equation*}%
  \DXlimEHA_{\Curve}=\sum_{\substack{\Curve'=[\Csegp_1\Csegp_2\cdots\Csegp_k] \\ \text{$\Z$-convex \PAL}}} \frac{(-1)^{k-1}}{(qt)^{k-1+a(\Curve,\Curve')}} \cdot  \DXlimEHA_{\Curve'_1}\cdot \DXlimEHA_{\Curve'_2}\cdots \DXlimEHA_{\Curve'_k},
\end{equation*}
where the summation is over all $\Z$-convex \PAL curves $\Curve'$ weakly above $\Curve$, $k=\nseg(\Curve')$ is arbitrary, and $a(\Curve,\Curve')$ is the number of lattice points strictly between $\Curve$ and~$\Curve'$. 
\end{corollary}
\begin{proof}
By~\eqref{eq:intro:prod}, we have $\DXlimEHA_{\Curve'}=\DXlimEHA_{\Curve'_1}\cdot \DXlimEHA_{\Curve'_2}\cdots \DXlimEHA_{\Curve'_k}$. The result is obtained by induction on the number of lattice points between $\Curve$ and the almost linear curve from $(0,0)$ to $(\rectXY)$. For the base case, when $\Curve$ is itself almost linear, the result is vacuously true. The induction step consists of applying~\eqref{eq:skein_\N}.
\end{proof}

We now translate \cref{prop:DASXn_vs_Pgen} via~\eqref{eq:eha_daha_iso} to express the almost linear elements $\DASXlimEHA_{d\bx}$ in terms of the renormalized generators $\Peha_\bx\in\EHA$ introduced in~\eqref{eq:Peha_dfn}.
\begin{corollary}\label{cor:DASXn_vs_Pgen_EHA}
For all $d\geq1$ and primitive $\bx\in\BZp$, the element $\DASXlimEHA_{d\bx}$ is expressed in $(\Peha_{k\bx})_{k\geq0}$ with the same coefficients as 
$(-t)^d e_d\left[\XSym(1-\ti)\right]$
 is expressed in $(p_k)_{k\geq0}$.
\end{corollary}

\section{Skein algebras and modules}\label{sec:skein}
We review the background on the skein-theoretic approach of~\cite{MS17,MS21} and use it to prove \cref{thm:intro_curve_to_DASHA}.

\subsection{Skein of a surface}\label{sec:skein_of_a_surface}

Let $\Surf$ be a surface. A \emph{framed (oriented) link} $L$ is an embedding of a disjoint union of several ribbons $S^1\times\Interval$ into $\Surf\times\Interval$, where each component of $L$ is the image of $S^1\times\{1/2\}$. If an unframed oriented link $L$ is drawn in the plane, we often equip it with the \emph{blackboard framing}, so that each ribbon is obtained by taking a thin neighborhood of the corresponding component of $L$ in the plane.

\begin{definition}
  The \emph{skein algebra} $\Sk(\Surf)$ is the algebra of $\Cvq$-linear combinations of isotopy classes of framed oriented links inside $\Surf\times \Interval$ subject to the relations
  \begin{equation}\label{eq:Sk_dfn}
    \LPlusfr-\LMinusfr=\left(\qqi-\qq\right) \LZerofr \quad\text{and}\quad \Lcclwfr=v^{-1} \Lstraightfr.
\end{equation}
\end{definition}

Since the relation~\eqref{eq:Sk_dfn} preserves homology classes of curves, we see that $\Sk(\Surf)$ is $H_1(\Surf)$-graded.

We will be interested in the surfaces $\Surf\in\{\Tor-\Disk,\Tor,\Ann,\Disk\}$, where $\Tor=\R^2/\Z^2$ is a torus, $\Ann=(\R/\Z)\times \Interval$ is an annulus, and $\Disk$ is a two-dimensional disk.

\subsubsection{Disk}
The \emph{\FLY polynomial} $\Pcal(L)=\PHOML(a,q)$ of an (oriented, unframed) link $L$ in a disk is defined by the skein recurrence
\begin{equation}\label{eq:FLY_dfn}
  a\Pcal(\LPlus)-a^{-1}\Pcal(\LMinus)=\left(\qqi-\qq\right) \Pcal(\LZero) \quad\text{and}\quad \Pcal(\unkn)=\frac{a-a^{-1}}{\qqi-\qq}.
\end{equation}
Here, $\unkn$ denotes the unknot, and this normalization is chosen so that the value of $\Pcal$ on the empty link is equal to $1$.

In order to relate the variable $v$ in~\eqref{eq:Sk_dfn} to the variable $a$ in~\eqref{eq:FLY_dfn}, we apply~\cite[Theorem~6.1]{MS17}. Given a framed link $L$ in $\Disk$, let $L'$ be a (link diagram of a framed) link isotopic to $L$ such that the framing of $L'$ coincides with the blackboard framing. Then define the \emph{writhe} $\wop(L)\in\Z$  of $L$ to be the number of positive crossings in $L'$ minus the number of negative crossings in $L'$. Applying~\cite[Theorem~6.1]{MS17} and using the fact that~\eqref{eq:FLY_dfn} gives rise to a well-defined link invariant, we find that the identity
\begin{equation}\label{eq:L=HOMFLY}
L|_{v=a^{-1}} = a^{\wop(L)} \PHOML(a,q)
\end{equation}
holds inside $\Sk(\Disk)$, where $L$ is an arbitrary framed link and the right-hand side is considered as a multiple of the empty link. 

\subsubsection{Annulus}\label{sec:skein_annulus}
The \emph{positive part} $\Skp(\Ann)\subset\Sk(\Ann)$ consists of linear combinations of links whose lift to the universal cover $\R\times \Interval$ only intersects vertical grid lines $x=k$, $k\in \Z$ from left to right.

The algebra $\Skp(\Ann)$ is easily seen to be commutative. One may view it as the algebra of linear combinations of closures of $\N$-strand braids, where $\N=0,1,2,\dots$\;.
\begin{theorem}[\cite{Turaev,AistonThesis,AiMo}]\label{thm:Tur}
  The algebra $\Skp(\Ann)$ is isomorphic to the algebra $\Laq$ of symmetric functions over $\Cq$.
\end{theorem}
We describe this isomorphism explicitly in two different ways.
First, for each $\N=0,1,2,\dots$, let
\begin{equation}\label{eq:Skq_Laq_iso_Z_A}
  \PMS_\N:=\frac{\qq-\qqi}{q^{\N/2}-q^{-\N/2}}\sum_{i=0}^{\N-1} A_{i,\N-1-i},
\end{equation}
where $A_{i,\N-1-i}$ is the annular closure of the braid $\sigma_1\cdots\sigma_i\cdot \sigma_{i+1}^{-1}\cdots \sigma_{\N-1}^{-1}$; see the figure above~\cite[Remark~2.4]{MS17}.
\begin{notation}\label{not:sigma_positive}
In our figures, the generators $\sigma_i$ of $\BraidGroup_\N$ are depicted with \emph{positive} crossings $\LPlus$; cf. \cref{not:MS21}. The homomorphism $\BraidGroup_\N\to \Heckeq_\N$ sends $\sigma_i\mapsto T_i^{-1}$. 
\end{notation}
\begin{remark}
The braids $\beta_{i,\N-1-i}:=\sigma_1\cdots\sigma_i\cdot \sigma_{i+1}^{-1}\cdots \sigma_{\N-1}^{-1}$ and $\beta'_{\N-1-i,i}:=\sigma_1^{-1}\cdots\sigma_{\N-1-i}^{-1}\cdot \sigma_{\N-i}\cdots \sigma_{\N-1}$ are conjugate, and thus the element $\PMS_\N$ does not depend on whether one takes $A_{i,\N-1-i}$ to be the annular closure of $\beta_{i,\N-1-i}$ or of $\beta'_{\N-1-i,i}$.
\end{remark}
By definition, the isomorphism of \cref{thm:Tur} sends
\begin{equation}\label{eq:Tur_iso1}
  \PMS_\N\mapsto p_\N \quad\text{for $\N=0,1,\dots$\;.}
\end{equation}
(For $\N=0$, we set $\PMS_0:=1\in\Skp(\Ann)$ and $p_0:=1\in\Laq$.) 

Alternatively, let $\Heckeq_\N$ be the \emph{Hecke algebra} generated by $T_1,\dots,T_{\N-1}$ modulo relations
\begin{equation}\label{eq:Heckeq_dfn}
  T_iT_{i+1}T_i=T_{i+1}T_iT_{i+1} \quad\text{and}\quad (T_i-\qq)(T_i+\qqi)=0.
\end{equation}
Note that the second relation is equivalent to
$T_i^{-1}-T_i=(\qqi-\qq)$; cf.~\eqref{eq:Sk_dfn} and \cref{not:T_i}. Upon setting $q=\ti$, we recover the relation~\eqref{eq:T_i_relns}.

The algebra $\Sk(\Ann)$ is $H_1(\Ann)=\Z$-graded, and the graded piece of $\Skp(\Ann)$ of degree $\rectX\geq0$ is isomorphic (as a vector space over $\Cvq$) to $\Heckeqcomm_\rectX\otimes \Cv$, where $\Heckeqcomm_\rectX:=\Heckeq_\rectX/[\Heckeq_\rectX,\Heckeq_\rectX]$. We thus have an identification
\begin{equation}\label{eq:Skp_Ann_to_Hecke}
  \Skp(\Ann)\cong\bigoplus_{\rectX=0}^\infty \Heckeqcomm_\rectX\otimes \Cv.
\end{equation}

The irreducible representations $V_\la$ of $\Heckeq_\rectX$ are well known~\cite{GePf} to be indexed by Young diagrams $\la$ with $\rectX$ boxes. For instance, the \emph{trivial representation} $V_{(\rectX)}$ (resp., the \emph{sign representation} $V_{(1^\rectX)}$) is the one-dimensional representation where each $T_i$ acts by multiplication by $\qq$ (resp., by $-\qqi$). 

Define a map
\begin{equation}\label{eq:Tr_Laq_dfn}
  \Tr_{\Laq}:\Heckeq_\rectX\to\Laq,\quad x\mapsto \sum_{\la\vdash \rectX} \Tr(x;V_\la) s_\la,
\end{equation}
where $s_\la$ is a Schur function~\cite[Section~7.10]{EC2}. Clearly, this map factors through the quotient map $\Heckeq_\rectX\to \Heckeqcomm_\rectX$.

The map $\Tr_{\Laq}$ composed with~\eqref{eq:Skp_Ann_to_Hecke} gives another description of the isomorphism of \cref{thm:Tur}.
\begin{proposition}[{\cite[Corollary~2.22]{GoWe}}]\label{prop:GoWe_coincide}
The isomorphism~\eqref{eq:Tur_iso1} coincides with the composition of~\eqref{eq:Skp_Ann_to_Hecke} and ~\eqref{eq:Tr_Laq_dfn}.\footnote{Our conventions for $\Tr_{\Laq}$ differ from those of~\cite{GoWe} by an application of the $\omega$ automorphism, which explains the extra $(-1)^{n-1}$ sign in \cite[Corollary~2.22]{GoWe}.}
\end{proposition}

\begin{example}
For $\N=2$, the only two irreducible representations of the Hecke algebra $\Heckeq_2$ are $V_2$ (trivial) and $V_{11}$ (sign). Thus, 
\begin{equation}\label{eq:Tr_Laq(T1)}
  \Tr_{\Laq}(T_1^{r})=(-1)^r q^{-\frac r2} s_{11} + q^{\frac r2} s_2 \quad\text{for all $r\in\Z$.}
\end{equation}
On the other hand,~\eqref{eq:Skq_Laq_iso_Z_A} gives $\PMS_2=\frac1{\qq+\qqi} (T_1+T_1^{-1})$, under the identification~\eqref{eq:Skp_Ann_to_Hecke}. Thus, 
\begin{equation*}%
  \Tr_{\Laq}(\PMS_2)=-s_{11}+s_2 = p_2,
\end{equation*}
in agreement with \cref{prop:GoWe_coincide}.
\end{example}

We can also relate $\Skp(\Ann)$ to $\Skp(\Disk)$ as follows. Observe that any embedding $\Surf\to\Surf'$ of surfaces gives rise to a homomorphism
\begin{equation}\label{eq:Sk_emb_Sk}
  \Sk(\Surf)\to\Sk(\Surf')
\end{equation}
 of the associated skein algebras. In particular, identifying $\Ann$ with $\Disk - \eps\Disk$ for some $0<\eps<1$, we obtain an algebra homomorphism
\begin{equation}\label{eq:Skp_Ann_to_Disk}
  \Skp(\Ann)\to\Sk(\Disk).
\end{equation}

\subsubsection{Torus}
The algebra $\Sk(\Tor)$ is generated by certain elements $\WMSTor_\bx$, $\bx\in\BZast$, defined as follows. (See~\cite{MS17} for a beautiful description of relations between these elements.) When $\bx$ is primitive, we let $\WMSTor_\bx$ correspond to the embedded curve of homology class $\bx$. Suppose now that $\bx=k\bx_0$, where $k\geq1$ and $\bx_0$ is primitive. Consider a primitive curve $\Curve_{\bx_0}$ of homology class $\bx_0$. For $\N\geq1$ and a braid $\beta\in\BraidGroup_\N$, we may consider a \emph{$\beta$-decoration} of $\Curve_{\bx_0}$, which is obtained by taking the image of the annular closure of $\beta$ inside $\Ann$ under the homeomorphism from $\Ann$ to a small neighborhood of $\Curve_{\bx_0}$. We extend this construction to linear combinations of braids, and define $\WMSTor_\bx$ as the $\PMS_k$-decoration of $\Curve_{\bx_0}$.

The algebra $\Sk(\Tor)$ is $H_1(\Tor)=\Z^2$-graded. We let $\vrap:\Sk(\Tor)\xrasim\Sk(\Tor)$ be the graded automorphism multiplying each $(i,j)$-graded component by $v^j$. 

\begin{figure}
\begin{tabular}{cccc}
  \includegraphics[width=0.2\textwidth]{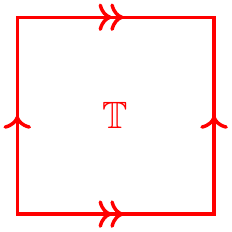}
&
  \includegraphics[width=0.2\textwidth]{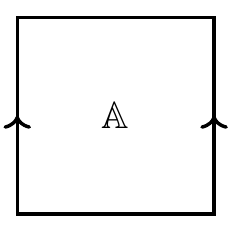}
&
\begin{tikzpicture}[baseline=(Z.base)]
\coordinate(Z) at (0,-1.5);
\node[scale=2](A) at (0,0){$\to$};
\end{tikzpicture}
&
\begin{tikzpicture}[baseline=(Z.base)]
\coordinate(Z) at (0,-1.4);
\node(A) at (0,0){\includegraphics[width=0.3\textwidth]{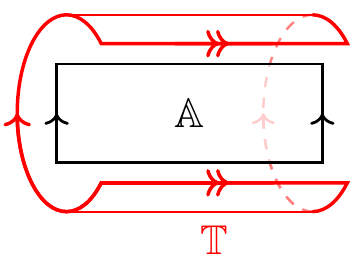}};
\end{tikzpicture}
\end{tabular}
  \caption{\label{fig:wrap} Wrapping the torus around the annulus; see \cref{dfn:Skp_Tor_act_on_Ann}.}
\end{figure}
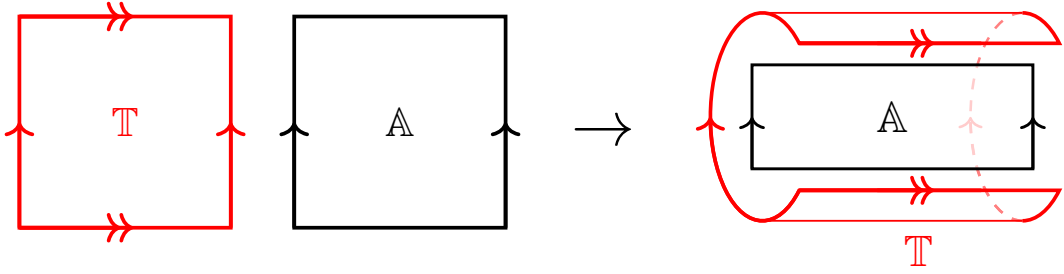

\begin{definition}\label{dfn:Skp_Tor_act_on_Ann}
The action of $\Sk(\Tor)$ on $\Sk(\Ann)$ is obtained by first applying $\vrap$ and then ``wrapping $\Tor$ around $\Ann$'' as shown in \cref{fig:wrap}. More precisely, $\Sk(\Ann)$ consists of framed links inside $\Ann\times\Interval$, and $\Ann\times\Interval$ is homeomorphic to the solid torus $S^1\times \Disk$. Identifying $\Tor$ with the boundary of $S^1\times \Disk$, we obtain the desired action. See \cref{dfn:TC} for a more precise description.
\end{definition}

The \emph{positive part}  $\Skp(\Tor)\subset\Sk(\Tor)$ consists of linear combinations of elements $\WMSTor_\bx$ for $\bx\in\BZp$. In particular, the action of $\Sk(\Tor)$ on $\Sk(\Ann)$ restricts to an action of $\Skp(\Tor)$ on $\Skp(\Ann)$.

\subsubsection{Punctured torus}\label{sec:skein_PTor}

Let $\Tor-\Disk$ be the punctured torus, obtained from $\Tor$ by removing a small disk $\Disk$. We consider the algebra $\Skti(\PTor)$. Let $\bx\in\BZast$, and suppose that $\bx=k\bx_0$ with $\bx_0$ primitive and $k\geq1$. Let $\Curve_{\bx_0}$ be a primitive curve of homology class $\bx_0$ avoiding the puncture $\Disk$. As before, we let $\WMSPTorti_\bx\in\Skti(\Tor-\Disk)$ be obtained by decorating the curve $\Curve_{\bx_0}$ by the element $\PMSti_k$.  

Morton--Samuelson~\cite[Definition~5.6]{MS21} define the positive part $\SkptiMS(\Tor-\Disk)\subset\Skti(\Tor-\Disk)$ to be the subalgebra generated by $\WMSPTorti_\bx$, $\bx\in\BZp$. Recall the algebra $\SkPTor$ defined in~\eqref{eq:intro:SkPTor}.
We let $\SkpPTor\subset\SkPTor$ be the subalgebra generated by $\WC^{\PTor}$ for all monotone curves $\Curve$.
 One can check that $\SkptiMS(\Tor-\Disk)\otimes \C(\qq)\subset\SkpPTor$, but the containment is strict in general.\footnote{We that P.~Samuelson for pointing out this observation to us.}

Applying~\eqref{eq:Sk_emb_Sk} to the inclusion $\PTor\hookrightarrow \Tor$ and substituting $t=q^{-1}$, we get a homomorphism 
\begin{equation}\label{eq:Skp_PTor_to_Tor}
  \SkpPTor\to\Skpx_{\qi}(\Tor).
\end{equation}

\subsection{Braid and tangle skein algebras}
We recall the construction of~\cite[Section~3.1]{MS21} of the braid skein algebra $\BSkn(\Tor,\ast)$ of a torus $\Tor$. Let $\N\geq1$ and $0<\eps<1$. Place $\N$ points $\brpt_i:=\frac\eps{\N+1}(i,i)$, $i=1,2,\dots,\N$ in $\Tor$, and let $\brpt_\ast:=\left(\eps,\eps\right)$. The \emph{base string} is the line segment $\brpt_\ast\times\I \subset \Tor\times\I$. By definition, a \emph{braid} is a collection of $\N$ curves in $\Tor\times\I$ connecting $\{(\brpt_1,0),\dots,(\brpt_\N,0)\}$ to $\{(\brpt_1,1),\dots,(\brpt_\N,1)\}$ (in some order), so that the $\I$-coordinate of each curve is monotone increasing. The curves are required to be disjoint from each other and from the base string. See \cref{fig:toric-braids} for examples.

Let $\BSkn(\Tor,\ast)$ consist of $\C(\qq,\tt)$-linear combinations of such braids subject to the local relations
\begin{equation}\label{eq:FLY_skein_t}
  \LPlus-\LMinus=\left(\tt-\tti\right) \LZero \quad\text{and}\quad \BaseStringA=q^{-1}\BaseStringB,
\end{equation}
where $\ast$ denotes the base string. The first relation coincides with~\eqref{eq:intro_FLY_skein}. %
The multiplication comes from stacking in the $\I$ direction.

\begin{figure}
\def\wid{0.2}
  \setlength{\tabcolsep}{2pt}
\begin{tabular}{cc|cc}
\multicolumn{2}{c|}{\includegraphics[width=\wid\textwidth]{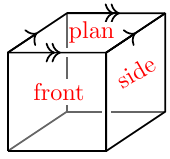}}
&
  \includegraphics[width=\wid\textwidth]{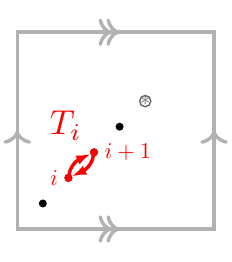}
&
  \includegraphics[width=\wid\textwidth]{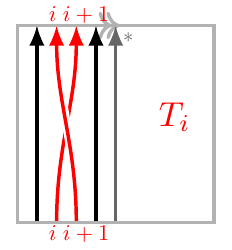}
\\
\multicolumn{2}{c|}{(a) plan/front/side view of $\Tor\times\Interval$} & (b) plan & (b) front/side \\ 
\hline
  \includegraphics[width=\wid\textwidth]{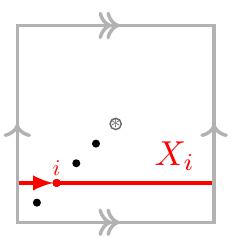}
&
  \includegraphics[width=\wid\textwidth]{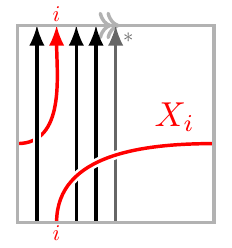}
&
  \includegraphics[width=\wid\textwidth]{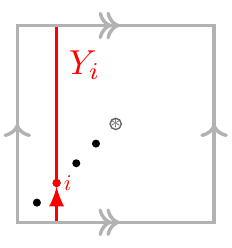}
&
  \includegraphics[width=\wid\textwidth]{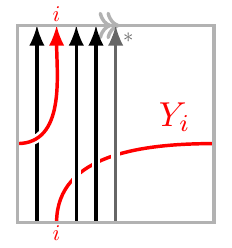}
\\
 (c) plan & (c) front & (d) plan & (d) side 
\end{tabular}
  \caption{\label{fig:toric-braids}The toric braids $\Tbraid_i$, $\Xbraid_i$, $\Ybraid_i$ from \cref{dfn:XYT_braids}.}
\end{figure}
\begin{definition}\label{dfn:XYT_braids}
We have braids $\Tbraid_i$ for $i=1,2,\dots,\N-1$ and $\Xbraid_i,\Ybraid_i$ for $i=1,2,\dots,\N$ shown in \cref{fig:toric-braids}:
\begin{itemize}
\item $\Tbraid_i$ rotates the points $\brpt_i$ and $\brpt_{i+1}$ clockwise; %
\item $\Xbraid_i$ moves the point $\brpt_i$ to the right by vector $(1,0)$;
\item $\Ybraid_i$ moves the point $\brpt_i$ up by vector $(0,1)$.
\end{itemize}
\end{definition}

\begin{notation}\label{not:MS21}
Our notation is obtained from that of~\cite{MS21} by substituting $(x_i,y_i,\sigma_i,c,s)\mapsto(\Xbraid_i,\Ybraid_i, \Tbraid_i^{-1},\qqi,\tt)$. In particular, their generators $\sigma_i$ are represented by positive crossings, in agreement with \cref{not:T_i}.
\end{notation}

\begin{theorem}[{\cite[Theorem~3.7]{MS21}}]
  The map $(\Tbraid_i,\Xbraid_i,\Ybraid_i)\mapsto (T_i,X_i,Y_i)$ induces an isomorphism 
\begin{equation}\label{eq:BSk_to_DAHA}
  \BSkn(\Tor,\ast)\xrasim \DAHA_\N.
\end{equation}
\end{theorem}

Next, we consider the algebra $\Skn(\Tor,\ast)$ of \emph{framed tangles}. By definition, a tangle is a collection of $\N$ oriented curves in $\Tor\times\I$ connecting $\{(\brpt_1,0),\dots,(\brpt_\N,0)\}$ to $\{(\brpt_1,1),\dots,(\brpt_\N,1)\}$, together with some closed curves in $\Tor\times\I$. As before, the curves are required to be disjoint from each other and from the base string. (However, they are no longer required to be monotone in the $\I$ direction.)
 In a framed tangle, each curve is additionally given together with a framing. The algebra $\Skn(\Tor,\ast)$ consists of $\Cqt$-linear combinations of such framed tangles modulo the relations~\eqref{eq:FLY_skein_t}, and the multiplication again comes from stacking in the $\I$ direction.

 By~\cite[Theorem~4.1]{MS21}, we have a homomorphism 
\begin{equation}\label{eq:BSkn_to_Skn}
  \BSk_\N(\Tor,\ast)\to\Skn(\Tor,\ast)
\end{equation}
 obtained by taking a braid and choosing its framing in a canonical way: for instance, one can let the framing be spanned by the braid strand and the vector $(1,0,0)$, which is never tangent to the strand. By~\cite[Theorem~4.2]{MS21}, this homomorphism is surjective. That is,  given a framed tangle, one can apply the skein relations~\eqref{eq:FLY_skein_t} to express it as a linear combination of braids equipped with canonical framing. It turns out that this map is actually an isomorphism.
\begin{theorem}[{\cite[Conjecture~1.5]{MS21} and \cite[Theorem~5.10]{BCMN}}]\label{thm:MS_conjecture_true}
  The map~\eqref{eq:BSkn_to_Skn} is an isomorphism.
\end{theorem}

\subsection{From monotone curves to DAHA}
Let $\Curve$ be a primitive curve given as a plot of a monotone function $f:[0,\rectX]\to[0,\rectY]$. Recall from \cref{sec:intro:links} that we may consider a curve $\Curve_1$ in $\Tor\times\I$ parameterized by $\Curve_1(t)=(x,f(x),x/\rectX)$ for $x\in[0,\rectX]$. The curve $\Curve_1$ connects $(0,0,0)$ to $(\rectXY,1)=(0,0,1)$ in $\Tor\times\I$. Let $\Curve_2$ be a vertical line segment connecting $(0,0,1)$ to $(0,0,0)$. The union $\WC^\Tor:=\Curve_1\cup\Curve_2$ is a closed curve in $\Tor\times\I$.

Suppose that $\Tor-\Diske$ is a punctured torus such that the puncture $\Diske$ contains the line segment $[0,\brpt_\ast]$, where $\brpt_\ast=\left(\eps,\eps\right)$. For $\N\geq1$, we have a homomorphism 
\begin{equation}\label{eq:Skp_to_Skn}
  \Skpti(\Tor-\Diske)\to \Skn(\Tor,\ast)
\end{equation}
  obtained by inserting the identity $\N$-braid with a base string into $\Diske$; see \figref{fig:curve-PTor}(a). 

Let $\Curve$ be a primitive curve, and consider the curves $\Curve_\pm$ obtained by shifting $\Curve$ by $(\mp\eps,\pm\eps)$ as in \cref{sec:intro_skein}. Thus, we may consider the elements $\W^{\PTor}_{\Curve_\pm}\in\Skti(\PTor)$ and $\WC^\PTor\in\SkPTor$ defined in \cref{sec:intro_skein}. The map~\eqref{eq:intro_Sk_to_DASHA} is the composition
\begin{equation}\label{eq:composite_Skp_to_DASHA}
  \SkPTor\to \Skn(\Tor,\ast) \xrasim \BSkn(\Tor,\ast) \xrasim \DAHA_\N \to \DASHA_\N.
\end{equation}
of~\eqref{eq:Skp_to_Skn}, the inverse of~\eqref{eq:BSkn_to_Skn}, \eqref{eq:BSk_to_DAHA}, and the idempotent projection $\DAHA_\N\to\DASHA_\N$, $x\mapsto \S x\S$.

\begin{figure}
\def\wid{0.18}
\begin{tabular}{cc|cc}
  \includegraphics[width=\wid\textwidth]{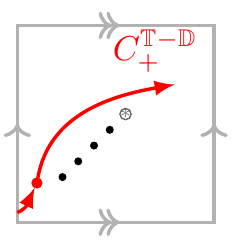}
&
  \includegraphics[width=\wid\textwidth]{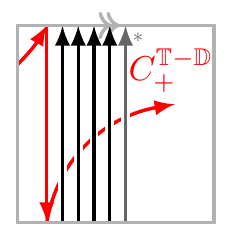}
&
  \includegraphics[width=\wid\textwidth]{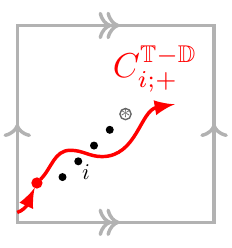}
&
  \includegraphics[width=\wid\textwidth]{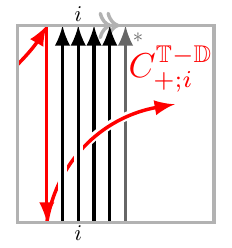}
\\
(a) plan & (a) front & (b) plan & (b) front
\\
\hline
  \includegraphics[width=\wid\textwidth]{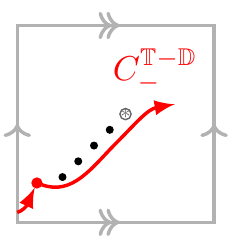}
&
  \includegraphics[width=\wid\textwidth]{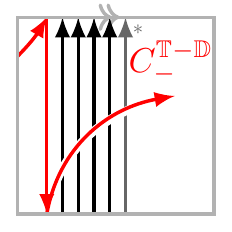}
&
  \includegraphics[width=\wid\textwidth]{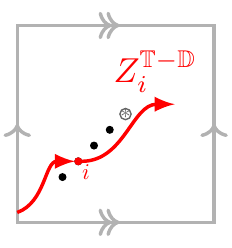}
&
  \includegraphics[width=\wid\textwidth]{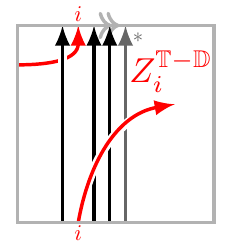}
\\
(c) plan & (c) front & (d) plan & (d) front
\end{tabular}
  \caption{\label{fig:curve-PTor} Toric braids and tangles in the proof of \cref{thm:intro_curve_to_DASHA}, each shown in plan view and front view; cf. \figref{fig:toric-braids}(a).}
\end{figure}

\begin{proof}[Proof of \cref{thm:intro_curve_to_DASHA}]
Let $\Curve$ be a primitive curve. Recall from \cref{dfn:intro:DCn} that $\PathC$ is the highest up-right lattice path  below $\Curve$. The neighborhood of the corner (i.e., the beginning and the end) of $\WC^\PTor_+$ (resp., $\WC^\PTor_-$) is shown in \figref{fig:curve-PTor}(a) (resp., \figref{fig:curve-PTor}(c)).

For $i=0,1,\dots,\N$, let $\WC^\PTor_{i;+}$ be obtained by replacing the leftmost $i$ undercrossings in \figref{fig:curve-PTor}(a), front view, with overcrossings; see \figref{fig:curve-PTor}(b) and~\eqref{eq:Skq_Laq_iso_Z_A}. For $i=1,\dots,\N$, let $Z^\PTor_i$ be obtained from $\WC^\PTor_{i;+}$ by uncrossing the $i$-th overcrossing; see \figref{fig:curve-PTor}(d).  Thus, $\WC^\PTor_{0;+}=\WC^\PTor_+$ 
 and $\WC^\PTor_{\N;+}=q\WC^\PTor_-$ by the second relation in~\eqref{eq:FLY_skein_t}. Applying the first relation in~\eqref{eq:FLY_skein_t}, we obtain
\begin{align*}%
  \WC^\PTor_+=\WC^\PTor_{0;+} 
&= \WC^\PTor_{1;+}+\left(\tti-\tt\right) Z_1^\PTor\\
&= \WC^\PTor_{2;+}+\left(\tti-\tt\right) (Z_1^\PTor+Z^\PTor_2)\\
&= \cdots\\
&= q\WC^\PTor_- + \left(\tti-\tt\right) (Z^\PTor_1+Z^\PTor_2+\dots+Z^\PTor_{\N}).
\end{align*}
Each tangle $Z^\PTor_i$ is isotopic to a braid (also denoted $Z^\PTor_i$). The braid $Z^\PTor_1$ is obtained as a product of: 
\begin{itemize}
\item $Y_1$ for each $U$ step of $\PathC$;
\item $Y_1X_1Y_1^{-1}$ for each $R$ step of $\PathC$.
\end{itemize}
For $i=1,2,\dots,\N$, let $Z_i\in\DAHA_\N$ be the image of the braid $Z^\PTor_i\in\BSkn(\Tor,\ast)$ under the map~\eqref{eq:BSk_to_DAHA}. We see that $Z_i=T_{i-1}^{-1}\cdots T_2^{-1}T_1^{-1} Z_1 T_1^{-1}T_2^{-1}\cdots T_{i-1}^{-1}$. Projecting to $\DASHA_\N$ and applying~\eqref{eq:ST=TS=S}, we find $\S Z_i\S=t^i\S Z_1\S$. Comparing the above description of $Z^\PTor_1$ to \cref{dfn:intro:DCn}, we find $\DCn=\gam\S Z_1\S$. Summarizing, the image of $\WC^\PTor_+-q\WC^\PTor_-$ in $\DASHA_\N$ under the composition~\eqref{eq:composite_Skp_to_DASHA} is given by 
\begin{equation*}%
  \frac{\left(\tti-\tt\right)}{\gam} (1+t+\cdots+t^{\N-1}) \DCn=\tti \DCn.
\end{equation*}
Thus, the image of $\WC^\PTor$ under~\eqref{eq:composite_Skp_to_DASHA} is indeed given by~\eqref{eq:intro_Sk_to_DASHA_image} when $\Curve$ is primitive. The case of arbitrary $\Curve$ follows by~\eqref{eq:intro_WC_PTor_product} and~\eqref{eq:intro:prod}. 
\end{proof}

\begin{figure}
\begin{tabular}{ccc}
 \includegraphics[width=0.2\textwidth]{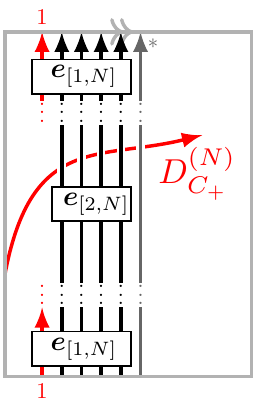} 
&
 \includegraphics[width=0.2\textwidth]{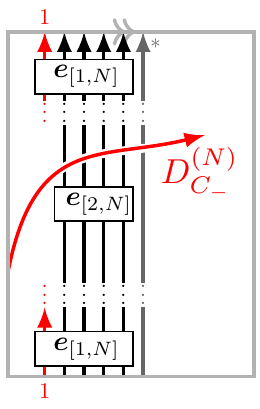}
&
 \includegraphics[width=0.2\textwidth]{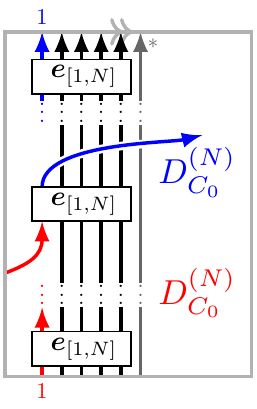}
\end{tabular}
  \caption{\label{fig:skein-diag}A diagrammatic proof of the skein relation~\eqref{eq:intro:skein} (front view); cf \cref{rmk:skein_diagrammatic}. Compare with~\cite{Hog} or~\cite[Equation~(1.2)]{ElHo}.} 
\end{figure}

\begin{remark}\label{rmk:skein_diagrammatic}
The recurrence~\eqref{eq:symmetrizer_recurrence} may be represented diagrammatically; see e.g.~\cite[Equation~(1.2)]{ElHo}. Similarly, our proof of the skein relation~\eqref{eq:intro:skein} given in \cref{sec:skein} may be interpreted as shown in \cref{fig:skein-diag}. Given that~\eqref{eq:symmetrizer_recurrence} admits a lift to the level of Rouquier complexes computing KR homology~\cite{Hog,ElHo,Mellit_torus}, it is natural to expect that our skein relation admits a similar lift to KR homology or, perhaps, to some other (yet to be constructed) homology theory associating a symmetric function to a link in a punctured torus.
\end{remark}

\section{A commutative diagram}\label{sec:diagram}

In this section, we explore the commutative diagram in \cref{fig:giant}, which is an expanded version of the diagram in \cref{fig:big}. The maps between the various algebras and modules are shown on the left, and the images of a curve $\Curve$ under these maps are shown on the right. 
For each of the squares \sqa, \sqb1, \sqb2, \sqc1, \sqc2, we will discuss the maps involved 
 and check that the square commutes. Throughout this section, we let $\Curvex$ be an almost linear curve from $(0,0)$ to $\bx:=(\rectXY)$. We set $d:=\gcd(\rectXY)$, $\bx_0:=\frac1d\bx$, and recall that  $\mul:=\ttonemqi$;  cf.~\eqref{eq:intro_Sk_to_DASHA_image}. In particular, $\mulqqi=\frac1{\qqi-\qq}$. 
\begin{figure}
\begin{adjustbox}{max width=\textwidth}
\begin{tabular}{cc}
$\displaystyle\begin{tikzcd}[column sep=0.2em,row sep=-0.2em]
\SkpPTor \arrow[rrrr,"\eqref{eq:composite_Skp_to_DASHA}"] 
\arrow[dd,"\eqref{eq:Skp_PTor_to_Tor}","t=\qi"']
&&
 && \DASHAp_\N \arrow[dd,"\N\to\infty"]\\
&& \SQA &&
 \\
\Skp(\Tor) \arrow[dd,"\text{Dfn.~\ref{dfn:Skp_Tor_act_on_Ann}}"] \arrow[rr,"\eqref{eq:MS17_iso}","\sim"'] 
&& \EHApqi\arrow[dd,"\acton",outer sep=-5pt] 
&& 
\EHAp \arrow[ll,"t=\qi"']\arrow[dd,"\acton",outer sep=-5pt]\\
& \SQB1 && \SQB2 & \\
\Skp(\Ann)  \arrow[rr,"\eqref{eq:Skp_Ann_to_Laq_twisted}","\sim"'] 
\arrow[dd,"\eqref{eq:Skp_Ann_to_Disk}"]
&& \Laq \arrow[dd,"\eqref{eq:Laq_to_Caq}"]
 && \Laqt\arrow[dd,"\eqref{eq:Laqt_to_Caqt}"] \arrow[ll,"t=\qi"'] 
\\ 
& \SQC1 && \SQC2 & \\
\Sk(\Disk)  \arrow[rr,"\eqref{eq:L=HOMFLY}","\sim"'] && \Caq && \Caqt \arrow[ll,"t=\qi"'] 
\end{tikzcd}$
&
$\displaystyle\begin{tikzcd}[column sep=-0.5em,row sep=-0.2em]
\WC^\PTor \arrow[rrrr,mapsto] 
\arrow[dd,mapsto]
&&
 && \mul^k \DCn \arrow[dd,mapsto] \arrow[dd,mapsto]\\
&& \SQA &&
 \\
\WC^\Tor \arrow[dd,mapsto] \arrow[rr,mapsto] 
&& \mul^k\DClimEHA|_{t=\qi} \arrow[dd,mapsto] 
&& 
\mul^k\DClimEHA \arrow[ll,mapsto]\arrow[dd,mapsto]\\
& \SQB1 && \SQB2 & \\
\WC^\Ann  \arrow[rr,mapsto] 
\arrow[dd,mapsto]
&&  
(-\qq)^{-\rectX}\Tro_{\Laq}(\TC)
\arrow[dd,mapsto]
 && 
\mul^k\DClimEHA\cdot 1
\arrow[dd,mapsto] \arrow[ll,mapsto] 
\\ 
& \SQC1 && \SQC2 & \\
\WC^\Disk  \arrow[rr,mapsto]
&& 
a^{\wC} \PHOML(a,q)
&& \mul^k\PEHAC(a,q,t) \arrow[ll,mapsto] 
\end{tikzcd}$
\end{tabular}
\end{adjustbox}
\caption{\label{fig:giant} An expanded version of \cref{fig:big}. Here, $\Curve$ is a curve from $(0,0)$ to $(\rectXY)$ with  $k:=\nseg(\Curve)$ lattice segments.}
\end{figure}

\subsection{Square \texorpdfstring{\sqa}{A}}\label{sec:SQA}

By \cref{thm:intro_curve_to_DASHA}, for any curve $\Curve$, the image of $\Curve^\PTor$ under~\eqref{eq:composite_Skp_to_DASHA} is $\mul^{\nseg(\Curve)}\DASXn_{\Curve}$, and by \cref{thm:intro_limit_exists}, its limit under $\N\to\infty$ is $\mul^{\nseg(\Curve)}\ehagenb_\Curve$. We now check that the isomorphism $\Skp(\Tor)\xrasim \EHApqi$ constructed in~\cite{MS17} sends $\WC^\Tor\mapsto \mul^{\nseg(\Curve)}\ehagenb_\Curve|_{t=\qi}$.

\begin{notation}
Our notation is obtained from that of~\cite{MS17} via\footnote{For instance, the substitution $(q,t)\mapsto(\qi,t)$ comes from comparing~\eqref{eq:EHA_alpha_dfn} to the definition of $\alpha_i$ in~\cite[Section~5]{MS17}.}
\begin{equation}\label{eq:MS17_conventions}
  (q,t,s,v,P_{\bx})\mapsto (\qi,t,\tt,v,\WMSTor_{\bx}).
\end{equation}
\end{notation}

We claim that the isomorphism of~\cite[Theorem~5.6]{MS17} is given by 
\begin{equation}\label{eq:MS17_iso}
  \Skp(\Tor)\xrasim \EHApqi,\quad \WMSTor_{k\bx_0}\mapsto
  -q^{-k/2}\Peha_{k\bx_0}|_{t=\qi}.
\end{equation}
Indeed, in the notation of~\cite[Theorem~5.6]{MS17}, the isomorphism $\Skp(\Tor)\xrasim \EHApqi$ sends $P_{\bx}\mapsto (q^{\dop(\bx)/2}-q^{-\dop(\bx)/2}) u_{\bx}$. Converting this to our notation and applying~\eqref{eq:Peha_dfn}, we obtain~\eqref{eq:MS17_iso}.

\begin{figure}
\begin{tabular}{ccc}
\includegraphics[width=0.13\textwidth]{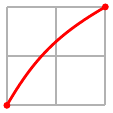}
&
\includegraphics[width=0.13\textwidth]{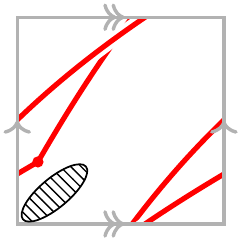}
&
\includegraphics[width=0.13\textwidth]{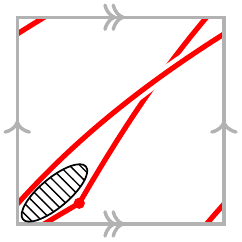}
\\
\textcolor{red}{$\Curve$} & \textcolor{red}{$\WC^\PTor_+$} & \textcolor{red}{$\WC^\PTor_-$}
\end{tabular}
  \caption{\label{fig:stuck}For an almost linear curve $\Curve$ from $(0,0)$ to $(2,2)$, $\WC^\PTor_+$ is a $\sigma_1$-decoration of a curve $\Curve_{(1,1)}$ in $\PTor$ of homology class $(1,1)$ while $\WC^\PTor_-$ is not. The shaded ellipse represents the puncture $\Disk$ in $\PTor$.}
\end{figure}

The element $\WCx^\Tor$ is obtained by taking a straight line segment $\Curvexo$ from $(0,0)$ to $\bx_0$ and decorating $\W^\Tor_{\Curvexo}$ by the Coxeter braid $\sigma_1\sigma_2\cdots\sigma_{d-1}$.\footnote{Letting $\Curve_{\bx,+}$ and $\Curve_{\bx,-}$ be the shifts of $\Curvex$ as in \cref{sec:intro_skein}, we see that $\W^\PTor_{\Curve_{\bx,+}}$ is a $\sigma_1\sigma_2\cdots\sigma_{d-1}$-decoration of $\W^\Tor_{\Curvexo}$. However, the same is not true for $\W^\PTor_{\Curve_{\bx,-}}$, since the puncture ``gets stuck'' inside the decoration; see \cref{fig:stuck}.} By \cref{not:sigma_positive} and~\cite[Proposition~2.6]{GoWe}, we have
\begin{equation*}%
  \Tr_{\Laq}(T_1^{-1}T_2^{-1}\cdots T_{d-1}^{-1})=\frac{(-1)^{d-1}}{\qq-\qqi} e_d[\XSym(\qq-\qqi)].
\end{equation*}
Embedding $\Ann$ into a small ribbon neighborhood of $\Curvexo$ and applying~\eqref{eq:Sk_emb_Sk}, we obtain a homomorphism $\Sk(\Ann)\to\Sk(\Tor)$. Combining this with \cref{prop:GoWe_coincide}, we obtain the following result.
\begin{corollary}
  The element $\WCx^\Tor$ is expressed in $(\WMSTor_{k\bx_0})_{k\geq0}$ with the same coefficients as
  $\frac{(-1)^{d-1}}{\qq-\qqi} e_d[\XSym(\qq-\qqi)]$
  is expressed in $(p_k)_{k\geq0}$.
\end{corollary}
By~\cite[Theorem~5.7]{MS21}, the map~\eqref{eq:composite_Skp_to_DASHA} sends $\WMSPTorti_{k\bx_0}\mapsto \frac{t^{k/2}-t^{-k/2}}{q^k-1}\Peha_{k\bx_0}$. For $t=\qi$, this recovers the isomorphism~\eqref{eq:MS17_iso} 
 of~\cite{MS17}; 
 see also~\cite[Remark~2.5]{MS21}. 
 Recall that the substitution $p_k\mapsto -q^{-k/2}p_k$ corresponds to the plethystic substitution $\Laq\to\Laq$, $F\mapsto F[-q^{-1/2}\XSym]$. We obtain the following.
\begin{corollary}\label{cor:MS17_iso_coefs}
  The image of $\WCx^\Tor$ under~\eqref{eq:MS17_iso} is expressed in $(\Peha_{k\bx_0})_{k\geq0}$ with the same coefficients as
 $\frac{(-1)^{d-1}}{\qq-\qqi} e_d[\XSym(\qi-1)]$  
is expressed in $(p_k)_{k\geq0}$.
\end{corollary}
\begin{proposition}
Square \sqa in \cref{fig:giant} is commutative.
\end{proposition}
\begin{proof}
  Recall that by \cref{thm:intro_curve_to_DASHA,thm:intro_limit_exists}, the image of $\WCx^\PTor$ in $\EHAp$ is given by $\mul \DASXlimEHA_{\bx}$. Substituting $t=\qi$ into \cref{cor:DASXn_vs_Pgen_EHA}, we see that $\mul \DASXlimEHA_{\bx}|_{t=\qi}=\frac1{\qqi-\qq} \DASXlimEHA_{\bx}|_{t=\qi}$ is expressed in $(\Peha_{k,0})_{k\geq0}$ with the same coefficients as $\frac{(-\qi)^d}{\qqi-\qq} e_d\left[\XSym(1-q)\right]$ is expressed in $(p_k)_{k\geq0}$. Since these coefficients agree with those in \cref{cor:MS17_iso_coefs}, it follows that the image of $\WCx^\Tor$ under~\eqref{eq:MS17_iso} agrees with the $t=\qi$-specialization of the image of $\WCx^\PTor$ under~\eqref{eq:composite_Skp_to_DASHA}. In other words, square~\sqa in \cref{fig:giant} is commutative for almost linear curves.

Since all maps involved respect multiplication, we see that square~\sqa commutes for piecewise almost linear curves. We now check that the map~\eqref{eq:MS17_iso} also respects the skein relation~\eqref{eq:intro:skein}. (This is clear for all other maps involved in square~\sqa.) Substituting $\DCn\mapsto \mul^{-\nseg(\Curve)} \WC^\Tor$ (cf.~\eqref{eq:intro_Sk_to_DASHA_image}) into the skein relation~\eqref{eq:intro:skein} and letting $t=\qi$, we 
arrive at the identity
\begin{equation}\label{eq:skein_relation_qt1}
  \W^{\Tor}_{\Curve_+}=\W^{\Tor}_{\Curve_-}+(\qqi-\qq) \W^\Tor_{\Curve_0}
\end{equation}
for $\Curve_+,\Curve_-,\Curve_0$ as in \cref{thm:intro:skein}. This agrees with~\eqref{eq:Sk_dfn} since $\Curve_+$ contains a positive crossing $\LPlus$ where $\Curve_-$ contains a negative crossing $\LMinus$. Thus,~\eqref{eq:skein_relation_qt1} indeed holds inside $\Sk(\Tor)$, which by \cref{rmk:intro:skein} implies that square~\sqa in \cref{fig:giant} is commutative for arbitrary curves.
\end{proof}

\begin{figure}
  \setlength{\tabcolsep}{2pt}
\begin{tabular}{cc|cc|cc}
\includegraphics[width=0.15\textwidth]{figures/example-2x2-1-9-plus-proj}
&
\includegraphics[width=0.15\textwidth]{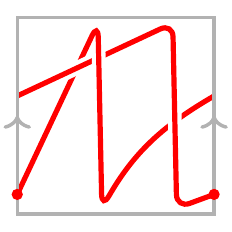}
&
\includegraphics[width=0.15\textwidth]{figures/example-2x2-1-9-minus-proj}
&
\includegraphics[width=0.15\textwidth]{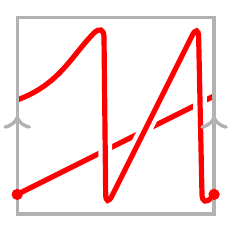}
&
\includegraphics[width=0.15\textwidth]{figures/example-2x2-1-9-zero-proj}
&
\includegraphics[width=0.15\textwidth]{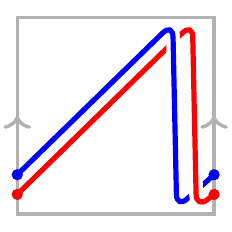}
\\
$\WC^{\Tor}_+$ & $\br^\Ann_{\Curve_+}=\sigma_1^3$ &
$\WC^{\Tor}_-$ & $\br^\Ann_{\Curve_-}=\sigma_1$ &
$\WC^{\Tor}_0$ & $\br^\Ann_{\Curve_0}=\sigma_1^2$ 
\end{tabular}
\caption{\label{fig:TC}Computing the braid $\br^\Ann_\Curve$ associated to a curve $\Curve$; see \cref{dfn:TC}.}
\end{figure}

\subsection{Squares \texorpdfstring{\sqb1 and \sqb2}{B1 and B2}}
Recall from \cref{dfn:Skp_Tor_act_on_Ann} that the vertical map $\Skp(\Tor)\to\Skp(\Ann)$ consists of first applying an automorphism $\vrap$  of $\Skp(\Tor)$ sending $\WC^\Tor\mapsto v^n\WC^\Tor$
 and then inserting a solid torus inside $\Tor$ as shown in \cref{fig:wrap}. 
\begin{definition}\label{dfn:TC}
  Let $\Curve$ be a curve from $(0,0)$ to $(\rectXY)$. Consider the link $\LC$ whose link diagram is drawn on $\Tor$ as in \cref{sec:intro:links}. For each $x\in\Interval$ such that the point $(x,1)$ (equivalently, $(x,0)$) belongs to the link diagram of $\LC$, draw a vertical line segment from $(x,1)$ to $(x,0)$ passing above all other strands of $\LC$. The result is a planar link diagram drawn on $\Ann$, where the link is obtained as the annular closure of an $\rectX$-strand braid denoted $\br^\Ann_\Curve$. We let $\TC\in\Heckeq_\rectX$ be the image of $\br^\Ann_\Curve$ under the standard homomorphism $\BraidGroup_\rectX\to\Heckeq_\rectX$, and let $\WC^\Ann\in\Skp(\Ann)$ be the image of $\TC$ under the identification~\eqref{eq:Skp_Ann_to_Hecke}. See \cref{fig:TC}.
\end{definition}
\begin{remark}
The ``wrapping'' procedure described in \cref{dfn:Skp_Tor_act_on_Ann} produces a framed link $\LC'$ which coincides with $\LC$ (with blackboard framing) except that the framing of $\LC'$ contains $\rectY$ extra full twists. This discrepancy is accounted for by the application of the automorphism $\vrap$ in \cref{dfn:Skp_Tor_act_on_Ann}.
\end{remark}

The horizontal map $\phi_{\Ann}:\Skp(\Ann)\xrasim\Laq$ is then defined as follows. For a braid $\br\in\BraidGroup_\rectX$, we set
\begin{equation}\label{eq:Skp_Ann_to_Laq_twisted}
  \phi_\Ann(\brh):=(-\qq)^{-\rectX}\omega\Tr_{\Laq}(T_\br),
\end{equation}
where $\omega$ is the omega involution~\cite[Section~7.6]{EC2}. It is clear that $\phi_{\Ann}$ is an algebra homomorphism.

\begin{proposition}
Squares \sqb1 and \sqb2 in \cref{fig:giant} are commutative.
\end{proposition}
\begin{proof}
The result is obvious for \sqb2. To check the commutativity of \sqb1, we need to show that for any curve $\Curve$ from $(0,0)$ to $(\rectXY)$, we have
\begin{equation}\label{eq:DClimEHA_vs_Tr_TC}
  \left(\qqi-\qq\right)^{-\nseg(\Curve)} \DClimEHA\cdot 1|_{t=\qi} = (-\qq)^{-\rectX}\omega \Tr_{\Laq}(\TC).
\end{equation}
This may be deduced from~\cite[Proposition~7.17]{MS17} after matching our conventions to theirs; see also~\eqref{eq:MS17_conventions} and~\cite[Section~7.1.1 and Equation~(7.7)]{MS17}.
\end{proof}

\begin{example}
Consider the three curves in \cref{fig:2x2}. We compute the left-hand side of~\eqref{eq:DClimEHA_vs_Tr_TC} using~\eqref{eq:DCp.1}--\eqref{eq:DC0.1}. For the right-hand side, we first use \cref{fig:TC} to find $\TX_{\Curve_+}=T_1^{-3}$, $\TX_{\Curve_-}=T_1^{-1}$, $\TX_{\Curve_0}=T_1^{-2}$, and then apply~\eqref{eq:Tr_Laq(T1)}. We calculate that for $\Curve=\Curve_+,\Curve_-,\Curve_0$, both sides of~\eqref{eq:DClimEHA_vs_Tr_TC} are given by $q^{-\frac52}s_{11}-q^{\frac12}s_2$, $q^{-\frac32}s_{11}-q^{-\frac12}s_2$, $q^{-2}s_{11}+s_2$, respectively.
\end{example}

\subsection{Squares \texorpdfstring{\sqc1 and \sqc2}{C1 and C2}}
Before discussing the maps in squares \sqc1 and \sqc2, we revisit the symmetric function $\FC$ introduced in~\eqref{eq:intro_FC_dfn}. By definition, $\FC$ is obtained from $\mul^{\nseg(\Curve)} \DClimEHA\cdot 1$ via the map $F\mapsto \mul^{-\nseg(\Curve)}F \left[\frac{\XSym}{1-t}\right]$. Specializing $t=\qi$ and applying~\eqref{eq:DClimEHA_vs_Tr_TC}, we get
\begin{equation}\label{eq:*}
  \FC|_{t=\qi}= \left(\qqi-\qq\right)^{\nseg(\Curve)} \Tr_{\Laq}(\TC)\left[\frac{\XSym}{\qqi-\qq}\right].
\end{equation}
Next, let the right vertical map in square \sqc2 be given by 
\begin{equation}\label{eq:Laqt_to_Caqt}
  \Laqt\to\Caqt,\quad F\mapsto  F \left[\frac{a-a^{-1}}{1-t}\right].
\end{equation}
Specializing $t=\qi$, we define the left vertical map in square \sqc2 by
\begin{equation}\label{eq:Laq_to_Caq}
  \Laq\to\Caq,\quad F\mapsto  F \left[\frac{a-a^{-1}}{1-\qi}\right].
\end{equation}
The left vertical map in square \sqc1 is given by~\eqref{eq:Skp_Ann_to_Disk}.

Recall the braid $\br^\Ann_\Curve$ from \cref{dfn:TC}. Our goal is to compute the \emph{writhe} $\wC$ of $\br^\Ann_\Curve$, defined as the sum of exponents $e_1+e_2+\cdots+e_t$ of the expression $\br^\Ann_\Curve=\sigma_{i_1}^{e_1}\sigma_{i_2}^{e_2}\cdots \sigma_{i_t}^{e_t}$ of $\br^\Ann_\Curve$ in terms of the standard generators of $\BraidGroup_\rectX$.
\begin{proposition}
  For a curve $\Curve$ from $(0,0)$ to $(\rectXY)$,
  \begin{equation}\label{eq:wC}
    \wC=(\nseg(\Curve)-1) + (\rectX-1) + 2\bC,
  \end{equation}
  where $\bC$ is the number of lattice points in $[1,\rectX-1]\times[1,\rectY-1]$ strictly below $\Curve$.
\end{proposition}
\begin{proof}
  The result is readily checked by induction on $\bC$. For the induction base, when $\bC=0$ and $\Curve$ is primitive, we have $\br^\Ann_\Curve=\sigma_1\sigma_2\cdots \sigma_{\rectX-1}$. Next, suppose that $\Curve_+$, $\Curve_-$, $\Curve_0$ are three curves passing above, below, and through some lattice point $\point$ as in \cref{thm:intro:skein}. The associated braids satisfy $\wop(\br^\Ann_{\Curve_+})-1=\wop(\br^\Ann_{\Curve_0})=\wop(\br^\Ann_{\Curve_-})+1$. Thus, if~\eqref{eq:wC} holds for $\Curve_-$ then it holds for $\Curve_0$ and $\Curve_+$, which gives the induction step.
\end{proof}

\begin{proposition}
Squares~\sqc1 and~\sqc2 in \cref{fig:giant} are commutative.
\end{proposition}
\begin{proof}
Square \sqc2 commutes by construction. The fact that square \sqc1 commutes follows by combining~\eqref{eq:L=HOMFLY}, \eqref{eq:wC}, and~\cite[Proposition~2.3(b)]{GoWe}.
\end{proof}

\begin{proof}[{Proof of \cref{prop:tqi}}] This follows from the commutative diagram \cref{fig:giant}.  The left-hand side of \eqref{eq:intro_PEHAC=HOMFLY} is the image of $\mul^k\PEHAC(a,q,t)$ under the bottom horizontal map ($t = \qi$) in square \sqc2.
\end{proof}

\section{Comparisons and specializations}\label{sec:relations}
In this section, we view the EHA element $\DXlimEHA_C$ as an operator on symmetric functions.  We first recall the formalism of \emph{plethysm}.  If $A$ is an expression in terms of some variables, we define $p_k[A]$ to be the result of substituting $a^k$ for each variable $a$, and define $f[A]$ for any $f \in \La$ by extending $p_k \mapsto p_k[A]$ to an algebra homomorphism $f \mapsto f[A]$.  Note that the symbols $q,t$ also count as variables.  By convention, the plethysm $f[\XSym]$ is obtained by taking $\XSym = x_1+x_2+ \cdots$.  Let $\phi: \Laqt \to \Laqt$ be the automorphism of symmetric functions given by the plethysm
$f \mapsto f\left[\frac{\XSym}{1-t}\right]$, and recall the involution $\omega: \Laqt \to \Laqt$.

For a monotone curve $C$, let $\DTXlimEHA_C := \phi \circ \DXlimEHA_C \circ \phi^{-1}$ denote the symmetric function operator obtained by conjugating $\DXlimEHA_C$ by $\phi$.  Thus by definition $F_C = \DTXlimEHA_C \cdot 1$.
\subsection{Shuffle conjectures}\label{sec:shuffle}
Let $C$ be the almost linear curve from $(0,0)$ to $(m,n)$; then $L_C$ is the $(m,n)$-almost torus knot.  The associated operator $\DTXlimEHA_{m,n}$ is well studied, going back to \cite{GN} (and \cite{Cherednik_Jones,AS15}) for the case $\gcd(m,n)=1$ of a torus knot.  In particular, the (compositional) rational shuffle conjecture \cite{BGLX,Mellit_rat} gives a positive, combinatorial expression for $\DTXlimEHA_{m,n} \cdot 1$:
\def\dinv{{\rm dinv}}
\def\ides{{\rm ides}}
\def\Park{{\rm Park}}
\begin{equation}\label{eq:shuffle}
\omega(\DTXlimEHA_{m,n} \cdot 1) = \sum_{\pi \in \Park_{n,m}} t^{\area(\pi)} q^{\dinv(\pi)} F_{\ides(\pi)}.
\end{equation}
Here, $\Park_{n,m}$ denotes a certain set of parking functions, and $F_{\ides(\pi)}$ denotes a fundamental quasisymmetric function.  We refer the reader to \cite{BGLX} for definitions of $\area,\dinv,\ides$.

\subsubsection{} The original shuffle conjecture \cite{HHLRU} proved in \cite{CaMe} is a combinatorial expression for the symmetric function $\nabla e_n$.  The operator $\nabla$ of~\cite{BeGa} can be interpreted in terms of the $\SL_2(\Z)$-action on $\EHA$ as the following identity of operators on $\Laqt$:
$$
\phi \circ u_{m,m+n}  \circ \phi^{-1} = (\omega \circ \nabla \circ \omega) (\phi  \circ u_{m,n} \circ \phi^{-1}) (\omega\circ\nabla \circ \omega)^{-1}.
$$

\subsubsection{} Up to a transformation $t \mapsto \ti$, and a factor of $(1-t)$ (see~\eqref{eq:ex_DASXn_11}), our $\DXlimEHA_{m,n}$ is equal to the element $P_{m,n}$ defined as a limit in \cite[Section 3]{GN} of the elements $P^N_{m,n}$ below \cite[Definition 2.5]{GN}.  

\begin{remark}
On the other hand, for an almost linear curve $\Curve$ from $(0,0)$ to $(n,n+1)$, we have $F_C = \DTXlimEHA_{n,n+1} \cdot 1 = \omega(\nabla e_n)$, in contrast to \cite[Corollary 6.5]{GN}, where $\nabla e_n$ is obtained.  This discrepancy appears to arise due to the use of \cite[Proposition 2.3]{GN}, which differs from our symmetric function actions by the automorphism $\omega$.  For instance, according to \cite[Proposition~6.4 and Corollary~6.5]{GN}, we have $\tilde P_{0,1} \cdot p_n=\tilde P_{n,1}\cdot1=e_n$. We believe that one should instead have $\tilde P_{0,1}\cdot p_n=\frac{q^n-1}{1-\ti}h_n$ and $\tilde P_{n,1}\cdot 1=\frac1{1-\ti}h_n$ in the notation of~\cite{GN}; see~\cite[Equations~(20) and~(31)]{GN}. The factor $q^n-1$ appears to be missing from~\cite[Equation~(22)]{GN}.
\end{remark}

\subsubsection{}\label{ssec:BGLX} The symmetric function operators $\mathbf{Q}_{n,m}: \Laqt \to \Laqt$ appearing in the shuffle conjectures of \cite{BGLX} are related to ours by
\begin{equation}\label{eq:Qnm}
\DTXlimEHA_{m,n} = (-1)^{m+\gcd(m,n)} \omega \circ \mathbf{Q}_{n,m} \circ \omega.
\end{equation}
Thus, \eqref{eq:shuffle} is a rephrasing of \cite[Conjecture 3.1]{BGLX}.

\subsubsection{}\label{ssec:BHMPS}

We give an explicit description of the operator $\DTXlimEHA_{m,1}$.  
Set ${M:=(1-t)(1-q)}$.  Following \cite{BHMPS}, define $\BD_k$, $k \in \Z$ by the identity 
\begin{equation}\label{eq:Dk}
\sum_{k\in\Z} (-z)^{-k}\BD_k  := \Omega[-z^{-1} \XSym]^\bullet \Omega[zM\XSym]^\perp 
\end{equation}
where for $f \in \La$, we denote by $f^\bullet$ the multiplication operator by $f$, and by $f^\perp$ the adjoint operator to $f^\bullet$ for the Hall inner product.\footnote{In~\cite{BHMPS}, the operator $\BD_k$ was denoted $D_k$. It differs from the operator denoted $\BD_k$ in~\cite[Equation~(2.1)]{BGLX} by a sign $(-1)^k$; cf. the discussion after~\cite[Equation~(35)]{BHMPS}.}
The following equality of symmetric function operators is a special case of \eqref{eq:Qnm}:
\begin{equation}\label{eq:Dm}
\DTXlimEHA_{m,1}
= \omega \circ \BD_m \circ \omega.
\end{equation}

\subsection{Positroid links}\label{sec:positroid_links}
In \cite{GL_qtcat}, we introduced positroid links $L_f$ associated to positroid varieties in the Grassmannian, or to  bounded affine permutations.  In \cite{GL_cat_combin}, we studied repetition-free positroids, which led us to the study of convex curves.  The following result follows from \cite[Section 2.1]{GL_plabic_links} where concave curves were used instead. Specifically, the curves of~\cite{GL_cat_combin,GL_plabic_links} are related to our curves by a $180^\circ$-rotation, and the associated links $\LC$ coincide. Note that in \cref{dfn:TC}, the vertical line segments are drawn above all other strands, while they are drawn below all other strands in~\cite{GL_cat_combin,GL_plabic_links} (see e.g.~\cite[Figure~1(d)]{GL_plabic_links}).

\begin{proposition}\label{prop:poslink}
Let $C$ be a convex monotone curve, possibly passing through some lattice points.  Then $L_C$ is a positroid link. 
\end{proposition}
The positroid knots appearing in \cref{prop:poslink} coincide with the positroid knots associated to repetition-free permutations of \cite{GL_cat_combin}.  Outside of the convex class, we do not know the answer to the following question.

\begin{question}
Which monotone links are isotopic to positroid links?
\end{question}

For example, the monotone curve \figref{fig:convex-examples}(c) is not convex, however, it coincides with the positroid knot discussed in~\cite[Example~4.21]{GL_qtcat}. See \cref{ex:non_convex_positroid} for further discussion.

\subsection{Coxeter links}\label{sec:cox_links}
We relate our monotone links to the Coxeter links introduced in~\cite{ObRo}.  Consider a pair $(\AB,\EPS)$, where $\AB=(\ab_1,\ab_2,\dots,\ab_\rectX)$ is a sequence of nonnegative integers and $\EPS=(\eps_1,\eps_2,\dots,\eps_{\rectX-1})$ is a $0,1$-sequence.

 Let
 $\cox(\EPS):=\sigma_{1}^{\eps_1}\sigma_{2}^{\eps_2}\cdots \sigma_{\rectX-1}^{\eps_{\rectX-1}}$ 
be the \emph{quasi-Coxeter braid} associated to $\EPS$; cf. \cref{not:sigma_positive}. For $i=1,2,\dots,\rectX$, let $\JM_i:=\sigma_{i-1}\cdots \sigma_1 \sigma_1\cdots \sigma_{i-1}$ be a Jucys--Murphy element. By convention, $\JM_1$ is the identity braid. Finally, define 
\begin{equation}\label{eq:cox_braid_dfn}
  \betacox_{\AB,\EPS}:=\JM_1^{\ab_1}\JM_2^{\ab_2}\cdots\JM_{\rectX}^{\ab_{\rectX}} \cox(\EPS).
\end{equation}
 Braids of the form~\eqref{eq:cox_braid_dfn} are called \emph{Coxeter braids}, and their annular closures are called \emph{Coxeter links}. 

Let $\Curve$ be a curve from $(0,0)$ to $(\rectXY)$. For $i=0,1,\dots,\rectX$, let $\la_i$ be the maximal integer such that the point $(\rectX-i,\la_i)$ is weakly below $\Curve$. (In particular, $\la_0=\rectY$.) For $i=1,2,\dots,\rectX$, set $\ab_i:=\la_{i-1}-\la_{i}$ and let $\ABC:=(\ab_1,\ab_2,\dots,\ab_{\rectX})$ be the resulting sequence of nonnegative integers.  
 We also let $\SC$ be the set of $i\in\{1,2,\dots,\rectX-1\}$ such that $\Curve$ passes through a lattice point $(\rectX-i,j)$ for some $j\in\Z$. Let $\EPSC:=(\epsC_1,\epsC_2,\dots,\epsC_{\rectX-1})$ be given by $\epsC_i:=1$ if $i\notin \SC$ and $\epsC_i:=0$ otherwise.

We set $\coxC:=\cox(\EPSC)$ and $\betacoxC:=\betacox_{\ABC,\EPSC}$. 

\begin{example}
For the curve $\Curve$ in \figref{fig:cox}(a), we have $(\la_0,\la_1,\dots,\la_\rectX)=(4,3,3,2,0)$, $\ABC=(1,0,1,2)$, $\SC=\{3\}$, and $\EPSC=(1,1,0)$.
\end{example}

\begin{figure}
\begin{adjustbox}{max width=\textwidth}
\begin{tabular}{cccc}
\includegraphics[width=0.2\textwidth]{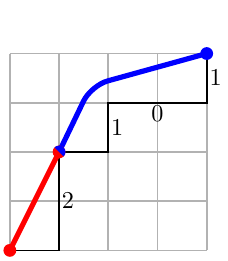}
&
\includegraphics[width=0.2\textwidth]{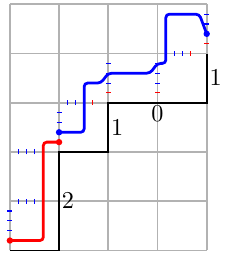}
&
\includegraphics[width=0.26\textwidth]{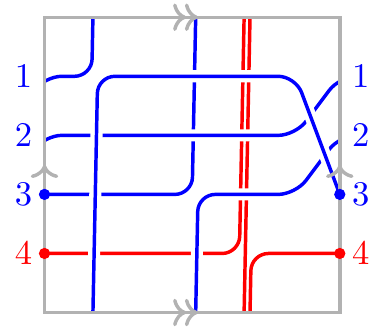}
&
\begin{tikzpicture}[baseline=(Z.base)]
\coordinate(Z) at (0,-2.15);
\node(A) at (0,0){\includegraphics[width=0.26\textwidth]{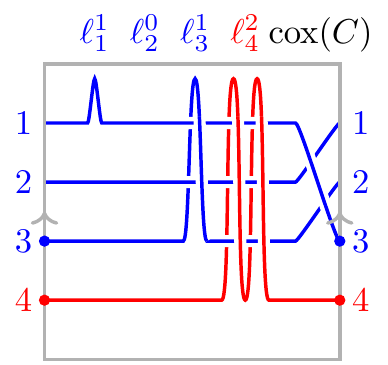}};
\end{tikzpicture}
\\
(a) 
\begin{tabular}{l}
$\ABC=(1,0,1,2)$\\
$\EPSC=(1,1,0)$
\end{tabular}
& (b) $\Curve'$ & (c) $(\Curve')^{\Tor}$ & (d) $\betacoxC:=\JM_1^{1}\JM_2^{0}\JM_3^1\JM_{4}^{2} \coxC$
\end{tabular}
\end{adjustbox}
\caption{\label{fig:cox} From monotone links to Coxeter links.}
\end{figure}
\begin{proposition}\label{prop:Coxeter=monotone}
The braid $\betacoxC$ is conjugate to the braid $\br^\Ann_\Curve$ constructed in \cref{dfn:TC}. In particular, any monotone link, when viewed as a link in $\Ann\times\Interval$, is a Coxeter link, and any Coxeter link can be obtained in this way.
\end{proposition}
\begin{proof}
This result was shown for primitive monotone curves in~\cite[Section~4.4]{GL_plabic_links}. We briefly reproduce the argument here, extending it to arbitrary curves. Let $\Curve=[\Cseg_1\Cseg_2\dots\Cseg_k]$ be a monotone curve from $(0,0)$ to $(\rectXY)$. The map $\Curve\mapsto\WC^\Tor$ described in \cref{sec:intro:links,sec:intro_skein} may be extended from monotone curves to curves from $(0,0)$ to $(\rectXY)$ whose first coordinate is strictly monotone increasing (dropping the monotonicity assumption on the second coordinate). Let us call such curves \emph{$x$-monotone}. It is easy to see that the map $\Curve\mapsto\WC^\Tor$ is invariant under applying an isotopy $\Cseg_i(t)$, $t\in\Interval$, to each segment $\Cseg_i$ of $\Curve$, such that the endpoints of $\Cseg_i(t)$ are fixed (i.e., do not depend on $t$) and for each $t$, $\Cseg_i(t)$ passes through no lattice points other than the endpoints; see also~\cite[Corollary~4.1]{GL_plabic_links}. Furthermore, we may shift each $\Cseg_i$ vertically by some vector $(0,y_i)$; the projected curve $\Csegproj_i$ will be shifted vertically; its diagram will be drawn above $\Csegproj_j$ for $j<i$ and below $\Csegproj_j$ for $j>i$, and thus such a vertical shift results in an isotopy inside $\Tor$.

Applying the above operations, we may transform each segment $\Cseg_i$ of $\Curve$ into a ``lattice path'' segment $\Cseg_i'$ as shown in \figref{fig:cox}(b). Namely, let $\piC\in\Sn$ be the permutation associated to the braid $\coxC$. For each $i=0,1,\dots,\rectX-1$, consider the part of $\Curve$ connecting $(i,y_i)$ to $(i+1,y_{i+1})$ for some $y_i<y_{i+1}$. Let $j_i:=\lfloor y_i \rfloor$ and $j_{i+1}:=\lfloor y_{i+1} \rfloor$. Let $\eps>0$ be small. Consider a new $x$-monotone curve passing through the following points: 
\begin{itemize}
\item start at $\left(i,j_i+\frac{i+1}{\rectX+1}\right)$;
\item proceed horizontally to $\left(i+1-\frac{i+1}{\rectX+1}-\eps,j_i+\frac{i+1}{\rectX+1}\right)$;
\item proceed (almost) vertically\footnote{Here we have inserted $\eps$ to keep the curve $x$-monotone.} to $\left(i+1-\frac{i+1}{\rectX+1}+\eps,j_{i+1}+\frac{i+1}{\rectX+1}\right)$;
\item proceed horizontally to $\left(i+1-\frac{1}{\rectX+1},j_{i+1}+\frac{i+1}{\rectX+1}\right)$;
\item end at $\left(i+1,j_{i+1}+\frac{\piC(i)+1}{\rectX+1}\right)$.
\end{itemize}
Let $\Curve'$ be the union of these curves for $i=0,1,\dots,\rectX-1$. By the above discussion, the projected curves $\WC^\Tor,(\Curve')^\Tor$ are related by isotopy in $\Tor$; the curve $(\Curve')^\Tor$ is shown in \figref{fig:cox}(c). Applying the map $(\Curve')^\Tor\mapsto \beta^\Ann_{\Curve'}$ from \cref{dfn:TC}, we obtain the braid $\beta^\Ann_{\Curve'}=\betacoxC$; see \figref{fig:cox}(d). Thus, the braids $\beta^\Ann_{\Curve}$ and $\betacoxC$ are conjugate.
\end{proof}

\subsection{Shuffle algebra}
We relate the operators $\DTXlimEHA_{C}$ to variants of \Negut's shuffle algebra elements.
Let $S$ denote the shuffle algebra, which we consider a subspace of $\Q(q,t)[x^{\pm 1}_1,x^{\pm 1}_2,\ldots]$, endowed with a noncommutative ``shuffle product", and an isomorphism of algebras $\psi: S \to \EHAp$; see \cite{SV13}.  The precise description of $S$ will not be important for us since we are only interested in certain distinguished elements.  The shuffle algebra has an action on $\Laqt$ compatible with $\psi$ and the action of $\EHAp$ on $\Laqt$ discussed in \cref{sec:EHA_action}, conjugated by the plethysm $\phi$.  Under this action,
$$
x_1^{a_1} \cdots x_k^{a_k} \text{ acts by } \BD_{a_1} \cdots \BD_{a_k},
$$
where $\BD_a$ is defined in \eqref{eq:Dk}.

Let $(\AB,\EPS)$ be as in \cref{sec:cox_links}.  Define the rational functions
\begin{equation*}%
  \eta'_{\AB,\EPS}=\frac{x_1^{b_1} \cdots x_k^{b_m}}{\prod_{i=2}^{m}(1-qtx_{i-1}/x_{i})^{\eps_{i-1}}}
\end{equation*}
The following result is a variant of \cite[Proposition 6.2]{Negut}.
\begin{proposition}
There are elements $\eta_{\AB,\EPS} \in S$ characterized by the equality
$$
\H(\eta_{\AB,\EPS}) = \H(\eta'_{\AB,\EPS}).
$$
where $\H$ denotes the $q,t$-symmetrization operator of \cite[Equation (46)]{BHMPS}. 
\end{proposition}
\noindent In particular, $\eta_{\AB,\EPS}$ are Laurent polynomials.

The following skein relation is a generalization of the second displayed equation in \cite[Proof of Proposition 6.13]{Negut}.\begin{proposition}\label{prop:shufskein}
The skein relation holds for the elements $\eta_{\AB,\EPS}$:
$$
  \eta_{\AB_{\Curve_+},\EPS_{\Curve_+}}=qt \eta_{\AB_{\Curve_-},\EPS_{\Curve_-}}+ \eta_{\AB_{\Curve_0},\EPS_{\Curve_0}}.
$$
\end{proposition}
\begin{proof}
This follows from linearity of $\H$ and the equality
\begin{equation*}%
 \eta'_{\AB_{\Curve_+},\EPS_{\Curve_+}}=qt \eta'_{\AB_{\Curve_-},\EPS_{\Curve_-}}+ \eta'_{\AB_{\Curve_0},\EPS_{\Curve_0}}. \qedhere
\end{equation*}
\end{proof}

\begin{proposition}\label{prop:DCshuffle}
For a curve $C$, we have the equality of symmetric function operators
\begin{equation}\label{eq:compnegut}
\DTXlimEHA_{C} = \omega \circ \eta_{\AB_C,\EPS_C} \circ \omega.
\end{equation}
Furthermore, we have (cf.~\cite[Equation~(146)]{BHMPS})
\begin{equation*}%
F_C = \DTXlimEHA_{C}\cdot 1  = \omega (D_{\ABC,\EPSC}\cdot 1) = \H\left(\frac{x_1^{b_1} \cdots x_k^{b_m}}{\prod_{i=2}^{m}(1-qtx_{i-1}/x_{i})^{\varepsilon_{i-1}}}\right)_{\pol},
\end{equation*}
where $(\cdots)_{\pol}$ denotes the ``polynomial part" operator, removing all Laurent monomials in $x_i$ that are not polynomials; see \cite[Section 2.3]{BHMPS}.
\end{proposition}
\begin{proof}[Proof sketch]
When $C$ is an almost linear curve, the result follows from the discussion in \cref{sec:shuffle} and known relations between the shuffle algebra and $\EHAp$; see \cite{Negut, GN}.   By \cref{prop:shufskein} and \eqref{eq:intro:skein} both sides of \eqref{eq:compnegut} satisfy the skein relation.  Finally, both sides of the stated equality are compatible with composition of operators, or concatenation of monotone curves.  It follows that \eqref{eq:compnegut} holds for all monotone curves.

The second statement follows from \cite[Proposition 3.5.2]{BHMPS}.
\end{proof}

\begin{remark}\label{rem:BHMPS}
\cref{prop:DCshuffle} implies that the operators $D_{b_1,\ldots,b_l}$ of \cite{BHMPS} are special cases of our operators $\DTXlimEHA_C$.
\end{remark}

\begin{remark}
In \cite[Proposition 6.4]{Negut}, for each $(m,n)$, certain elements in the shuffle algebra, and thus the EHA, are constructed, depending on a binary string $\delta = (\delta_1,\ldots,\delta_{d-1})$ where $d = \gcd(m,n)$.  Let $C_\delta$ be the primitive monotone curve from $(0,0)$ to $(m,n)$, staying close to the diagonal, and passing above or below a diagonal lattice point depending on whether $\delta_i = 0$ or $1$.  Then the image of \Negut's $X^\delta$ in $\EHA$ coincides with our operator $\DXlimEHA_{C_{\delta}}$.  Specifically, for $(m,n) = (d,0)$, using $\DASXn_{C_\delta}$ instead of $\DASXn_{d,0}$, we would obtain $s_{R_\delta}\left[\XSym(1-\ti)\right]$ instead of $e_d\left[\XSym(1-\ti)\right]$ in \cref{prop:DASXn_vs_Pgen}, where $R_\delta$ is a ribbon skew shape with $d$ boxes.   Note however that the operators $\DXlimEHA_{C_{\delta}}$ do not in general satisfy the positivity \cref{conj:intro:Schur_pos}.
\end{remark}

\subsection{Magic formula}

We state a slight generalization of the \emph{magic formula} of~\cite{Negut,GN} expressing the symmetric function $\FC$ as a sum over \emph{standard Young tableaux}. We follow the exposition of~\cite{GHSR}. We draw Young diagrams in French notation, so that the boxes are located in the nonnegative orthant with the bottom left box having coordinates $(1,1)$. Let $T$ be a standard Young tableau with $\rectX$ boxes. For $i=1,2,\dots,\rectX$, we set $z_i:=z_i(T):=q^{c-1}t^{r-1}$, where the box of $T$ containing $i$ has coordinates $(r,c)$. For a $\{0,1\}$-sequence $\EPS=(\eps_1,\eps_2,\dots,\eps_{\rectX-1})$ and a standard Young tableau $T$, we set
\begin{equation*}%
  \wt(T;\EPS):=\prod_{i=2}^{\rectX} \frac{1}{(1-z_i^{-1})(1-qtz_{i-1}/z_i)^{\eps_{i-1}}} 
\prod_{i<j} \frac{(1-z_i/z_j)(1-qtz_i/z_j)}{(1-qz_i/z_j)(1-tz_i/z_j)},
\end{equation*}
where the zero factors in the numerator and denominator are omitted by convention. Given a sequence $\AB=(\ab_1,\ab_2,\dots,\ab_{\rectX})$ of integers, introduce a symmetric function
\begin{equation}\label{eq:nabla_H_tilde}
  F_{\AB,\EPS}:=\sum_{T\in\SYT(\rectX)} \wt(T;\EPS) z_1^{\ab_1}z_2^{\ab_2}\cdots z_\rectX^{\ab_\rectX}  \nabla^{-1} \Ht_{\sh(T)},
\end{equation}
where the summation is over all standard Young tableaux $T$ with $\rectX$ boxes, $\sh(T)$ is the shape of $T$, $\Ht_\la$ are the modified Macdonald polynomials, and $\nabla$ is the nabla operator of~\cite{BeGa}.

When $\EPS=(1,1,\dots,1)$, Equation~\eqref{eq:nabla_H_tilde} is obtained from~\cite[Equation~(3)]{GHSR} by inserting the term $\nabla^{-1} \Ht_{\sh(T)}$. We thank E.~Gorsky for suggesting this modification to us. 

The proof of the following formula (also known as the \emph{magic formula}; see~\cite{ElHo}) may be obtained by adapting the methods in the proof of~\cite[Theorem~1.1]{GN}.
\begin{proposition}
For any curve $\Curve$, we have
\begin{equation*}%
  \FC=\omega F_{\ABC,\EPSC}.
\end{equation*}
\end{proposition}

\subsection{The specialization $t = 1$.}\label{sec:specialization_t=1}
\def\coarea{{\rm coarea}}
\def\area{{\rm area}}
Let $\Pcal$ denote an up-right lattice path from $(0,0)$ to $(m,n)$.  We let $\area(\Pcal)$ denote the number of unit squares within the rectangle and below $\Pcal$.  Recall that $\PathC$ denotes the highest lattice path staying weakly below a monotone curve $C$.  Define $h_\Pcal:= h_{a_0} h_{a_1}  \cdots h_{a_n} \in \Laqt$, where the maximal horizontal portions of the path $\Pcal$ are of sizes $a_0,a_1,\ldots,a_n$.  Note that $h_\Pcal$ always has degree $m$.  The following result is a variant of \cite[Equation (148)]{BHMPS}.

\begin{proposition}\label{prop:t1}
For a monotone curve $C$, we have
\begin{equation}\label{eq:FC}
F_C|_{t = 1} = \sum_{\Pcal} q^{\area(\PathC) - \area(\Pcal)} h_{\Pcal}
\end{equation}
summed over lattice paths $\Pcal$ weakly below $C$, such that $\Pcal$ passes through all the lattice points that $C$ passes through.
\end{proposition}
\begin{proof}
When $C$ is an almost linear curve, this follows from the shuffle conjecture (i.e., the right-hand side of \eqref{eq:shuffle}).  More generally, when $C$ is a primitive monotone curve, the statement follows from \cite[Equation (148)]{BHMPS} and \cref{rem:BHMPS}, after translating their notation into ours.  Finally, it is straightforward to see that if \eqref{eq:FC} holds for two of the curves in \eqref{eq:intro:skein_F}, then the skein relation implies that it holds for the third.  This proves \eqref{eq:FC} for all monotone curves by induction on the number of lattice points that $C$ passes through.
\end{proof}

\begin{remark}
It follows from \cite[Proposition 4.9]{KiTs} that $\DTXlimEHA_C|_{t=1}$ acts on $\Laqt$ by multiplication by $F_C|_{t=1}$.  We thank Eugene Gorsky for pointing out this reference.
\end{remark}
%
%
%
%
%
%
%

\begin{corollary}
The coefficient of $s_m$ in $F_C|_{t=q=1}$ is equal to the number of lattice paths $\Pcal$ weakly below $C$, such that $\Pcal$ passes through all the lattice points that $C$ passes through.
\end{corollary}

In view of \cref{conj:intro:Schur_pos}, we have the following problem.
\begin{problem}
For a primitive $\Z$-convex curve $C$, find a statistic $\dinv$ on lattice paths so that the coefficient of $s_m$ in $F_C$ is given by
$$
\langle s_m, F_C \rangle = 
\sum_\Pcal q^{\area(\PathC) - \area(\Pcal)} t^{\dinv(\Pcal)}
$$
summed over lattice paths $\Pcal$ weakly below $C$.
\end{problem}

\subsection{The specialization \texorpdfstring{$t = \qi$}{t=1/q}.}

At $t = \qi$, we do not have a combinatorial formula for $F_C$.  However, we record the following consequence of \cref{fig:giant}, giving the precise meaning of $F_C$ in terms of the skein of the annulus.  

\begin{proposition}
Under the isomorphism $\phi_{\Ann}:\Skp(\Ann)\xrasim\Laq$ defined in~\eqref{eq:Skp_Ann_to_Laq_twisted}, the image of $\WC^\Ann$ is equal to $\left(\qqi-\qq\right)^{-\nseg(\Curve)} F_C|_{t = \qi}[\XSym(1-\qi)]$.
\end{proposition}

\section{$\Z$-convex \coaxial links are algebraic}\label{sec:multitorus}
\def\top{{\rm top}}
\def\lk{{\rm lk}}
\def\alg{{\rm alg}}
We shall use the language of iterated torus cables and splice diagrams, and refer the reader to \cite{EN} for a thorough treatment.  
The \emph{$(\rectX,\rectY )$-torus cable} of a knot $K$ (or of a component $K$ of a link $L$) in $\R^3$ is defined as follows.  The boundary of a tubular neighborhood $N(K)$ of $K$ in $\R^3$ is a torus $S = \partial N(K)$. The \emph{topological longitude} $L_\top$ and \emph{meridian} $M$ are the simple closed curves on $S$ characterized (up to isotopy) by the following relations (see \cite[p.21]{EN}):
\begin{alignat*}{4}
M &\sim 0 &\quad\text{ and }\quad&& L_\top &\sim K  &&\quad\text{ in } H_1(N(K));\\
\lk(M,K)&=1  &\quad\text{ and }\quad&& \lk(K,L_\top)&=0,
\end{alignat*}
where $\lk(A,A')$ denotes the linking number of the knots $A,A'$.  By definition, the \emph{$(\rectX,\rectY )$-cable} of $K$ is obtained by replacing $K$ by the simple closed curve $\gamma$ on $S$ satisfying $\gamma \sim \rectX L_\top + \rectY M$.

Let $(\rectX ,\rectY )$ be a pair of positive integers and with $d = \gcd(\rectX ,\rectY )$, we write $\rectX  = dp$ and $\rectY  = dq$.  Let $T'_{(\rectX ,\rectY )}$ be the $(d,dpq+1)$-cable of the torus knot $T_{p,q}$.  It has the \emph{splice diagram}:
\def\ooplus(#1,#2){
\node[scale=0.8,draw,circle,fill=white,inner sep=0.5pt](OP#1#2) at (#1,#2) {$+$};
}

\tikzset{TStyle/.style={inner sep=1pt}}

$$
\begin{tikzpicture}
\draw (0,0)--(2,0)--(4,0);
\draw[-latex] (4,0)--(6,0);
\draw (2,0)--(2,-1);
\draw (4,0)--(4,-1);
\draw[black,fill=black] (0,0) circle (.4ex);
\draw[black,fill=black] (2,-1) circle (.4ex);
\draw[black,fill=black] (4,-1) circle (.4ex);
\ooplus(2,0)
\ooplus(4,0)
\node[TStyle,anchor=south east] at (OP20.180) {$\scriptstyle q$};
\node[TStyle,anchor=south west] at (OP20.0) {$\scriptstyle 1$};
\node[TStyle,anchor=north west] at (OP20.-90) {$\scriptstyle p$};
\node[TStyle,anchor=south east] at (OP40.180) {$\scriptstyle dpq+1$};
\node[TStyle,anchor=south west] at (OP40.0) {$\scriptstyle 1$};
\node[TStyle,anchor=north west] at (OP40.-90) {$\scriptstyle d$};
\end{tikzpicture}
$$
The $\begin{tikzpicture}[baseline=(Z.base)]
\coordinate(Z) at (0,-0.1);
\ooplus(0,0)
\end{tikzpicture}$ vertices are called \emph{nodes}; the remaining vertices are called \emph{leaves}.  Roughly speaking, each node represents a 3-manifold, and each edge incident to a node represents a 2-torus boundary component.  An arrowhead leaf represents a component of a link, and the edge incident to it represents the 2-torus obtained as the boundary of its tubular neighborhood.  A bulletpoint leaf represents a 2-torus boundary component that has been filled in with a solid torus.

Recall the definition of an almost torus knot from \cref{rmk:intro_torus_link_knot}.
\begin{lemma}\label{lemma:cabling}
Let $C$ be an almost linear curve from $(0,0)$ to $(\rectX ,\rectY )$.  Then the $(\rectX ,\rectY )$-almost torus knot $L_C$ is isotopic to the knot $T'_{(\rectX ,\rectY )}$.
\end{lemma}
\begin{proof}
  Let $S = \partial N(T_{p,q})$ denote the boundary of a tubular neighborhood of the torus knot $T_{p,q}$.  Recall that we have defined curves $L_\top$ and $M$ on $S$. Consider the torus $\Tor=\R^2/\Z^2$ equipped with a standard embedding $\Tor\hookrightarrow\R^3$. Viewing $T_{p,q}\subset\Tor$ as the projection of a straight line segment $\Curve_{p,q}$ from $(0,0)$ to $(p,q)$, the projection of $\Curve_{p,q}+(-\eps,\eps)$ can also be viewed as a longitude $L_\alg$ on $S$, called the \emph{algebraic longitude}.  The linking number of $L_\alg$ and $S$ is given by $\lk(S, L_\alg) = pq$, and thus $L_\alg \sim L_\top + pq M$ (cf. \cite[Proof of Proposition 1A.1 on p.52]{EN}).  %
By definition, $T'_{\rectX ,\rectY }$ is (isotopic to) the simple closed curve on $S$ satisfying $T'_{\rectX ,\rectY } \sim dL_\top + (dpq+1)M$.

Recall from \cref{sec:SQA} that $\Curve$ is a $\sigma_1\sigma_2\cdots\sigma_{d-1}$-decoration of $\Curve_{p,q}$. The annular closure of the braid $\sigma_1\sigma_2\cdots\sigma_{d-1}$ is a $(d,1)$-torus knot, which we may view as embedded in $S$. 
 Thus, the projection of $C$ to $S$ is isotopic to a curve on $S$ with homology class $C \sim dL_\alg + M \sim dL_\top + (dpq+1)M$, as required.
\end{proof}

Using \cref{lemma:cabling}, we can construct a splice diagram of $\LC$ for an arbitrary \PAL curve $\Curve$. Let $(\rectX _1,\rectY _1),\ldots,(\rectX _r,\rectY _r)$ be a sequence of pairs of positive integers.  Let $\Tor_1,\Tor_2,\ldots,\Tor_r$ be tori in $\R^3$ with the same central circle $S^1$, and strictly decreasing radii.  Place the almost torus knot $T_{(\rectX _i,\rectY _i)}$ on $\Tor_i$.  The \coaxial link $T_{(\rectX _1,\rectY _1),\ldots,(\rectX _r,\rectY _r)}$ is the $r$-component link in $\R^3$ obtained as the union of all these almost torus knots.  In particular, if $\gcd(p,q)=1$, the $r$-component link $T_{(p,q),\ldots,(p,q)}$ is the torus link $T_{rp,rq}$.  Note that if $\rectX _i/\rectY _i = \rectX _{i+1}/\rectY _{i+1}$ then swapping the order of $(\rectX _i,\rectY _i)$ and $(\rectX _{i+1},\rectY _{i+1})$ results in an isotopic link.

\begin{figure}
\begin{tabular}{ccc}

\begin{tikzpicture}[baseline=(Z.base)]
\coordinate(Z) at (0,-5);
\draw (2,1)--(2,0)--(4,0);
\draw[-latex] (4,0)--(6,0);
\draw[-latex] (4,-2)--(6,-2);
\draw[-latex] (2,0)--(4,2);
\node[TStyle,anchor=west,inner sep=10pt] (A) at (4,2) {$\scriptstyle(d_1=1)$};
\draw (2,0)--(2,-1);
\draw (2,0)--(4,-2);
\draw (4,0)--(4,-1);
\draw (4,-2)--(4,-3);
\draw[black,fill=black] (4,-1) circle (.4ex);
\draw[black,fill=black] (4,-3) circle (.4ex);
\ooplus(2,0)
\ooplus(4,0)
\ooplus(4,-2)
\def\oneisp{5pt}
\node[TStyle,anchor=south east] at (OP20.90)  {$\scriptstyle q$};
\node[TStyle,anchor=north east] at (OP20.-90)  {$\scriptstyle p$};
\node[TStyle,anchor=-120,inner sep=\oneisp] at (OP20.60)  {$\scriptstyle 1$};
\node[TStyle,anchor=-150,inner sep=\oneisp] at (OP20.0)  {$\scriptstyle 1$};
\node[TStyle,anchor=165,inner sep=\oneisp] at (OP20.-45)  {$\scriptstyle 1$};

\node[TStyle,anchor=125,inner sep=5pt,scale=1.2] at (OP20.-90)  {\rotatebox{30}{$\scriptstyle \dots$}};

\node[TStyle,anchor=south east] at (OP40.180)  {$\scriptstyle d_2pq+1$};
\node[TStyle,anchor=south west] at (OP40.0)   {$\scriptstyle 1$};
\node[TStyle,anchor=north west] at (OP40.-90)  {$\scriptstyle d_2$};

\node[TStyle,anchor=-5,inner sep=5pt] at (OP4-2.135)  {$\scriptstyle d_3pq+1$};
\node[TStyle,anchor=south west] at (OP4-2.0)  {$\scriptstyle 1$};
\node[TStyle,anchor=north west] at (OP4-2.-90)  {$\scriptstyle d_3$};

\end{tikzpicture}
&
\hspace{0.2in}
&
\begin{tikzpicture}
\draw (2,1)--(2,0)--(4,0);
\draw[-latex] (4,0)--(6,0);
\draw[-latex] (4,-2)--(6,-2);
\draw[-latex] (2,0)--(4,2);
\draw[-latex] (4,-4)--(6,-4);
\draw[-latex] (2,-4)--(4,-6);
\draw[-latex] (2,-4)--(4,-7);
\draw (2,0)--(2,-4);
\draw (2,-4)--(2,-5);
\draw (2,0)--(4,-2);
\draw (4,0)--(4,-1);
\draw (4,-2)--(4,-3);
\draw (2,-4)--(4,-4);
\draw (4,-4)--(4,-5);
\draw[black,fill=black] (2,-5) circle (.4ex);
\draw[black,fill=black] (2,1) circle (.4ex);
\draw[black,fill=black] (4,-1) circle (.4ex);
\draw[black,fill=black] (4,-3) circle (.4ex);
\draw[black,fill=black] (4,-5) circle (.4ex);
\ooplus(2,0)
\ooplus(4,0)
\ooplus(4,-2)
\ooplus(2,-4)
\ooplus(4,-4)
\def\oneisp{5pt}
\node[TStyle,anchor=south east] at (OP20.90)  {$\scriptstyle 2$};
\node[TStyle,anchor=north east] at (OP20.-90)  {$\scriptstyle 3$};
\node[TStyle,anchor=-120,inner sep=\oneisp] at (OP20.60)  {$\scriptstyle 1$};
\node[TStyle,anchor=-150,inner sep=\oneisp] at (OP20.0)  {$\scriptstyle 1$};
\node[TStyle,anchor=165,inner sep=\oneisp] at (OP20.-45)  {$\scriptstyle 1$};

\node[TStyle,anchor=south east] at (OP40.180)  {$\scriptstyle 13$};
\node[TStyle,anchor=south west] at (OP40.0)   {$\scriptstyle 1$};
\node[TStyle,anchor=north west] at (OP40.-90)  {$\scriptstyle 2$};

\node[TStyle,anchor=south east] at (OP4-2.180)  {$\scriptstyle 13$};
\node[TStyle,anchor=south west] at (OP4-2.0)   {$\scriptstyle 1$};
\node[TStyle,anchor=north west] at (OP4-2.-90)  {$\scriptstyle 2$};

\node[TStyle,anchor=south east] at (OP2-4.90)  {$\scriptstyle 5$};
\node[TStyle,anchor=north east] at (OP2-4.-90)  {$\scriptstyle 2$};
\node[TStyle,anchor=-150,inner sep=\oneisp] at (OP2-4.0)  {$\scriptstyle 1$};
\node[TStyle,anchor=165,inner sep=\oneisp] at (OP2-4.-45)  {$\scriptstyle 1$};
\node[TStyle,anchor=115,inner sep=\oneisp] at (OP2-4.-80)  {$\scriptstyle 1$};

\node[TStyle,anchor=south east] at (OP4-4.180)  {$\scriptstyle 31$};
\node[TStyle,anchor=south west] at (OP4-4.0)   {$\scriptstyle 1$};
\node[TStyle,anchor=north west] at (OP4-4.-90)  {$\scriptstyle 3$};

\end{tikzpicture}
\\
(a) $(\rectX _{i+j},\rectY _{i+j})=(d_jp,d_jq)$, $j=1,2,\dots,s$ & &(b) $T_{(3,2),(6,4),(6,4),(6,15),(2,5),(2,5)}$
\end{tabular}
  \caption{\label{fig:splice}Constructing the splice diagram of a \coaxial link.}
\end{figure}

The splice diagram of a \coaxial link $T_{(\rectX _1,\rectY _1),\ldots,(\rectX _r,\rectY _r)}$ is given as follows.  Suppose that $\rectX _{i+1}/\rectY _{i+1} = \rectX _{i+2}/\rectY _{i+2} = \cdots = \rectX _{i+s}/\rectY _{i+s} = p/q$, where $i,s$ are chosen so that $s\geq1$ is maximal and $\gcd(p,q)=1$.  Define $d_1,\ldots,d_s$ so that $(\rectX _{i+j},\rectY _{i+j}) = (d_jp,d_jq)$ for $j=1,2,\dots,s$.
Then the splice diagram will contain $1+ \#\{j  \mid d_j > 1\} $ nodes, arranged as in \figref{fig:splice}(a). 
We place a node $N$
 in the left column incident to $s+2$ edges. For each $j=1,2,\dots,s$, if $d_j=1$ then we add an edge connecting $N$ to an arrowhead leaf in the right column. Otherwise, we draw an edge connecting $N$ to another node in the right column with three edges labeled $(d_jpq+1,1,d_j)$ connecting it to $N$, an arrowhead leaf, and a bulletpoint leaf, respectively. 

We also create two vertical edges labeled $q$ and $p$ emanating from $N$ up and down, respectively.
The resulting pieces are then connected to each other using these vertical edges, ordered from top to bottom according to decreasing radii of the tori.
 For example, the splice diagram of $T_{(3,2),(6,4),(6,4),(6,15),(2,5),(2,5)}$ is shown in \figref{fig:splice}(b), where the two slopes are $p/q=3/2$ and $p/q=2/5$.

\begin{proposition}
The \coaxial link $T_{(\rectX _1,\rectY _1),\ldots,(\rectX _r,\rectY _r)}$ is isotopic to the link $L_C$, where $C = C_{(\rectX _1,\rectY _1),\ldots,(\rectX _r,\rectY _r)} = [C_1C_2\cdots C_r]$ is a piecewise almost linear curve, and each $C_i$ is an $(\rectX _{i},\rectY _{i})$-almost linear curve.
\end{proposition}

The \coaxial link $T_{(\rectX _1,\rectY _1),\ldots,(\rectX _r,\rectY _r)}$ is called \emph{$\Z$-convex} if $C_{(\rectX _1,\rectY _1),\ldots,(\rectX _r,\rectY _r)}$ is $\Z$-convex.  This is the case if and only if $\frac{\rectX _1}{\rectY _1} \geq \frac{\rectX _2}{\rectY _2} \geq \cdots \geq \frac{\rectX _r}{\rectY _r}$; see~\eqref{eq:Z_convex_geq}.
\begin{proposition}
$\Z$-convex \coaxial links are algebraic.
\end{proposition}
\begin{proof}
According to \cite[Theorem 9.4]{EN}, for each edge connecting two nodes $A$ and $B$, we need to check the inequality $xy > \prod_i z_i$, where $x$ and $y$ are the two labels on the edge, and $z_i$-s are the other labels on edges incident to either $A$ or $B$.  This inequality follows from the $\Z$-convexity condition.
\end{proof}

\appendix
\section{Examples}\label{sec:examples}

\begin{example}
\Cref{tab:small_curves} lists $\FC$ for all curves $\Curve$ from $(0,0)$ to $(\rectXY)$ with $1\leq\rectXY\leq3$. 
\end{example}

\begin{figure}
\includegraphics[width=0.35\textwidth]{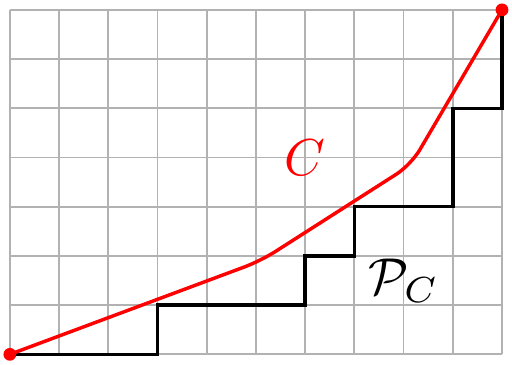}
  \caption{\label{fig:mellit-example}The smallest\protect\footnotemark\  primitive convex curve $\Curve$ (shown together with the lattice path $\PathC$) such that the EHA superpolynomial $\PEHAC$ is not $q,t$-unimodal; see \cref{ex:non_unimodal}.}
\end{figure}
\footnotetext{Here $(\rectXY)=(10,7)$, and one can check that for any primitive convex curve with $\rectX+\rectY<17$, the EHA superpolynomial is $q,t$-unimodal.}

\begin{example}\label{ex:non_convex_positroid}
Let $\rectX=\rectY=3$ and let $\Curve$ be the curve in \figref{fig:convex-examples}(c). It is the unique primitive curve from $(0,0)$ to $(3,3)$ which is not $\Z$-convex; cf. \cref{tab:small_curves}. Moreover, the associated symmetric function $\FC$ is not Schur positive. Even though $\Curve$ is not convex, it turns out that the associated link $\LC$ is still a positroid link. Namely, it is the smallest known positroid link (associated to a positroid variety in the Grassmannian $\Gr(7,14)$) whose odd KR homology does not vanish; see~\cite[Example~4.21]{GL_qtcat}. The coefficient of $s_3$ in $\FC$ is 
\begin{equation*}%
  q^{4} + q^{3} t + q^{2} t^{2} + q t^{3} + t^{4} + q^{2} t + q t^{2} - q t.
\end{equation*}
 By contrast, the top $a$-degree coefficient of $\PKRLC$, computed in~\cite[Example~4.21]{GL_qtcat}, is 
\begin{equation*}%
  q^{4} + q^{3} t + q^{2} t^{2} + q t^{3} + t^{4} + q^{2} t + q t^{2} - (q t)^{\frac32}.
\end{equation*}
\end{example}

\begin{example}\label{ex:non_unimodal}
The following example was discovered by A.~Mellit and communicated to us by M.~Mazin. Consider the primitive curve $\Curve$ in \cref{fig:mellit-example}, shown together with the highest lattice path 
\begin{equation*}%
  \PathC=RRRURRRURURRUURUU
\end{equation*}
 below $\Curve$; cf. \cref{dfn:intro:DCn}. The coefficient of $s_\rectX$ in $\FC$ is given by 
\begin{equation}\label{eq:Mellit_PEHAC}
  \begin{split}
&q^{16} + q^{15} t + q^{14} t^{2} + q^{13} t^{3} + q^{12} t^{4} + q^{11} t^{5} + q^{10} t^{6} + q^{9} t^{7} + q^{8} t^{8} + q^{7} t^{9} + q^{6} t^{10} + q^{5} t^{11} \\
&+ q^{4} t^{12} + q^{3} t^{13} + q^{2} t^{14} + q t^{15} + t^{16} + q^{14} t + q^{13} t^{2} + q^{12} t^{3} + q^{11} t^{4} + q^{10} t^{5} + q^{9} t^{6} + q^{8} t^{7} \\
&+ q^{7} t^{8} + q^{6} t^{9} + q^{5} t^{10} + q^{4} t^{11} + q^{3} t^{12} + q^{2} t^{13} + q t^{14} + q^{13} t + 2 \, q^{12} t^{2} + 2 \, q^{11} t^{3} + 2 \, q^{10} t^{4}\\
& + 2 \, q^{9} t^{5} + 2 \, q^{8} t^{6} + 2 \, q^{7} t^{7} + 2 \, q^{6} t^{8} + 2 \, q^{5} t^{9} + 2 \, q^{4} t^{10} + 2 \, q^{3} t^{11} + 2 \, q^{2} t^{12} + q t^{13} \\
&+ q^{12} t + 3 \, q^{11} t^{2} + 4 \, q^{10} t^{3} + 4 \, q^{9} t^{4} + 4 \, q^{8} t^{5} + 4 \, q^{7} t^{6} + 4 \, q^{6} t^{7} + 4 \, q^{5} t^{8} + 4 \, q^{4} t^{9} \\
&+ 4 \, q^{3} t^{10} + 3 \, q^{2} t^{11} + q t^{12} + 2 \, q^{10} t^{2}+ 5 \, q^{9} t^{3} + 6 \, q^{8} t^{4} + \textcolor{red}{6 \, q^{7} t^{5} + 5 \, q^{6} t^{6} + 6 \, q^{5} t^{7}} \\
&+ 6 \, q^{4} t^{8} + 5 \, q^{3} t^{9} + 2 \, q^{2} t^{10} + 2 \, q^{8} t^{3} + 5 \, q^{7} t^{4} + 8 \, q^{6} t^{5} + 8 \, q^{5} t^{6} + 5 \, q^{4} t^{7} + 2 \, q^{3} t^{8}.
  \end{split}
\end{equation}
The terms $6 \, q^{7} t^{5} + 5 \, q^{6} t^{6} + 6 \, q^{5} t^{7}$ highlighted in red violate the \emph{$q,t$-unimodality} property of~\cite[Theorem~1.5]{GL_qtcat}. Since $\Curve$ is convex, by the results of~\cite{GL_cat_combin,GL_plabic_links}, we see that $\LC$ is a positroid knot. The right-hand side  of~\eqref{eq:intro:KR} was shown to be $q,t$-unimodal in~\cite{GL_qtcat}. Therefore,~\eqref{eq:intro:KR} cannot hold for $\Curve$. Moreover, the top $a$-degree coefficient of the \FLY polynomial $\PHOMLC$ of the link $\LC$ is not \emph{parity-unimodal}:
\begin{align*}%
 \PHOMLC= \big(&q^{32} + q^{30} + q^{29} + 2 \, q^{28} + 2 \, q^{27} + 3 \, q^{26} + 4 \, q^{25} + 5 \, q^{24} + 5 \, q^{23} + 8 \, q^{22} + 7 \, q^{21} + 9 \, q^{20}  \\
& + 10 \, q^{19} + \textcolor{red}{9 \, q^{18}} + 13 \, q^{17} + \textcolor{red}{8 \, q^{16}} + 13 \, q^{15} + \textcolor{red}{9 \, q^{14}} + 10 \, q^{13} + 9 \, q^{12} + 7 \, q^{11} + 8 \, q^{10} \\
&+ 5 \, q^{9} + 5 \, q^{8} + 4 \, q^{7} + 3 \, q^{6} + 2 \, q^{5} + 2 \, q^{4} + q^{3} + q^{2} + 1\big) \frac1{a^{32}}+\cdots\;.
\end{align*}
(By~\eqref{prop:tqi}, this coefficient is obtained by plugging $t=\qi$ into~\eqref{eq:Mellit_PEHAC} and multiplying the result by $q^{16}/a^{32}$.) Combining this with the results of~\cite{GL_qtcat}, we see that the odd KR homology of $\LC$ cannot vanish. This disproves~\cite[Conjecture~7.1(ii,iv)]{GL_cat_combin}. 
\end{example}

\begin{table}
\def\FCscl{0.805}
  \setlength{\tabcolsep}{3pt}
\hspace{-0.1in}
\makebox[1.0\textwidth]{
\begin{tabular}{cc}
  \setlength{\tabcolsep}{2pt}
\scalebox{0.75}{
\begin{tabular}[t]{|c|c|c|c|}\hline
$\Curve$ & \rotatebox{90}{s-positive?} & \rotatebox{90}{$\Z$-convex?} & $\FC$ \\\hline
\scalebox{\FCscl}{\begin{tikzpicture}[scale=0.50, baseline=(Z.base)]
\coordinate(Z) at (0,0.00);
\draw[line width=0.50pt,draw=black!30] (0,0) grid (1.00,1.00);
\draw[line width=1.50pt, draw=red, rounded corners=5] (0.00,0.00) -- (1.00,1.00) ;
\begin{scope}[xshift=0.00cm,yshift=0.00cm]
\fill[red] (0,0) circle (0.13);
\end{scope}
\begin{scope}[xshift=1.00cm,yshift=1.00cm]
\fill[red] (0,0) circle (0.13);
\end{scope}
\end{tikzpicture}
 } & Y & Y & \scalebox{\FCscl}{$s_{1}$} \\\hline 
\scalebox{\FCscl}{\begin{tikzpicture}[scale=0.50, baseline=(Z.base)]
\coordinate(Z) at (0,1.00);
\draw[line width=0.50pt,draw=black!30] (0,0) grid (1.00,2.00);
\draw[line width=1.50pt, draw=red, rounded corners=5] (0.00,0.00) -- (1.00,2.00) ;
\begin{scope}[xshift=0.00cm,yshift=0.00cm]
\fill[red] (0,0) circle (0.13);
\end{scope}
\begin{scope}[xshift=1.00cm,yshift=2.00cm]
\fill[red] (0,0) circle (0.13);
\end{scope}
\end{tikzpicture}
 } & Y & Y & \scalebox{\FCscl}{$s_{1}$} \\\hline 
\scalebox{\FCscl}{\begin{tikzpicture}[scale=0.50, baseline=(Z.base)]
\coordinate(Z) at (0,0.00);
\draw[line width=0.50pt,draw=black!30] (0,0) grid (2.00,1.00);
\draw[line width=1.50pt, draw=red, rounded corners=5] (0.00,0.00) -- (2.00,1.00) ;
\begin{scope}[xshift=0.00cm,yshift=0.00cm]
\fill[red] (0,0) circle (0.13);
\end{scope}
\begin{scope}[xshift=2.00cm,yshift=1.00cm]
\fill[red] (0,0) circle (0.13);
\end{scope}
\end{tikzpicture}
 } & Y & Y & \scalebox{\FCscl}{$s_{2}$} \\\hline 
\scalebox{\FCscl}{\begin{tikzpicture}[scale=0.50, baseline=(Z.base)]
\coordinate(Z) at (0,1.00);
\draw[line width=0.50pt,draw=black!30] (0,0) grid (2.00,2.00);
\draw[line width=1.50pt, draw=red, rounded corners=5] (0.00,0.00) -- (1.35,0.65) -- (2.00,2.00) ;
\begin{scope}[xshift=0.00cm,yshift=0.00cm]
\fill[red] (0,0) circle (0.13);
\end{scope}
\begin{scope}[xshift=2.00cm,yshift=2.00cm]
\fill[red] (0,0) circle (0.13);
\end{scope}
\end{tikzpicture}
 } & Y & Y & \scalebox{\FCscl}{$s_{2}$} \\\hline 
\scalebox{\FCscl}{\begin{tikzpicture}[scale=0.50, baseline=(Z.base)]
\coordinate(Z) at (0,1.00);
\draw[line width=0.50pt,draw=black!30] (0,0) grid (2.00,2.00);
\draw[line width=1.50pt, draw=red, rounded corners=5] (0.00,0.00) -- (0.65,1.35) -- (2.00,2.00) ;
\begin{scope}[xshift=0.00cm,yshift=0.00cm]
\fill[red] (0,0) circle (0.13);
\end{scope}
\begin{scope}[xshift=2.00cm,yshift=2.00cm]
\fill[red] (0,0) circle (0.13);
\end{scope}
\end{tikzpicture}
 } & Y & Y & \scalebox{\FCscl}{$s_{11} + \left(q + t\right)s_{2}$} \\\hline 
\scalebox{\FCscl}{\begin{tikzpicture}[scale=0.50, baseline=(Z.base)]
\coordinate(Z) at (0,1.00);
\draw[line width=0.50pt,draw=black!30] (0,0) grid (2.00,2.00);
\draw[line width=1.50pt, draw=red, rounded corners=5] (0.00,0.00) -- (1.00,1.00) ;
\draw[line width=1.50pt, draw=red, rounded corners=5] (1.00,1.00) -- (2.00,2.00) ;
\begin{scope}[xshift=1.00cm,yshift=1.00cm]
\fill[red] (0,0) circle (0.13);
\fill[red] (0,0) -- (-45:0.13) arc (-45:135:0.13)--cycle;
\end{scope}
\begin{scope}[xshift=0.00cm,yshift=0.00cm]
\fill[red] (0,0) circle (0.13);
\end{scope}
\begin{scope}[xshift=2.00cm,yshift=2.00cm]
\fill[red] (0,0) circle (0.13);
\end{scope}
\end{tikzpicture}
 } & Y & Y & \scalebox{\FCscl}{$s_{11} + \left(q + t-qt\right)s_{2}$} \\\hline 
\scalebox{\FCscl}{\begin{tikzpicture}[scale=0.50, baseline=(Z.base)]
\coordinate(Z) at (0,1.00);
\draw[line width=0.50pt,draw=black!30] (0,0) grid (2.00,3.00);
\draw[line width=1.50pt, draw=red, rounded corners=5] (0.00,0.00) -- (1.35,0.65) -- (2.00,3.00) ;
\begin{scope}[xshift=0.00cm,yshift=0.00cm]
\fill[red] (0,0) circle (0.13);
\end{scope}
\begin{scope}[xshift=2.00cm,yshift=3.00cm]
\fill[red] (0,0) circle (0.13);
\end{scope}
\end{tikzpicture}
 } & Y & Y & \scalebox{\FCscl}{$s_{2}$} \\\hline 
\scalebox{\FCscl}{\begin{tikzpicture}[scale=0.50, baseline=(Z.base)]
\coordinate(Z) at (0,1.00);
\draw[line width=0.50pt,draw=black!30] (0,0) grid (2.00,3.00);
\draw[line width=1.50pt, draw=red, rounded corners=5] (0.00,0.00) -- (2.00,3.00) ;
\begin{scope}[xshift=0.00cm,yshift=0.00cm]
\fill[red] (0,0) circle (0.13);
\end{scope}
\begin{scope}[xshift=2.00cm,yshift=3.00cm]
\fill[red] (0,0) circle (0.13);
\end{scope}
\end{tikzpicture}
 } & Y & Y & \scalebox{\FCscl}{$s_{11} + \left(q + t\right)s_{2}$} \\\hline 
\scalebox{\FCscl}{\begin{tikzpicture}[scale=0.50, baseline=(Z.base)]
\coordinate(Z) at (0,1.00);
\draw[line width=0.50pt,draw=black!30] (0,0) grid (2.00,3.00);
\draw[line width=1.50pt, draw=red, rounded corners=5] (0.00,0.00) -- (0.65,2.35) -- (2.00,3.00) ;
\begin{scope}[xshift=0.00cm,yshift=0.00cm]
\fill[red] (0,0) circle (0.13);
\end{scope}
\begin{scope}[xshift=2.00cm,yshift=3.00cm]
\fill[red] (0,0) circle (0.13);
\end{scope}
\end{tikzpicture}
 } & Y & \textcolor{red}{N} & \scalebox{\FCscl}{$\left(q + t\right)s_{11} + \left(q^{2} + q t + t^{2}\right)s_{2}$} \\\hline 
\scalebox{\FCscl}{\begin{tikzpicture}[scale=0.50, baseline=(Z.base)]
\coordinate(Z) at (0,1.00);
\draw[line width=0.50pt,draw=black!30] (0,0) grid (2.00,3.00);
\draw[line width=1.50pt, draw=red, rounded corners=5] (0.00,0.00) -- (1.00,2.00) ;
\draw[line width=1.50pt, draw=red, rounded corners=5] (1.00,2.00) -- (2.00,3.00) ;
\begin{scope}[xshift=1.00cm,yshift=2.00cm]
\fill[red] (0,0) circle (0.13);
\fill[red] (0,0) -- (-45:0.13) arc (-45:135:0.13)--cycle;
\end{scope}
\begin{scope}[xshift=0.00cm,yshift=0.00cm]
\fill[red] (0,0) circle (0.13);
\end{scope}
\begin{scope}[xshift=2.00cm,yshift=3.00cm]
\fill[red] (0,0) circle (0.13);
\end{scope}
\end{tikzpicture}
 } & Y & \textcolor{red}{N} & \scalebox{\FCscl}{\begin{tabular}{c} $\left(q + t-qt\right)s_{11}$ \\ $+\left( q^{2} + q t + t^{2}-q^{2} t - q t^{2}\right)s_{2}$ \end{tabular}} \\\hline 
\scalebox{\FCscl}{\begin{tikzpicture}[scale=0.50, baseline=(Z.base)]
\coordinate(Z) at (0,1.00);
\draw[line width=0.50pt,draw=black!30] (0,0) grid (2.00,3.00);
\draw[line width=1.50pt, draw=red, rounded corners=5] (0.00,0.00) -- (1.00,1.00) ;
\draw[line width=1.50pt, draw=red, rounded corners=5] (1.00,1.00) -- (2.00,3.00) ;
\begin{scope}[xshift=1.00cm,yshift=1.00cm]
\fill[red] (0,0) circle (0.13);
\fill[red] (0,0) -- (-45:0.13) arc (-45:135:0.13)--cycle;
\end{scope}
\begin{scope}[xshift=0.00cm,yshift=0.00cm]
\fill[red] (0,0) circle (0.13);
\end{scope}
\begin{scope}[xshift=2.00cm,yshift=3.00cm]
\fill[red] (0,0) circle (0.13);
\end{scope}
\end{tikzpicture}
 } & Y & Y & \scalebox{\FCscl}{$s_{11} + \left(q + t-qt\right)s_{2}$} \\\hline 
\scalebox{\FCscl}{\begin{tikzpicture}[scale=0.50, baseline=(Z.base)]
\coordinate(Z) at (0,1.00);
\draw[line width=0.50pt,draw=black!30] (0,0) grid (3.00,2.00);
\draw[line width=1.50pt, draw=red, rounded corners=5] (0.00,0.00) -- (2.35,0.65) -- (3.00,2.00) ;
\begin{scope}[xshift=0.00cm,yshift=0.00cm]
\fill[red] (0,0) circle (0.13);
\end{scope}
\begin{scope}[xshift=3.00cm,yshift=2.00cm]
\fill[red] (0,0) circle (0.13);
\end{scope}
\end{tikzpicture}
 } & Y & Y & \scalebox{\FCscl}{$s_{3}$} \\\hline 
\scalebox{\FCscl}{\begin{tikzpicture}[scale=0.50, baseline=(Z.base)]
\coordinate(Z) at (0,1.00);
\draw[line width=0.50pt,draw=black!30] (0,0) grid (3.00,2.00);
\draw[line width=1.50pt, draw=red, rounded corners=5] (0.00,0.00) -- (3.00,2.00) ;
\begin{scope}[xshift=0.00cm,yshift=0.00cm]
\fill[red] (0,0) circle (0.13);
\end{scope}
\begin{scope}[xshift=3.00cm,yshift=2.00cm]
\fill[red] (0,0) circle (0.13);
\end{scope}
\end{tikzpicture}
 } & Y & Y & \scalebox{\FCscl}{$s_{21} + \left(q + t\right)s_{3}$} \\\hline 
\scalebox{\FCscl}{\begin{tikzpicture}[scale=0.50, baseline=(Z.base)]
\coordinate(Z) at (0,1.00);
\draw[line width=0.50pt,draw=black!30] (0,0) grid (3.00,2.00);
\draw[line width=1.50pt, draw=red, rounded corners=5] (0.00,0.00) -- (0.65,1.35) -- (3.00,2.00) ;
\begin{scope}[xshift=0.00cm,yshift=0.00cm]
\fill[red] (0,0) circle (0.13);
\end{scope}
\begin{scope}[xshift=3.00cm,yshift=2.00cm]
\fill[red] (0,0) circle (0.13);
\end{scope}
\end{tikzpicture}
 } & Y & \textcolor{red}{N} & \scalebox{\FCscl}{$\left(q + t\right)s_{21} + \left(q^{2} + q t + t^{2}\right)s_{3}$} \\\hline 
\scalebox{\FCscl}{\begin{tikzpicture}[scale=0.50, baseline=(Z.base)]
\coordinate(Z) at (0,1.00);
\draw[line width=0.50pt,draw=black!30] (0,0) grid (3.00,2.00);
\draw[line width=1.50pt, draw=red, rounded corners=5] (0.00,0.00) -- (2.00,1.00) ;
\draw[line width=1.50pt, draw=red, rounded corners=5] (2.00,1.00) -- (3.00,2.00) ;
\begin{scope}[xshift=2.00cm,yshift=1.00cm]
\fill[red] (0,0) circle (0.13);
\fill[red] (0,0) -- (-45:0.13) arc (-45:135:0.13)--cycle;
\end{scope}
\begin{scope}[xshift=0.00cm,yshift=0.00cm]
\fill[red] (0,0) circle (0.13);
\end{scope}
\begin{scope}[xshift=3.00cm,yshift=2.00cm]
\fill[red] (0,0) circle (0.13);
\end{scope}
\end{tikzpicture}
 } & Y & Y & \scalebox{\FCscl}{$s_{21} + \left(q + t-qt\right)s_{3}$} \\\hline 
\scalebox{\FCscl}{\begin{tikzpicture}[scale=0.50, baseline=(Z.base)]
\coordinate(Z) at (0,1.00);
\draw[line width=0.50pt,draw=black!30] (0,0) grid (3.00,2.00);
\draw[line width=1.50pt, draw=red, rounded corners=5] (0.00,0.00) -- (1.00,1.00) ;
\draw[line width=1.50pt, draw=red, rounded corners=5] (1.00,1.00) -- (3.00,2.00) ;
\begin{scope}[xshift=1.00cm,yshift=1.00cm]
\fill[red] (0,0) circle (0.13);
\fill[red] (0,0) -- (-45:0.13) arc (-45:135:0.13)--cycle;
\end{scope}
\begin{scope}[xshift=0.00cm,yshift=0.00cm]
\fill[red] (0,0) circle (0.13);
\end{scope}
\begin{scope}[xshift=3.00cm,yshift=2.00cm]
\fill[red] (0,0) circle (0.13);
\end{scope}
\end{tikzpicture}
 } & Y & \textcolor{red}{N} & \scalebox{\FCscl}{\begin{tabular}{c} $\left(q + t-qt\right)s_{21}$ \\ 
$+\left( q^{2} + q t + t^{2}-q^{2} t - q t^{2}\right)s_{3}$ \end{tabular}} \\\hline 
\end{tabular}
}

 & \hspace{-0.2in}
  \setlength{\tabcolsep}{2pt}
\scalebox{0.75}{
\begin{tabular}[t]{|c|c|c|c|}\hline
$\Curve$ & \rotatebox{90}{s-positive?} & \rotatebox{90}{$\Z$-convex?} & $\FC$ \\\hline
\scalebox{\FCscl}{\begin{tikzpicture}[scale=0.50, baseline=(Z.base)]
\coordinate(Z) at (0,1.00);
\draw[line width=0.50pt,draw=black!30] (0,0) grid (3.00,3.00);
\draw[line width=1.50pt, draw=red, rounded corners=5] (0.00,0.00) -- (2.35,0.65) -- (3.00,3.00) ;
\begin{scope}[xshift=0.00cm,yshift=0.00cm]
\fill[red] (0,0) circle (0.13);
\end{scope}
\begin{scope}[xshift=3.00cm,yshift=3.00cm]
\fill[red] (0,0) circle (0.13);
\end{scope}
\end{tikzpicture}
 } & Y & Y & \scalebox{\FCscl}{$s_{3}$} \\\hline 
\scalebox{\FCscl}{\begin{tikzpicture}[scale=0.50, baseline=(Z.base)]
\coordinate(Z) at (0,1.00);
\draw[line width=0.50pt,draw=black!30] (0,0) grid (3.00,3.00);
\draw[line width=1.50pt, draw=red, rounded corners=5] (0.00,0.00) -- (1.35,0.65) -- (3.00,3.00) ;
\begin{scope}[xshift=0.00cm,yshift=0.00cm]
\fill[red] (0,0) circle (0.13);
\end{scope}
\begin{scope}[xshift=3.00cm,yshift=3.00cm]
\fill[red] (0,0) circle (0.13);
\end{scope}
\end{tikzpicture}
 } & Y & Y & \scalebox{\FCscl}{$s_{21} + \left(q + t\right)s_{3}$} \\\hline 
\scalebox{\FCscl}{\begin{tikzpicture}[scale=0.50, baseline=(Z.base)]
\coordinate(Z) at (0,1.00);
\draw[line width=0.50pt,draw=black!30] (0,0) grid (3.00,3.00);
\draw[line width=1.50pt, draw=red, rounded corners=5] (0.00,0.00) -- (1.35,0.65) -- (1.65,2.35) -- (3.00,3.00) ;
\begin{scope}[xshift=0.00cm,yshift=0.00cm]
\fill[red] (0,0) circle (0.13);
\end{scope}
\begin{scope}[xshift=3.00cm,yshift=3.00cm]
\fill[red] (0,0) circle (0.13);
\end{scope}
\end{tikzpicture}
 } & Y & Y & \scalebox{\FCscl}{$\left(q + t\right)s_{21} + \left(q^{2} + q t + t^{2}\right)s_{3}$} \\\hline 
\scalebox{\FCscl}{\begin{tikzpicture}[scale=0.50, baseline=(Z.base)]
\coordinate(Z) at (0,1.00);
\draw[line width=0.50pt,draw=black!30] (0,0) grid (3.00,3.00);
\draw[line width=1.50pt, draw=red, rounded corners=5] (0.00,0.00) -- (0.65,1.35) -- (2.35,1.65) -- (3.00,3.00) ;
\begin{scope}[xshift=0.00cm,yshift=0.00cm]
\fill[red] (0,0) circle (0.13);
\end{scope}
\begin{scope}[xshift=3.00cm,yshift=3.00cm]
\fill[red] (0,0) circle (0.13);
\end{scope}
\end{tikzpicture}
 } & Y & Y & \scalebox{\FCscl}{$\left(q + t\right)s_{21} + \left(q^{2} + q t + t^{2}\right)s_{3}$} \\\hline 
\scalebox{\FCscl}{\begin{tikzpicture}[scale=0.50, baseline=(Z.base)]
\coordinate(Z) at (0,1.00);
\draw[line width=0.50pt,draw=black!30] (0,0) grid (3.00,3.00);
\draw[line width=1.50pt, draw=red] (0.00,0.00) to [bend left=15] (3.00,3.00) ;
\begin{scope}[xshift=0.00cm,yshift=0.00cm]
\fill[red] (0,0) circle (0.13);
\end{scope}
\begin{scope}[xshift=3.00cm,yshift=3.00cm]
\fill[red] (0,0) circle (0.13);
\end{scope}
\end{tikzpicture}
 } & Y & Y & \scalebox{\FCscl}{\begin{tabular}{c} $s_{111} + \left(q^{2} + q t + t^{2} + q + t\right)s_{21}$ \\ $+\left(q^{3} + q^{2} t + q t^{2} + t^{3} + q t\right)s_{3}$ \end{tabular}} \\\hline 
\scalebox{\FCscl}{\begin{tikzpicture}[scale=0.50, baseline=(Z.base)]
\coordinate(Z) at (0,1.00);
\draw[line width=0.50pt,draw=black!30] (0,0) grid (3.00,3.00);
\draw[line width=1.50pt, draw=red, rounded corners=5] (0.00,0.00) -- (0.65,2.35) -- (3.00,3.00) ;
\begin{scope}[xshift=0.00cm,yshift=0.00cm]
\fill[red] (0,0) circle (0.13);
\end{scope}
\begin{scope}[xshift=3.00cm,yshift=3.00cm]
\fill[red] (0,0) circle (0.13);
\end{scope}
\end{tikzpicture}
 } & \textcolor{red}{N} & \textcolor{red}{N} & \scalebox{\FCscl}{\begin{tabular}{c} $\left(q + t - 1\right)s_{111} + \left(q^{3} + q^{2} t + q t^{2} + t^{3} + q^{2} + 2 q t + t^{2} - q - t\right)s_{21}$ \\ $+\left(q^{4} + q^{3} t + q^{2} t^{2} + q t^{3} + t^{4} + q^{2} t + q t^{2} - q t\right)s_{3}$ \end{tabular}} \\\hline 
\scalebox{\FCscl}{\begin{tikzpicture}[scale=0.50, baseline=(Z.base)]
\coordinate(Z) at (0,1.00);
\draw[line width=0.50pt,draw=black!30] (0,0) grid (3.00,3.00);
\draw[line width=1.50pt, draw=red, rounded corners=5] (0.00,0.00) -- (2.00,1.00) ;
\draw[line width=1.50pt, draw=red, rounded corners=5] (2.00,1.00) -- (3.00,3.00) ;
\begin{scope}[xshift=2.00cm,yshift=1.00cm]
\fill[red] (0,0) circle (0.13);
\fill[red] (0,0) -- (-45:0.13) arc (-45:135:0.13)--cycle;
\end{scope}
\begin{scope}[xshift=0.00cm,yshift=0.00cm]
\fill[red] (0,0) circle (0.13);
\end{scope}
\begin{scope}[xshift=3.00cm,yshift=3.00cm]
\fill[red] (0,0) circle (0.13);
\end{scope}
\end{tikzpicture}
 } & Y & Y & \scalebox{\FCscl}{$s_{21} + \left(q + t-qt\right)s_{3}$} \\\hline 
\scalebox{\FCscl}{\begin{tikzpicture}[scale=0.50, baseline=(Z.base)]
\coordinate(Z) at (0,1.00);
\draw[line width=0.50pt,draw=black!30] (0,0) grid (3.00,3.00);
\draw[line width=1.50pt, draw=red, rounded corners=5] (0.00,0.00) -- (1.00,2.00) ;
\draw[line width=1.50pt, draw=red, rounded corners=5] (1.00,2.00) -- (3.00,3.00) ;
\begin{scope}[xshift=1.00cm,yshift=2.00cm]
\fill[red] (0,0) circle (0.13);
\fill[red] (0,0) -- (-45:0.13) arc (-45:135:0.13)--cycle;
\end{scope}
\begin{scope}[xshift=0.00cm,yshift=0.00cm]
\fill[red] (0,0) circle (0.13);
\end{scope}
\begin{scope}[xshift=3.00cm,yshift=3.00cm]
\fill[red] (0,0) circle (0.13);
\end{scope}
\end{tikzpicture}
 } & \textcolor{red}{N} & \textcolor{red}{N} & \scalebox{\FCscl}{\begin{tabular}{c} $\left(q + t - 1-qt \right)s_{111} + \left( q^{3} + t^{3} + q^{2} + 2 q t + t^{2} - q - t -q^{3} t - q^{2} t^{2} - q t^{3}\right)s_{21}$ \\ $+\left( q^{4} + q^{3} t + q t^{3} + t^{4} + q^{2} t + q t^{2} - q t -q^{4} t - q^{3} t^{2} - q^{2} t^{3} - q t^{4}\right)s_{3}$ \end{tabular}} \\\hline 
\scalebox{\FCscl}{\begin{tikzpicture}[scale=0.50, baseline=(Z.base)]
\coordinate(Z) at (0,1.00);
\draw[line width=0.50pt,draw=black!30] (0,0) grid (3.00,3.00);
\draw[line width=1.50pt, draw=red, rounded corners=5] (0.00,0.00) -- (1.00,1.00) ;
\draw[line width=1.50pt, draw=red, rounded corners=5] (1.00,1.00) -- (2.35,1.65) -- (3.00,3.00) ;
\begin{scope}[xshift=1.00cm,yshift=1.00cm]
\fill[red] (0,0) circle (0.13);
\fill[red] (0,0) -- (-45:0.13) arc (-45:135:0.13)--cycle;
\end{scope}
\begin{scope}[xshift=0.00cm,yshift=0.00cm]
\fill[red] (0,0) circle (0.13);
\end{scope}
\begin{scope}[xshift=3.00cm,yshift=3.00cm]
\fill[red] (0,0) circle (0.13);
\end{scope}
\end{tikzpicture}
 } & Y & Y & \scalebox{\FCscl}{\begin{tabular}{c} $\left(q + t-qt\right)s_{21}$ \\ $+\left( q^{2} + q t + t^{2} - q^{2} t - q t^{2} \right)s_{3}$ \end{tabular}} \\\hline 
\scalebox{\FCscl}{\begin{tikzpicture}[scale=0.50, baseline=(Z.base)]
\coordinate(Z) at (0,1.00);
\draw[line width=0.50pt,draw=black!30] (0,0) grid (3.00,3.00);
\draw[line width=1.50pt, draw=red, rounded corners=5] (0.00,0.00) -- (1.00,1.00) ;
\draw[line width=1.50pt, draw=red, rounded corners=5] (1.00,1.00) -- (1.65,2.35) -- (3.00,3.00) ;
\begin{scope}[xshift=1.00cm,yshift=1.00cm]
\fill[red] (0,0) circle (0.13);
\fill[red] (0,0) -- (-45:0.13) arc (-45:135:0.13)--cycle;
\end{scope}
\begin{scope}[xshift=0.00cm,yshift=0.00cm]
\fill[red] (0,0) circle (0.13);
\end{scope}
\begin{scope}[xshift=3.00cm,yshift=3.00cm]
\fill[red] (0,0) circle (0.13);
\end{scope}
\end{tikzpicture}
 } & Y & Y & \scalebox{\FCscl}{\begin{tabular}{c} $s_{111} + \left( q^{2} + q t + t^{2} + q + t -q^{2} t - q t^{2} \right)s_{21}$ \\ $+\left( q^{3} + q^{2} t + q t^{2} + t^{3} + q t -q^{3} t - q^{2} t^{2} - q t^{3} \right)s_{3}$ \end{tabular}} \\\hline 
\scalebox{\FCscl}{\begin{tikzpicture}[scale=0.50, baseline=(Z.base)]
\coordinate(Z) at (0,1.00);
\draw[line width=0.50pt,draw=black!30] (0,0) grid (3.00,3.00);
\draw[line width=1.50pt, draw=red, rounded corners=5] (0.00,0.00) -- (1.00,1.00) ;
\draw[line width=1.50pt, draw=red, rounded corners=5] (1.00,1.00) -- (2.00,2.00) ;
\draw[line width=1.50pt, draw=red, rounded corners=5] (2.00,2.00) -- (3.00,3.00) ;
\begin{scope}[xshift=1.00cm,yshift=1.00cm]
\fill[red] (0,0) circle (0.13);
\fill[red] (0,0) -- (-45:0.13) arc (-45:135:0.13)--cycle;
\end{scope}
\begin{scope}[xshift=2.00cm,yshift=2.00cm]
\fill[red] (0,0) circle (0.13);
\fill[red] (0,0) -- (-45:0.13) arc (-45:135:0.13)--cycle;
\end{scope}
\begin{scope}[xshift=0.00cm,yshift=0.00cm]
\fill[red] (0,0) circle (0.13);
\end{scope}
\begin{scope}[xshift=3.00cm,yshift=3.00cm]
\fill[red] (0,0) circle (0.13);
\end{scope}
\end{tikzpicture}
 } & Y & Y & \scalebox{\FCscl}{\begin{tabular}{c} $s_{111} + \left(q^{2} t^{2} + q^{2} + q t + t^{2} + q + t  - 2 q^{2} t - 2 q t^{2}\right)s_{21}$ \\ $+\left(q^{3} t^{2} + q^{2} t^{3}  + q^{3} + q^{2} t + q t^{2} + t^{3} + q t - 2 q^{3} t - 2 q^{2} t^{2} - 2 q t^{3}\right)s_{3}$ \end{tabular}} \\\hline 
\scalebox{\FCscl}{\begin{tikzpicture}[scale=0.50, baseline=(Z.base)]
\coordinate(Z) at (0,1.00);
\draw[line width=0.50pt,draw=black!30] (0,0) grid (3.00,3.00);
\draw[line width=1.50pt, draw=red, rounded corners=5] (0.00,0.00) -- (1.35,0.65) -- (2.00,2.00) ;
\draw[line width=1.50pt, draw=red, rounded corners=5] (2.00,2.00) -- (3.00,3.00) ;
\begin{scope}[xshift=2.00cm,yshift=2.00cm]
\fill[red] (0,0) circle (0.13);
\fill[red] (0,0) -- (-45:0.13) arc (-45:135:0.13)--cycle;
\end{scope}
\begin{scope}[xshift=0.00cm,yshift=0.00cm]
\fill[red] (0,0) circle (0.13);
\end{scope}
\begin{scope}[xshift=3.00cm,yshift=3.00cm]
\fill[red] (0,0) circle (0.13);
\end{scope}
\end{tikzpicture}
 } & Y & Y & \scalebox{\FCscl}{\begin{tabular}{c} $\left(q + t - qt\right)s_{21}$ \\ $+\left(q^{2} + q t + t^{2} -q^{2} t - q t^{2} \right)s_{3}$ \end{tabular}} \\\hline 
\scalebox{\FCscl}{\begin{tikzpicture}[scale=0.50, baseline=(Z.base)]
\coordinate(Z) at (0,1.00);
\draw[line width=0.50pt,draw=black!30] (0,0) grid (3.00,3.00);
\draw[line width=1.50pt, draw=red, rounded corners=5] (0.00,0.00) -- (0.65,1.35) -- (2.00,2.00) ;
\draw[line width=1.50pt, draw=red, rounded corners=5] (2.00,2.00) -- (3.00,3.00) ;
\begin{scope}[xshift=2.00cm,yshift=2.00cm]
\fill[red] (0,0) circle (0.13);
\fill[red] (0,0) -- (-45:0.13) arc (-45:135:0.13)--cycle;
\end{scope}
\begin{scope}[xshift=0.00cm,yshift=0.00cm]
\fill[red] (0,0) circle (0.13);
\end{scope}
\begin{scope}[xshift=3.00cm,yshift=3.00cm]
\fill[red] (0,0) circle (0.13);
\end{scope}
\end{tikzpicture}
 } & Y & Y & \scalebox{\FCscl}{\begin{tabular}{c} $s_{111} + \left( q^{2} + q t + t^{2} + q + t -q^{2} t - q t^{2} \right)s_{21}$ \\ $+\left( q^{3} + q^{2} t + q t^{2} + t^{3} + q t -q^{3} t - q^{2} t^{2} - q t^{3}\right)s_{3}$ \end{tabular}} \\\hline 
\end{tabular}
}
\end{tabular}
}
  \caption{\label{tab:small_curves}The symmetric functions $\FC$ for all monotone curves $\Curve$ with $1\leq \rectXY\leq 3$. Here, Y/N indicates whether each curve is \emph{s-positive} (i.e., has $(q,t)$-Schur positive series~\eqref{eq:intro_Schur_pos}) and \emph{$\Z$-convex} (\cref{dfn:intro_Z_convex}).}
\end{table}

\bibliographystyle{alpha}
\bibliography{EHA}
\end{document}